\documentclass[leqno]{amsart}
\usepackage{amssymb}
\usepackage{mathrsfs}
\usepackage{amsmath, amsfonts, vmargin, enumerate}
\usepackage{graphics}
\usepackage{color}
\usepackage{verbatim}

\usepackage{amsthm}
\usepackage{latexsym, bm}
\usepackage{euscript}
\usepackage{dsfont}

\input xypic
\xyoption {all}

 \makeatletter

\newtheorem{thm}{Theorem}[section]
\newtheorem{prop}[thm]{Proposition}
\newtheorem{cor}[thm]{Corollary}
\newtheorem{lem}[thm]{Lemma}
\newtheorem{defn}[thm]{Definition}
\newtheorem{rem}[thm]{Remark}
\newtheorem{example}[thm]{Example}
\newtheorem{pb}[thm]{Problem}
\newtheorem{conj}[thm]{Conjecture}

\newenvironment{rmk}{\begin{rem}\rm}{\end{rem}}

\numberwithin{equation}{section}

\newcommand{\supp}{{\rm supp\,}}
\newcommand{\bmo}{{\rm bmo}}

\newcommand{\h}{{\rm h}}
\newcommand{\M}{\mathcal M}
\newcommand{\N}{\mathcal N}
\newcommand{\F}{\mathcal F}

\newcommand{\R}{\mathbb{R}^d}
\newcommand{\0}{\varphi_0}
\newcommand{\e}{\varepsilon}

\newcommand{\fk}{\mathsf{k} }

\newcommand{\un}{{\mathds {1}}}

\begin{document}

\title{Operator-valued Triebel-Lizorkin spaces}

\author{Runlian  XIA}

\address{Laboratoire de Math{\'e}matiques, Universit{\'e} de Franche-Comt{\'e},
25030 Besan\c{c}on Cedex, France, and Instituto de Ciencias Matem{\'a}ticas, 28049 Madrid, Spain}
\email{runlian91@gmail.com}

\thanks{{\it 2000 Mathematics Subject Classification:} Primary: 46L52, 42B30. Secondary: 46L07, 47L65}

\thanks{{\it Key words:} Noncommutative $L_p$-spaces, operator-valued Triebel-Lizorkin spaces, operator-valued Hardy spaces, Fourier multipliers, interpolation, characterizations, atomic decomposition}

\author{Xiao XIONG}

\address{Department of Mathematics and Statistics, University of Saskatchewan, Saskatoon, Saskatchewan, S7N 5E6, Canada}
\email{ufcxx56@gmail.com}

\maketitle

\markboth{R. Xia and X. Xiong}%
{Inhomogeneous Triebel-Lizorkin spaces}

\begin{abstract}
This paper is devoted to the study of operator-valued Triebel-Lizorkin spaces. We develop some Fourier multiplier theorems for square functions as our main tool, and then study the operator-valued Triebel-Lizorkin spaces on $\R$. As in the classical case, we connect these spaces with operator-valued local Hardy spaces via Bessel potentials. We show the lifting theorem, and get interpolation results for these spaces. We obtain Littlewood-Paley type, as well as the Lusin type square function characterizations in the general way. Finally, we establish smooth atomic decompositions for the operator-valued Triebel-Lizorkin spaces. These atomic decompositions play a key role in our recent study of mapping properties of pseudo-differential operators with operator-valued symbols.
\end{abstract}

\tableofcontents

\setcounter{section}{-1}

 \section{Introduction and preliminaries}
 Let $\varphi$ be a Schwartz function on $\R$ such that $\supp \varphi \subset \lbrace \xi:\frac{1}{2}\leq |  \xi |  \leq 2\rbrace$, $\varphi >0$ on $\lbrace \xi:\frac{1}{2}< |  \xi |  < 2\rbrace$, and $\sum_{k\in \mathbb{Z}}\varphi (2^{-k}\xi)=1$ for all $\xi \neq 0$.
For each $k \in \mathbb{N}$, let $\varphi_k$ be the function whose Fourier transform is equal to $\varphi(2^{-k}\cdot)$, and let $\varphi_0$ be the function whose Fourier transform is equal to $1-\sum_{k>0}\varphi (2^{-k}\cdot)$. Then $\{\varphi_k\}_{k\geq 0}$ gives a Littlewood-Paley decomposition on $\mathbb{R}^d$. The classical (inhomogeneous) Triebel-Lizorkin spaces $F_{p,q}^\alpha(\R)$ for $0<p<\infty$, $0<q\leq \infty$ and $\alpha\in \mathbb{R}$ are defined as 
$$F_{p,q}^\alpha(\R)  =\{f\in \mathcal{S}'(\R): \|f\|_{F_{p,q}^\alpha}   <\infty \}$$
 with the (quasi-)norm 
 $$\|f\|_{F_{p,q}^\alpha}    =       \big\| (\sum_{j\geq 0 }    2^{qj\alpha} |\varphi_j  *f |^q )^{\frac 1 q} 
 \big\|_p  \, .$$
  We refer the reader to Triebel's books \cite{Tri} and \cite{Tri1992} for more concrete definition and properties of Triebel-Lizorkin spaces on $\R$.
 This kind of function spaces is closely related to other function spaces, such as Sobolev and Besov spaces. In particular, Triebel-Lizorkin spaces can be viewed as generalizations of Hardy spaces, since the Bessel potential $J^\alpha$ is known to be an isomorphism between $F_{p,2}^\alpha(\R)$ and $\h_p(\R)$ (local Hardy spaces introduced in \cite{Goldberg1979}). All these spaces are basic for many branches of mathematics such as harmonic analysis, PDE, functional analysis and approximation theory.
 
 \medskip

This paper is devoted to the study of operator-valued Triebel-Lizorkin spaces. As in the classical case, it can be viewed as an extension of our recent work \cite{XX18} on operator-valued local Hardy spaces. On the other hand, the Triebel-Lizorkin spaces studied here are Euclidean counterparts of those on usual and quantum tori studied in \cite{XXY17}. Our main motivation is to build a kind of function spaces where we can carry out the investigation of pseudo-differential operators with operator-valued symbols.

\medskip

Due to  noncommutativity, there are several obstacles on our route, which do not appear in the classical case. First of all, in the noncommutative integration, the simple replacement of the usual absolute value by the
modulus of operators in the formula $\big\|  (\sum_{j\geq 0}2^{qj\alpha}|  {\varphi}_j*f | ^q)^\frac{1}{q}\big\| _p$ does not give a norm except for $q = 2$. Even though one could use Pisier's definition of $\ell_q$-valued noncommutative $L_p$-spaces by complex interpolation (see \cite{Pisier}), we will not study that kind of spaces and will focus only on the case $q = 2$. The reason for this choice is that, for $q=2$, the Triebel-Lizorkin spaces of operator-valued distributions are isomorphic to the Hardy spaces developed in \cite{XX18}, as mentioned above. Another difficulty is the lack of pointwise maximal functions in the noncommutative case. As is well known, the maximal functions play a crucial role in the classical theory; but they are no longer at our disposal in the noncommutative setting. In \cite{XXY17}, when studying the Triebel-Lizorkin spaces on quantum tori, we use Calder\'on-Zygmund and Fourier multiplier theory as substitution. In this paper, we will still rely heavily on this theory. However, we have to consider its local (or inhomogeneous) counterpart, since the theory used in \cite{XXY17} for quantum tori are nonlocal (or homogeneous) ones. Besides the local nature, we also develop Hilbert space valued Fourier multiplier theory, which will be used to deduce general characterizations of operator-valued Triebel-Lizorkin spaces by the Lusin type square function.

\medskip

Our definition of (column) operator-valued Triebel-Lizorkin spaces is 
$$
F_p^{\alpha,c}(\R,\M)=\lbrace f\in \mathcal{S}'(\R; L_1(\M)+\M): \|  f\| _{F_p^{\alpha,c}}<\infty\rbrace,
$$
where
$$
\|  f\| _{F_p^{\alpha,c}}= \big\|  (\sum_{j\geq 0}2^{2j\alpha}|  {\varphi}_j*f | ^2)^\frac{1}{2}\big\| _p.
$$
Here the norm $\|\cdot\|_p$ is the norm of the semi-commutative $L_p$-space $L_p(L_\infty(\R) \overline{\otimes} \M)$. Different from the classical case, we have also row and mixture versions; see section \ref{section-TL} for concrete definitions.

We present here two major results of this paper. The first one gives general characterizations of $F^{\alpha,c}_p (\R,\M)$ by any reasonable convolution kernels in place of the Littlewood-Paley decomposition $\{  \varphi_j\}_{j\geq 0}$. These characterizations can be realized either by the Littlewood Paley type $g$-function or by the Lusin type integral function, with the help of the Calder\'on-Zygmund theory and Fourier multiplier theory mentioned above. The second major result is the atomic decomposition of $F^{\alpha,c}_1 (\R,\M)$. When $\alpha=0$, in \cite{XX18}, we deduce from the $\h_1$-$\bmo$ duality an atomic decomposition of $\h_1^c(\R,\M)$, which does not require any smooth condition on each atom. In this paper, we refine the smoothness of that atomic decomposition by the Calder\'on reproducing identity, via tent spaces. Using the same trick, we extend that refinement to $F^{\alpha, c}_1 (\R,\M)$; but compared with the case of local Hardy spaces, subatoms enter in the game. These smooth atomic decompositions will play a crucial role in the study of pseudo-differential operators in the forthcoming paper \cite{XX18PDO}.

\bigskip

In the following, let us recall some notation and background in the interface between harmonic analysis and operator algebras that we will need throughout the paper, although they are probably well-known to experts.

\subsection*{Noncommutative $L_{p}$-spaces}
 We start with a brief introduction of noncommutative $L_p$ spaces.
Let $\M$ be  a von Neumann algebra equipped with a
normal semifinite faithful trace $\tau$; for $1\leq p\leq\infty$, let $L_p(\M)$ be the noncommutative $L_p$-space associated to $(\M,\tau)$.
The norm of  $L_p(\M)$ will be often denoted simply by $\|\cdot \|_p$. But if different  $L_p$-spaces  appear in a same context, we will sometimes precise the respective $L_p$-norms in order to avoid possible ambiguity. The reader is referred to \cite{PX2003} and  \cite{Xu2007} for more information on noncommutative $L_p$-spaces. We will also need Hilbert space-valued noncommutative $L_p$-spaces (see \cite{JLX2006} for more details).
Let $H$ be a Hilbert space and  $v \in H$ with $\|v\|=1$. Let  $p_v$ be the orthogonal projection onto
the one-dimensional subspace generated by $v$. Define
 $$L_p(\M; H^{r})=(p_v\otimes 1_{\M}) L_p(B(H)\overline\otimes \M)\;\textrm{ and }\;
 L_p(\M; H^{c})= L_p(B(H)\overline\otimes \M)(p_v\otimes 1_{\M}).$$
These are the row and column noncommutative $L_p$-spaces.   Like the classical $L_p$-spaces, noncommutative $L_p$-spaces form an interpolation scale with respect to the complex interpolation method: For $1\leq p_0<p_1\leq \infty$ and $0<\eta<1$, we have
 $$\big( L_{p_0}(\M),\, L_{p_1}(\M) \big)_{\eta}= L_{p}(\M)\;\text{ with equal norms},$$
where $\frac{1}{p}=\frac{1-\eta}{p_0}+\frac{\eta}{p_1}$. Since $ L_p(\M; H^c)$ and $ L_p(\M; H^r)$ are 1-complemented subspaces of $L_p(B(H)\overline\otimes \M)$, for the same indicies, we have 
 \begin{equation*} 
 \big( L_{p_0}(\M; H^c),\, L_{p_1}(\M; H^c) \big)_{\eta}= L_{p}(\M; H^c)\;\text{ with equal norms}.
 \end{equation*}

 \subsection*{Fourier analysis}
Fourier multipliers will be one of the most important tools of this paper. Let us give some Fourier multipliers that will be frequently used. They are all very well known in the classical harmonic theory.

First, we recall the symbols of Littlewood-Paley decomposition on $\R$.
Fix a Schwartz function $\varphi$ on $\mathbb{R}^d$ satisfying:
\begin{equation}\label{condition-phi}
\begin{cases}
\supp \varphi \subset \lbrace \xi:\frac{1}{2}\leq |  \xi |  \leq 2\rbrace.\\
\varphi >0 \mbox{ on } \lbrace \xi:\frac{1}{2}< |  \xi |  < 2\rbrace,\\
\sum_{k\in \mathbb{Z}}\varphi (2^{-k}\xi)=1, \forall \, \xi \neq 0.
\end{cases}
\end{equation}
For each $k \in \mathbb{N}$, let $\varphi_k$ be the function whose Fourier transform is equal to $\varphi(2^{-k}\cdot)$, and let $\varphi_0$ be the function whose Fourier transform is equal to $1-\sum_{k>0}\varphi (2^{-k}\cdot)$. Then $\{\varphi_k\}_{k\geq 0}$ gives a Littlewood-Paley decomposition on $\mathbb{R}^d$ such that
\begin{equation}
{\supp} \widehat\varphi_k\subset\{\xi\in\mathbb{R}^{d}:2^{k-1}\leq  |\xi |\leq2^{k+1}\},\quad \forall \, k\in\mathbb{N},\;\;\;\supp \widehat\varphi_0\subset\{\xi\in\mathbb{R}^{d}:  |\xi |\leq2\}\label{eq: supp varphi_j}
\end{equation}
and that
\begin{equation}
\sum_{k=0}^{\infty}\widehat\varphi_k(\xi)=1\quad \forall \,\xi\in\mathbb{R}^{d}.\label{eq:resolution of unity}
\end{equation}
The homogeneous counterpart of the above decomposition is given by $\{\dot \varphi_k\}_{k\in \mathbb{Z}}$. This time, for every $k\in \mathbb{Z}$, these functions are given by $\widehat{\dot{\varphi}}_k(\xi)=\varphi (2^{-k}\xi)$. We have
\begin{equation}\label{eq:resolution of unity homo}
\sum_{k\in \mathbb{Z}} \widehat{\dot{\varphi}}_k(\xi)=1\quad \forall \,\xi \neq 0.
\end{equation}

The Bessel potential and the Riesz potential are $J^\alpha =(1-(2\pi)^{-2}\Delta  )^{\frac \alpha  2}$ and  $I^\alpha =(-(2\pi)^{-2}\Delta  )^{\frac \alpha  2}$, respectively. If $\alpha=1$, we will abbreviate $J^1$ as $J$ and $I^1$ as $I$. We denote also $J_\alpha (\xi)=(1+|\xi|^2 )^{\frac \alpha  2}$ on $\mathbb{R}^d$ and $I_\alpha (\xi)=|\xi|^ \alpha  $ on $\mathbb{R}^d \setminus \{0\}$. Then $J_\alpha (\xi)$ and $I_\alpha (\xi)$ are the symbols of the Fourier multipliers $J^\alpha $ and $I^\alpha $, respectively.

Given a Banach space $X$, let $\mathcal{S}(\R; X)$ be the space of  $X$-valued rapidly decreasing  functions on $\R$ with the standard Fr\'{e}chet topology, and $\mathcal{S}'(\R; X)$  be the space of continuous linear maps from $\mathcal{S}(\R)$ to $X$. All operations on $\mathcal{S}(\R)$ such as derivations, convolution and Fourier transform transfer to $\mathcal{S}'(\R; X)$ in the usual way. On the other hand, $L_p(\R; X)$ naturally embeds into $\mathcal{S}'(\R; X)$ for $1\leq p \leq \infty$, where $L_p(\R; X)$ stands for the space of strongly $p$-integrable functions from $\R$ to $X$. By this definition, Fourier multipliers on $\R$, in particular the Bessel and Riesz potentials, extend to vector-valued tempered distributions in a natural way.

We denote by $H_2^\sigma(\R)$ the potential Sobolev space, consisting of all tempered distributions $f$ such that $J^\sigma(f)\in L_2(\R)$. If $\sigma >\frac  d  2 $, we have
 \begin{equation*}\begin{split}
  \big\|\mathcal F^{-1}(f)\big\|_1
  &=\int_{|s|\leq 1}\big|\mathcal F^{-1}(f)(s)\big|ds +\sum_{k\geq 0}\int_{2^{k}<|s|\leq 2^{k+1}}\big|\mathcal F^{-1}(f)(s)\big|ds \\
  &\leq C_1 \Big(\int_{|s|\leq 1}\big|\mathcal F^{-1}(f)(s)\big|^2ds +\sum_{k\geq 0}2^{2k \sigma}\int_{2^{k}<|s|\leq 2^{k+1}}\big|\mathcal F^{-1}(f)(s)\big|^2ds\Big)^{\frac12} \\
  &\leq C_2 \big\| f\big\|_{H_2^\sigma}\,,
  \end{split}\end{equation*}
where $C_1$ and $C_2$ are uniform constants. Therefore, if $\widehat{\phi }\in H_2^\sigma(\R) $, the following Young's inequality
\begin{equation}\label{Young-Banach}
\| \phi * g \|_{L_p(\R; X)} \leq  \|  \phi \|_{1}  \| g\|_{L_p(\R; X)}  \leq  C_2 \| \widehat{\phi} \|_{H_2^{\sigma}}  \| g\|_{L_p(\R; X)} 
\end{equation}
holds for any $g\in L_p(\R; X)$ with $1\leq p \leq \infty$. Here $X$ is an arbitrary Banach space. Inequality \eqref{Young-Banach} indicates that functions in $H_2^\sigma(\R)$ are the symbols of bounded Fourier multipliers, even in the vector-valued case.

In the sequel, we will mainly consider the case $X= L_1(\M) + \M$, i.e., consider operator-valued functions or distributions on $\R$. We will frequently use the following Cauchy-Schwarz type inequality
for operator-valued square function,
\begin{equation}\label{eq: 2-1}
  \big|\int_{\mathbb{R}^d}\phi (s)f(s)ds \big|^{2}\leq\int_{\mathbb{R}^d}  |\phi(s) |^{2}ds\int_{\mathbb{R}^d}  |f(s) |^{2}ds,
\end{equation}
where $\phi:\mathbb{R}^d\rightarrow\mathbb{C}$ and $f:\mathbb{R}^d\rightarrow L_{1}(\mathcal{M} )+\mathcal{M}$
are functions such that all integrations of the above inequality make sense.
We also require the operator-valued version of the Plancherel formula. For sufficiently nice functions $f: \mathbb{R}^{d}\rightarrow L_{1}  (\mathcal{M} )+\M$,
for example, for $f \in L_{2}  (\mathbb{R}^{d} )\otimes L_{2}  (\mathcal{M} )$,
we have
\begin{equation}
\int_{\mathbb{R}^d} |f  (s )|^2 ds=\int_{\mathbb{R}^d}  |  \widehat{f}  (\xi )|^2  d\xi.\label{eq: Planchel}
\end{equation}

\medskip
 
Throughout, we will use the notation $A\lesssim B$,
which is an inequality up to a constant: $A\leq cB$ for some constant
$c>0$. The relevant constants in all such inequalities may depend
on the dimension $d$, the test function $\Phi$ or $p$, etc, but
never on the function $f$ in consideration. The equivalence $A\approx B$
will mean $A\lesssim B$ and $B\lesssim A$ simultaneously.

\bigskip

The layout of this paper is the following. In the next section, we briefly introduce the definition of local Hardy spaces, and the main results in \cite{XX18}. In section \ref{section: Multiplier theorems}, we develop several Fourier multiplier theorems: the first one is the inhomogeneous version of the Fourier multiplier theorem proved in \cite{XXY17}, fitted to local Hardy spaces; the second is a Hilbertian Fourier multiplier theorem, in order to deal with the Lusin area square functions. In section \ref{section-TL}, we give the definition of Triebel-Lizorkin spaces, and some immediate properties. Section \ref{section-charact} is devoted to different characterizations of Triebel-Lizorkin spaces. The proofs in this section are technical and tedious, based on Calder\'on-Zygmund theory and Fourier multiplier theorems. In the last section, we demonstrate the smooth atomic decompositions of $F_p^{\alpha,c}(\R,\M)$: we begin with the space $F_p^{0,c}(\R,\M)= \h_p^c(\R,\M)$, and then extend the result to general $\alpha$ by a similar argument.

\section{Operator-valued local Hardy spaces}\label{section-oHardy}

Let us review the operator-valued local Hardy spaces studied in \cite{XX18}, and collect some of the main results there that will be useful in this paper. We keep the following notation: $(\M,\tau)$ is a von Neumann algebra with n.s.f. trace, and $\N=L_\infty(\R)\overline\otimes\M$ is equipped with the tensor trace; letters $s, t$ are used to denote variables of $\R$, while letters $x, y$ are reserved for operators in noncommutative $L_p$-spaces.

Let $\mathrm{P}$ be the Poisson kernel on $\R$:
 $$\mathrm{P}(s)=c_d\,\frac{1}{(|s|^2+1)^{\frac{d+1}2}}$$
with $c_d$ the usual normalizing constant and $|s|$ the Euclidean norm of $s$. Let
  $$\mathrm{P}_\e(s)=\frac1{\e^d}\, \mathrm{P}(\frac s\e)=c_d\,\frac{\e}{(|s|^2+\e^2)^{\frac{d+1}2}}\,.$$
For any function $f$ on $\R$ with values in $L_1(\M)+\M$,  its Poisson integral, whenever it exists, will be denoted by $\mathrm{P}_\e(f)$:
 $$\mathrm{P}_\e(f)(s)=\int_{\mathbb{R}^d}\mathrm{P}_\e(s-t)f(t)dt, \quad (s,\e)\in \mathbb{R}^{d+1}_+.$$
The truncated Lusin area square function of $f$
by
\begin{equation*}
s^{c}  (f )  (s )  =  \Big(\int_{\widetilde{\Gamma}}  \big|\frac{\partial}{\partial\varepsilon}\mathrm{P}_{\varepsilon}  (f )  (s+t ) \big|^{2}\frac{dtd\varepsilon}{\varepsilon^{d-1}}\Big)^{\frac{1}{2}},\,s\in\mathbb{R}^{d},
\end{equation*}
where $\widetilde{\Gamma}$ is the truncated cone $  \{ (t,\varepsilon )\in\mathbb{R}_{+}^{d+1}:  |t |<\varepsilon<1 \} $. Denote by $\mathrm{R}_d$ the Hilbert space $L_2(\R,\frac{dt}{1+|t|^{d+1}})$.
For $1\leq p<\infty$, define the column local Hardy space $\h _p^c(\R,\M)$ to be
\[
\h _p^c(\R,\M)=\{f\in L_1(\M; \mathrm{R}_d^c)+L_{\infty}(\M;\mathrm{R}_d^c) : \|  f\| _{\h_p^c} <\infty\},
\]
where the $\h_{p}^{c}  (\mathbb{R}^{d},\mathcal{M} )$-norm of $f$
is defined by
\begin{equation*}
  \|  f \|  _{\h_{p}^{c}  (\mathbb{R}^{d},\mathcal{M} )}  =    \|  s^{c}  (f )  \|  _{L_{p}  (\mathcal{N} )}+  \| \mathrm{P}*f \| _{L_{p}  (\mathcal{N} )}.
\end{equation*}
The row local Hardy space $\h_{p}^{r}(\R,\M)$ is the space of all $f$ such that $f^*\in \h_{p}^{c}(\R,\M)$, equipped with the norm $ \|  f\| _{\h_p^r}= \|  f^*\| _{\h_p^c}$.
Moreover, define the mixture space $\h_{p}  (\mathbb{R}^{d},\mathcal{M} )$
as follows:
\[
\h_{p}  (\mathbb{R}^{d},\mathcal{M} )=\h_{p}^{c}  (\mathbb{R}^{d},\mathcal{M} )+\h_{p}^{r}  (\mathbb{R}^{d},\mathcal{M} )\mbox{ for }1\leq p\leq2
\]
 equipped with the sum norm
\[
  \|  f \|  _{\h_{p}  (\mathbb{R}^{d},\mathcal{M} )}=\inf  \{   \|  g \|  _{\h_{p}^{c}}+  \|  h \|  _{\h_{p}^{r}}:\, f=g+h,g\in \h_{p}^{c}  (\mathbb{R}^{d},\mathcal{M} ),h\in \h_{p}^{r}  (\mathbb{R}^{d},\mathcal{M} ) \} ,
\]
 and
\[
\h_{p}  (\mathbb{R}^{d},\mathcal{M} )=\h_{p}^{c}  (\mathbb{R}^{d},\mathcal{M} )\cap \h_{p}^{r}  (\mathbb{R}^{d},\mathcal{M} )\mbox{ for }2<p<\infty
\]
equipped with the intersection norm
\[
  \|  f \|  _{\h_{p} }=\max  \{   \|  f \|  _{\h_{p}^{c}},  \|  f \|  _{\h_{p}^{r}} \} .
\]
The local analogue of the Littlewood-Paley $g$-function of $f$ is
defined by
\begin{equation*}
g^{c}  (f )  (s )  = \big(\int_{0}^1  |\frac{\partial}{\partial\varepsilon} \mathrm{P}_{\varepsilon}  (f )  (s ) |^{2}\varepsilon d\varepsilon \big)^{\frac{1}{2}},\,s\in\mathbb{R}^{d}.
\end{equation*}
It is proved in \cite{XX18} that
$$   \|  f \|  _{\h_{p}^{c}} \approx  \|  g^{c}  (f ) \|  _p +\| \mathrm{P}*f \| _p$$
for all $1\leq p <\infty$.

\medskip

The dual of $\h_1^c(\R,\M)$ is characterized as a local version of bmo space, defined as follows. For any cube $Q\subset \mathbb{R}^{d}$, we denote its volume by $  |Q |$.
Let $f\in L_{\infty}  (\mathcal{M};\mathrm{R}^c_d)$. The
mean value of $f$ over $Q$ is denoted by $f_{Q}:=\frac{1}{  |Q |}\int_{Q}f  (s )ds$.
Set
\begin{equation}\label{eq: def bmo}
  \|  f \|  _{{\bmo}^{c}  (\mathbb{R}^{d},\mathcal{M} )}  = \max   \Big\{ \sup_{  |Q |<1}  \big\|  (\frac{1}{  |Q |}\int_{Q}  |f-f_{Q} |^{2}dt )^{\frac{1}{2}} \big\|  _{\mathcal{M}} , \sup_{  |Q |=1}  \big\|  (\int_{Q}  |f |^{2}dt)^{\frac{1}{2}} \big\|  _{\mathcal{M}} \Big\}  .
\end{equation}
The local version of bmo spaces are defined as
\[
{\bmo}^{c}  (\mathbb{R}^{d},\mathcal{M} )=  \{ f\in L_{\infty}  (\mathcal{M};\mathrm{R}_{d}^{c} ):\,  \|  f \|  _{{\bmo}^{c}}<\infty \} .
\]
Define ${\bmo}^{r}  (\mathbb{R}^{d},\mathcal{M} )$
to be the space of all $f\in L^{\infty}  (\mathcal{M};\mathrm{R}_{d}^{r} )$
such that $  \|  f^{*} \|  _{{\bmo}^{c}  (\mathbb{R}^{d},\mathcal{M} )} $ is finite,
with the norm $  \|  f \|  _{{\bmo}^{r}}=  \|  f^{*} \|  _{{\bmo}^{c}}.$
And ${\bmo}  (\mathbb{R}^{d},\mathcal{M} )$ is defined
as the intersection of ${\bmo}^{c}  (\mathbb{R}^{d},\mathcal{M} )$ and ${\bmo}^{r}  (\mathbb{R}^{d},\mathcal{M} )$, equipped with the intersection norm.

\medskip

The above Hardy and bmo type spaces are local analogues of the spaces studied by Mei \cite{Mei2007}. They turn out to have similar properties with their non-local versions, such as duality and interpolation. The following two theorems are quoted from \cite{XX18}.

\begin{thm}\label{dual-hardy}
We have $  \h_{1}^{c}  (\mathbb{R}^{d},\mathcal{M} )^*={\bmo}^{c}  (\mathbb{R}^{d},\mathcal{M} )$
with equivalent norms. If $1< p <2$ and $q$ is its conjugate index, then $  \h_{p}^{c}  (\mathbb{R}^{d},\mathcal{M} )^*={\h}_q^{c}  (\mathbb{R}^{d},\mathcal{M} )$
with equivalent norms.
\end{thm}

\begin{thm}\label{interpolation-hardy}
Let $1<p<\infty$. We have
\begin{enumerate}[$\rm (1)$]
\item $
\big({\bmo}^c(\R,\M), \h_1^c(\R,\M)\big)_{\frac{1}{p}}=\h_p^c(\R,\M).$
\item $\big(X, Y \big)_{\frac 1 p } = L _p(\N)$, where $X = \bmo(\R,\M)$ or $L_\infty(\N)$, and $Y =\h_1(\R,\M)$ or $L_1(\N)$.
\end{enumerate}

\end{thm}

\subsection*{Calder\'on-Zygmund theory}
The usual Calder\'on-Zygmund operators which satisfy the H\"ormander condition are not  necessarily bounded on local Hardy spaces. In order to guarantee the boundedness of a Calder\'on-Zygmund operator on $\h_p^c(\R,\M)$, an extra decay at infinity is imposed on the kernel  in \cite{XX18}.
Let $K\in \mathcal{S}' (\R; L_{1}  (\mathcal{M} )+\M)$ coincide on $\mathbb{R}^{d}\setminus  \{ 0 \} $
with a locally integrable $L_{1}  (\mathcal{M} )+\mathcal{M}$-valued
function. We define the left singular integral operator $K^{c}$ associated to $K$ by
\[
K^{c}  (f )  (s )=\int_{\mathbb{R}^{d}}K  (s-t )f  (t )dt,
\]
 and the right singular integral operator $K^{r}$ associated to $K$ by
\[
K^{r}  (f )  (s )=\int_{\mathbb{R}^{d}}f  (t )K  (s-t )dt.
\]
Both   $K^{c}  (f )$ and $K^{r}  (f )$ are well-defined
for sufficiently nice functions $f$ with values in $L_{1}  (\mathcal{M} )\cap \M$,
for instance, for $f\in\mathcal{S}\otimes  (L_{1}  (\mathcal{M} )\cap M )$.

Let ${\bmo}_{0}^{c}  (\mathbb{R}^{d},\mathcal{M} )$ denote the subspace of 
${\bmo}^{c}  (\mathbb{R}^{d},\mathcal{M} )$ consisting of compactly supported 
functions. The extra decay of the kernel $K$ given in \cite{XX18} is condition (2) in the following lemma.

\begin{lem}
\label{C-Z lem}
Assume that
\begin{enumerate}[\rm(1)]
\item the Fourier transform of $K$ is
bounded: $\sup_{\xi\in\mathbb{R}^{d}}  \|  \widehat{K}  (\xi ) \|  _{\mathcal{M}}<\infty$;
\item $K$ satisfies a size estimate: there exist $C_1$ and $\rho >0$ such that 
\[
  \|  K  (s ) \|  _{\mathcal{M}}\leq \frac{C_1}{  |s |^{d+\rho}},\thinspace\forall |s|\geq 1;
\]
\item $K$ has the Lipschitz regularity:
there exist a constant $C_2$ and $\gamma >0$ such that
\[
  \|  K  (s-t )-K  (s ) \|  _{\mathcal{M}}\leq C_2\frac{  |t |^\gamma}{  |s-t |^{d+\gamma}},\thinspace\forall  |s |>2  |t |.
\]
\end{enumerate}
Then $K^{c}$ is bounded on $\h_{p}^{c}  (\mathbb{R}^{d} ,\mathcal{M} )$ for $1\leq p<\infty$
and from ${\bmo}_{0}^{c}  (\mathbb{R}^{d},\mathcal{M} )$
to ${\bmo}^{c}  (\mathbb{R}^{d},\mathcal{M} )$.

A similar statement also holds for $K^{r}$ and the corresponding
row spaces.
 \end{lem}

\subsection*{Characterizations}
Next, we are going to present the characterizations of local Hardy spaces obtained in \cite{XX18}, which will play an important role when studying the characterizations of Triebel-Lizorkin spaces in this paper. 

The main idea of these characterizations is to replace the Poisson kernel by good enough Schwartz functions. 
Let $\Phi$ be a Schwartz function on $\mathbb{R}^{d}$ of vanishing mean, and set $\Phi_{\varepsilon}(s)=\varepsilon^{-d}\Phi(\frac{s}{\varepsilon})$ for positive $\e$. $\Phi$ is said to be
nondegenerate if:
\begin{equation}\label{nondegenerate-Phi}
\forall\xi\in\mathbb{R}^{d}\setminus\{0\}  \,\,\exists\,\varepsilon>0\;\,\text{ s.t. }\,\widehat{\Phi}(\varepsilon\xi)\neq0.
\end{equation}
Then there exists a Schwartz function $\Psi$ of vanishing mean such that
\begin{equation}\label{eq: reproduce}
\int_{0}^{\infty}\widehat{\Phi}(\varepsilon\xi)\overline{\widehat{\Psi}(\varepsilon\xi)}\frac{d\varepsilon}{\varepsilon}=1,\quad \forall\xi\in\mathbb{R}^{d}\setminus\left\{ 0\right\} .
\end{equation}
Furthermore, we can find two functions $\phi$, $\psi$ such that  $\widehat{\phi}, \widehat{\psi}\in H_2^\sigma(\R)$, $\widehat{\phi}(0)>0, \widehat{\psi}(0)> 0$ and
\begin{equation}\label{eq: reproduce 2}
\widehat{\phi}(\xi)\overline{\widehat{\psi}(\xi)}=1-\int_{0}^{1}\widehat{\Phi}(\varepsilon\xi)\overline{\widehat{\Psi}(\varepsilon\xi)}\frac{d\varepsilon}{\varepsilon}.
\end{equation}

For any $f\in L_{1}  (\mathcal{M};\mathrm R_{d}^{c} )+L_{\infty}  (\mathcal{M};\mathrm R_{d}^{c} )$,
we define the local versions of the conic and radial square functions
of $f$ associated to $\Phi$ by
\begin{equation*}
\begin{split}
s_{\Phi}^{c}  (f )  (s ) & =  \Big(\iint_{\widetilde{\Gamma}}  |\Phi_{\varepsilon}*f  (s+t ) |^{2}\frac{dtd\varepsilon}{\varepsilon^{d+1}} \Big)^{\frac{1}{2}},\thinspace s\in\mathbb{R}^{d},\\
g_{\Phi}^{c}  (f )  (s ) & =  \Big(\int_{0}^{1}  |\Phi_{\varepsilon}*f  (s ) |^{2}\frac{d\varepsilon}{\varepsilon} \Big)^{\frac{1}{2}},\thinspace s\in\mathbb{R}^{d}.
\end{split}
\end{equation*}

Fix the four test functions $\Phi, \Psi, \phi, \psi$ as above. The following theorem is proved in \cite{XX18}.

\begin{thm}\label{thm main1}
Let $1\leq p<\infty$ and $\phi$, $\Phi$ be as above. For any $f\in L_{1}  (\mathcal{M};\mathrm R_{d}^{c} )+L_{\infty}  (\mathcal{M};\mathrm R_{d}^{c} )$, $f\in \h_{p}^{c}  (\mathbb{R}^{d},\mathcal{M} )$  if and only if
$s_{\Phi}^{c}  (f )\in L_{p}  (\mathcal{N} )$ and
$\phi*f\in L_{p}  (\mathcal{N} )$  if and only if $g_{\Phi}^{c}  (f )\in L_{p}  (\mathcal{N} )$
and $\phi*f\in L_{p}  (\mathcal{N} )$. If this is the
case, then
\begin{equation}\label{eq: main}
\|  f \|  _{\h_{p}^{c}}\approx \|  s_{\Phi}^{c}  (f )  \|  _{p}+  \|  \phi*f \|  _{p} \approx   \|  g_{\Phi}^{c}  (f )  \|  _{p}+  \|  \phi*f \|  _{p}
\end{equation}
 with the relevant constants depending only on $d, \Phi$ and $\phi$.
\end{thm}

We have a discrete version of Theorem \ref{thm main1}.
The square functions $s_{\Phi}^c$ and $g_{\Phi}^c$ can be discretized as follows:
\begin{equation*}
\begin{split}
 g_{\Phi}^{c,D}(f)(s) & =   \Big(\sum_{j\geq 1} |\Phi_j*f (s)|^2\Big)^{\frac 1 2},\\
 s_{\Phi}^{c,D}(f)(s) & = ,\Big(\sum_{j\geq 1} 2^{dj}\int_{B(s, 2^{-j})} |\Phi_j*f (t)|^2 dt\Big)^{\frac 1 2}.
\end{split}
\end{equation*}
Here $\Phi_j$ is the inverse Fourier transform of $\Phi(2^{-j}\cdot)$. This time, to get a resolvent of the unit on $\R$, we need to assume that $\Phi$ satisfies
 \begin{equation*}
  \forall\,\xi\in\R\setminus\{0\}\,\, \exists\, 0<2a\le b<\infty\;\text { s.t. }\; \widehat{\Phi}(\e\xi)\neq0,\;\forall\; \e\in (a,\,b].
 \end{equation*}
 Then adapting the proof of \cite[Lemma V.6]{Stein1993} , we can find a Schwartz function $\Psi$ of vanishing mean  such that
 \begin{equation}\label{eq: reproduceD}
  \sum_{j= -\infty}^{+\infty}\widehat{\Phi} (2^{-j}\xi)\, \overline{\widehat{\Psi} (2^{-j}\xi)} = 1,\quad \forall\xi\in\R\setminus\{0\}.
 \end{equation}
Again, there exist two functions $\phi$ and $\psi$ such that $\widehat{\varphi}, \widehat{\psi}\in H_2^\sigma(\R)$, $\widehat{\phi}(0)>0, \widehat{\psi}(0)> 0$  and
 \begin{equation}\label{eq: reproduceD 2}
  \sum_{j=1}^\infty\widehat{\Phi} (2^{-j}\xi)\, \overline{\widehat{\Psi} (2^{-j}\xi)}+\widehat{\phi}(\xi)\overline{\widehat{\psi}(\xi)}=1, \quad \forall \xi \in \R.
 \end{equation}
Now we fix the pairs $(\Phi, \Psi)$ and $(\phi,\psi)$ satisfying \eqref{eq: reproduceD} and \eqref{eq: reproduceD 2}.

\begin{thm}\label{thm: equivalence hpD}
Let $\phi$ and $\Phi$ be test functions as in \eqref{eq: reproduceD 2} and  $1\leq p<\infty$. Then for any $f\in L_{1}  (\mathcal{M};\mathrm R_{d}^{c} )+L_{\infty}  (\mathcal{M};\mathrm R_{d}^{c} )$,
 $f\in \h_p^c(\R,\M)$  if and only if  $s_{\Phi}^{c, D}(f)\in L_{p}(\N)$ and $\phi *f\in L_{p}(\N)$   if and only if $g_{\Phi}^{c, D}(f)\in L_{p}(\N)$ and $\phi *f\in L_{p}(\N)$. Moreover,
 $$\|f\|_{\h^c_p} \approx \|s_{\Phi}^{c, D}(f)\|_{L_{p}(\N)}+\|\phi *f\|_{p} \approx \|g_{\Phi}^{c, D}(f)\|_{p}+\|\phi *f\|_{p} $$
with the relevant constants depending only on $ d, \Phi$ and $\phi $.
\end{thm}

\subsection*{Atomic decomposition}
Finally, let us include the atomic decomposition of the local Hardy space $\h_{1}^{c}  (\mathbb{R}^{d},\mathcal{M} )$.
Let $Q$ be a cube in $\R$ with $| Q|\leq 1$. If $| Q|=1$, an $\h _1^c$-atom associated with $Q$ is a function $a\in L_{1}  (\mathcal{M};L_{2}^{c}  (\mathbb{R}^{d} ) )$ such that
\begin{itemize}
\item $\supp a\subset Q$;
\item $\tau\big(\int_{Q}  |a  (s ) |^{2}ds\big)^{\frac{1}{2}}\leq  |Q |^{-\frac{1}{2}}$.
\end{itemize}
If $|Q|<1$, we assume additionally: 
\begin{itemize}
\item $\int_{Q}a(s)ds=0.$
\end{itemize}
Let $\h_{1,at}^{c}  (\mathbb{R}^{d},\mathcal{M} )$ be the
space of all $f$ admitting a representation of the form
\[
f=\sum_{j=1}^\infty \lambda_ja_j,
\]
where the $a_{j}$'s are $\h_1^{c}$-atoms and $\lambda_{j}\in\mathbb{C}$ such that $\sum_{j=1}^\infty  |\lambda_{j} |<\infty$. The above
series  converges in the sense of distribution.
We equip $\h_{1,at}^{c}  (\mathbb{R}^{d},\mathcal{M} )$ with
the following norm:
\[
  \|  f \|  _{\h_{1,at}^{c}}=\inf  \{ \sum_{j=1}^\infty  |\lambda_{j} |: f=\sum_{j=1}^\infty\lambda_{j}a_{j};\,\mbox{\ensuremath{a_{j}}'s are \ensuremath{\h_1^{c}} -atoms, }\mbox{\ensuremath{\lambda}}_{j}\in\mathbb{C} \} .
\]
Similarly, we can define the row and mixture versions. The following theorem is also proved in \cite{XX18}.

\begin{thm}\label{thm:atomic h1}
We have $\h_{1,at}^{c}  (\mathbb{R}^{d},\mathcal{M} )=\h_{1}^{c}  (\mathbb{R}^{d},\mathcal{M} )$
with equivalent norms.
\end{thm}

\begin{rmk}\label{rem: replace support}
In the above definition of atoms, we can replace the support of atoms  $Q$ by any bounded multiple of $Q$.
\end{rmk}

\section{Multiplier theorems}\label{section: Multiplier theorems}

We are going to develop some Fourier multiplier theorems in this section. They can be viewed as a special case of Calder\'on-Zygmund theory, and will be used to investigate various square funtions that characterize the Triebel-Lizorkin spaces. Our presentation follows closely the argument in Section 4.1 of \cite{XXY17}.

Recall again that $\varphi$ is a fixed  function satisfying \eqref{condition-phi}, $\varphi_0$ is the inverse Fourier transform of $1-\sum_{k>0}\varphi (2^{-k}\cdot)$, and $\varphi_k$ is the inverse Fourier transform of $\varphi (2^{-k}\cdot)$ when $k>0$. Moreover,  we denote by $\varphi^{(k)} $ the Fourier transform of $\varphi_k$ for every $k\in \mathbb{N}_0$ ($ \mathbb{N}_0$ being the set of nonnegative integers).

\subsection{Global multipliers}

Firstly, let us state the following homogeneous version of \cite[Theorem~4.1]{XXY17}.
 \begin{thm}\label{multiplier-homogeneous}
Let $\sigma \in \mathbb{R}$ with $\sigma >\frac{d}{2}$.
Assume that $( {\phi_j})_{j\in \mathbb{Z}}$ and $( {\rho_j})_{j\in \mathbb{Z}}$ are two sequences of functions on $\R\backslash\lbrace 0\rbrace$ such that
\[
\supp  {\phi_j \rho_j} \subset  \{\xi : 2^{j-1}\leq |  \xi |  \leq 2^{j+1}\}, \;\; j\in \mathbb{Z}
\]
and
\begin{equation*}
\underset{\substack{j\in \mathbb{Z}\\ -2\leq k \leq 2}}{\sup} \|   {\phi} _j (2^{j+k}\cdot)\varphi\| _{H_2^\sigma(\R)}<\infty.
\end{equation*}
Let $1< p<\infty$. Then for any $f\in \mathcal{S}'(\R;L_1(\M)+\M)$, we have
\[
\big\| (\sum_{j\in \mathbb{Z}} 2^{2j\alpha} |  \check \phi_j*\check \rho_j*f |^2)^\frac{1}{2}\big\| _{p}\lesssim \underset{\substack{j\in \mathbb{Z}\\ -2\leq k \leq 2}}{\sup} \|   {\phi}_j (2^{j+k}\cdot)\varphi  \| _{H_2^\sigma}\big\| (\sum_{j\in \mathbb{Z}}  2^{2j\alpha} |  \check \rho_j*f |  ^2)^\frac{1}{2}\big\| _{p},
\]
where the constant depends on $p$, $\sigma$, $d$ and $\varphi$.

\end{thm}

\begin{proof}
Without loss of generality, we may take $\alpha=0$. It suffices to show that for any integer $K$, 
\begin{equation}\label{ineq-Fourier-Multi}
\big\| (\sum_{j\geq K} |  \check \phi_j*\check \rho_j*f |^2)^\frac{1}{2}\big\| _{p}\lesssim \underset{\substack{j\in \mathbb{Z}\\ -2\leq k \leq 2}}{\sup}\|   {\phi}_j (2^{j+k}\cdot)\varphi  \| _{H_2^\sigma} \big\| (\sum_{j\geq K}|  \check \rho_j*f |^2)^\frac{1}{2} \big\| _{p},
\end{equation}
with the relevant constant independent of $K\in \mathbb{Z}$. To this end, we set
$$\psi_{j-K}  = \phi_j(2^K \cdot) ,\;\;\;\eta_{j-K} = \rho_j(2^K \cdot) ,\;\;\mbox{and}\;\; \widehat{g} =\widehat{f}(2^K\cdot).$$
By easy computation, we have
$$
\supp  {\psi_j \eta_j} \subset  \{\xi : 2^{j-1}\leq |  \xi |  \leq 2^{j+1}\},  \;\;\forall\, j\geq0,
$$
and
$$ \check \phi_j*\check \rho_j*f=  2^{dK} \check \psi_{j-K}*\check \rho_{j-K}*g(2^K \cdot). $$
This ensures
\begin{equation}\label{eq:phi-rho-f}
\big\| (\sum_{j\geq K} |  \check \phi_j*\check \rho_j*f |^2)^\frac{1}{2} \big\| _{p}= 2^{\frac{(p-1)dK}{p}} \big\| (\sum_{j\geq 0} |  \check \psi_j*\check \eta_j*g |^2)^\frac{1}{2}\big\| _{p}.
\end{equation}
Similarly,
\begin{equation}\label{eq:rho-f}
\big\| (\sum_{j\geq K} |  \check \rho_j*f |^2)^\frac{1}{2}\big\| _{p}= 2^{\frac{(p-1)dK}{p}} \big\| (\sum_{j\geq 0} |   \check \eta_j*g|^2)^\frac{1}{2} \big\| _{p}.
\end{equation}
Moreover, since ${\psi}_j (2^{j+k}\cdot) = {\phi}_{j+K} (2^{j+k+K}\cdot)$, we have
\begin{equation}\label{ineq-psi-phi}
\begin{split}
\underset{\substack{j\geq0 \\ -2\leq k \leq 2}}{\sup} \|   {\psi}_j (2^{j+k}\cdot)\varphi  \| _{H_2^\sigma} & = \underset{\substack{j\geq K\\ -2\leq k \leq 2}}{\sup} \|   {\phi}_j (2^{j+k}\cdot)\varphi  \| _{H_2^\sigma }\\
&\leq  \underset{\substack{j\in \mathbb{Z}\\ -2\leq k \leq 2}}{\sup} \|   {\phi}_j (2^{j+k}\cdot)\varphi  \| _{H_2^\sigma} .
\end{split}
\end{equation}
Now applying \cite[Theorem~4.1]{XXY17} to $\psi_j$, $\rho_j$ and $g$ defined above, we obtain
$$\big\| (\sum_{j\geq 0} |  \check \psi_j*\check \eta_j*g|^2)^\frac{1}{2}\| _{p} \lesssim \underset{\substack{j\geq0 \\ -2\leq k \leq 2}}{\sup}\|   {\psi}_j (2^{j+k}\cdot)\varphi  \| _{H_2^\sigma(} \big\| (\sum_{j\geq 0} |  \check \eta_j*g|^2)^\frac{1}{2}\| _{p}.$$
Putting \eqref{eq:phi-rho-f}, \eqref{eq:rho-f} and \eqref{ineq-psi-phi} into this inequality, we then get \eqref{ineq-Fourier-Multi}, which yields Theorem \ref{multiplier-homogeneous} by approximation.
\end{proof}

Theorem \ref{multiplier-homogeneous} is developed to deal with the multiplier problem of square functions, and also the multiplier problem of the Hardy spaces $\mathcal{H}_p^c(\mathbb{R}^d,\M)$ by virtue of their characterizations (see \cite{XXX17}). In order to deal with the corresponding problems on the inhomogeneous versions of square functions or Hardy spaces, we need the following global version of Theorem \ref{multiplier-homogeneous}. The main difference is that in the inhomogeneous case, we need a careful analysis of the convolution kernel near the origin.

\begin{thm}\label{multiplier-global}
Let $1<p<\infty, \alpha \in \mathbb{R}$ and $\sigma >\frac{d}{2}$. Assume that $(\phi_j)_{j\geq 0}$ and $(\rho_j)_{j\geq 0}$ are two sequences of functions on $\mathbb{R}^d$ such that
\begin{gather*}
\supp (  \phi_j    \rho_j) \subset  \{\xi \in\R: 2^{j-1}\leq |  \xi |  \leq 2^{j+1}\},\;\;  j\in \mathbb{N},\\
\supp (  \phi_0   \rho_0) \subset  \{\xi \in \R: |  \xi |  \leq 2\},
\end{gather*}
and
\begin{equation}\label{eq:j condition thm}
\underset{\substack{j\geq 1\\ -2\leq k \leq 2}}{\sup} \|    {\phi} _j (2^{j+k}\cdot)\varphi\| _{H_2^\sigma (\R)}<\infty\quad \;\mbox{and }\;\;  \|    \phi_0 ({\varphi}^{(0)}+{\varphi}^{(1)})\| _{H_2^\sigma(\R)}<\infty.
\end{equation}
Then for any $L_1(\M)+\M$-valued distribution $f$,
\begin{equation*}
\begin{split}
 \big\| (\sum_{j\geq 0}2^{2j\alpha} |  \check \phi_j* \check \rho_j*f |  ^2)^\frac{1}{2}\big\| _{p}  & \lesssim  \max \big\{ \underset{\substack{j\geq 1 \\ -2\leq k \leq 2}}{\sup} \|    {\phi} _j (2^{j+k}\cdot)\varphi\| _{H_2^\sigma}, \|    \phi_0 ({\varphi}^{(0)}+{\varphi}^{(1)})\| _{H_2^\sigma}  \big\}\\
 &\; \;\;\;\cdot \big\| (\sum_{j\geq 0}2^{2j\alpha} |  \check \rho_j*f |  ^2)^\frac{1}{2}\big\| _{p},
\end{split}
\end{equation*}
where the constant depends only on $p$, $\sigma$, $d$ and $\varphi$.

\end{thm}

\begin{proof}
This theorem follows easily from its homogeneous version, i.e., Theorem \ref{multiplier-homogeneous}. Indeed, we can divide $\big\| (\sum_{j\geq 0}2^{2j\alpha} |  \check \phi_j* \check \rho_j*f |  ^2)^\frac{1}{2}\big\|_{p} $ into two parts
$$
\big\| (\sum_{j\geq 0}2^{2j\alpha} |  \check \phi_j* \check \rho_j*f |  ^2)^\frac{1}{2}\big\| _{p}  \approx \big\| (\sum_{j\geq 1}2^{2j\alpha} |  \check \phi_j* \check \rho_j*f |  ^2)^\frac{1}{2} \big\| _{p}  +\|   \check \phi_0* \check \rho_0*f  \| _{p} 
$$
and treat them separately. Applying Theorem  \ref{multiplier-homogeneous} to the sequences $(\phi_j)_{j\in \mathbb{Z}}$, $(\rho_j)_{j\in \mathbb{Z}}$ with $\phi_j=0$ and $\rho_j=0$ for $j\leq 0$, we get the estimate of the first term on the right hand side. The result is
$$\big\| (\sum_{j\geq 1}2^{2j\alpha} |  \check \phi_j* \check \rho_j*f |  ^2)^\frac{1}{2}\big\| _{p} \lesssim\underset{\substack{j\geq1\\ -2\leq k \leq 2}}{\sup}\|   {\phi}_j (2^{j+k}\cdot)\varphi  \| _{H_2^\sigma}\big\| (\sum_{j\geq 1}|  \check \rho_j*f |  ^2)^\frac{1}{2}\big\| _{p}.$$
The second term $\|   \check \phi_0* \check \rho_0*f  \| _{p}$ is also easy to handle. By the support assumption on $\phi_0 \rho_0$, we have 
$$    \check \phi_0* \check \rho_0*f =\mathcal{F}^{-1} \big(\phi_0 (\varphi^{(0)} +\varphi^{(1)})\big) * \check{\rho}_0 *f.$$
Hence,
\begin{equation*}
\|  \check \phi_0* \check \rho_0*f  \| _{p} \leq \| \mathcal{F}^{-1} \big(\phi_0 (\varphi^{(0)} +\varphi^{(1)})\big)\|_1 \|  \check \rho_0*f  \| _{p} \lesssim \|  \phi_0 (\varphi^{(0)} +\varphi^{(1)})\|_{H_2^\sigma}  \|  \check \rho_0*f  \| _{p}.
\end{equation*}
The assertion is proved.
\end{proof}

\subsection{Hilbert-valued multipliers}

In fact, both theorems above deal with Fourier multipliers acting on Hilbert-valued noncommutative $L_p$ spaces  (the Hilbert space being $\ell_2$). In this subsection titled ``Hilbert-valued multipliers", our target is to extend Theorem \ref{multiplier-global} to the general case where $\ell_2 $ is replaced with more complicated Hilbert spaces. Assume that we have a sequence of Hilbert spaces $H_j$ for every $j\in \mathbb{N}_0$, and denote $H = \oplus_{j=0}^\infty H_j$. Then an element $f\in L_p(\N; H^c)$ has the form $f =(f_j)_{j\geq 0} $ with $f_j \in L_p (\N; H_j^c)$ for every $j$. In this case, it still makes sense to consider the action of the Calder\'on-Zygmund operator $\fk  =(\check{\phi} _j)_{j\geq 0}$.

Since it will be frequently used in the following, we  introduce an  elementary inequality (see \cite[Lemma 4.2]{XXY17}):
\begin{equation}\label{ineq-potential-Sobolev}
\|  fg\| _{H_2^\sigma (\mathbb{R}^{d};\ell_2)}\leq \|  f\| _{H_2^\sigma (\mathbb{R}^{d};\ell_2)} \int _{\mathbb{R}^{d}}(1+|  s | ^2)^{\sigma}| \mathcal{F}^{-1}(g)(s)|  ds,
\end{equation}
where $\sigma>\frac d 2 $, and the functions $f:\mathbb{R}^{d}\rightarrow \ell_2$ and $g:\mathbb{R}^{d}\rightarrow\mathbb{C}$ satisfy
\[
f\in H_2^\sigma (\mathbb{R}^{d};\ell_2) \mbox{ and } \int _{\mathbb{R}^{d}}(1+|  s | ^2)^{\sigma}| \mathcal{F}^{-1}(g)(s)|  ds<\infty.
\]
Here $H_2^\sigma(\R; \ell_2) $ is the $\ell_2$-valued Potential Sobolev space of order $\sigma$.
Note also that $\ell_2$ could be an $\ell_2$-space on an arbitrary index set, depending on the problems in consideration.

The following lemma is an analogue of Lemma 4.3 in \cite{XXY17}. The main difference is that in order to get a Calder\'on-Zygmund operator which is bounded on local Hardy or $\bmo$ spaces, we need to consider the Littlewood-Paley decomposition covering the origin.

\begin{lem}\label{lem:C-Z l_2}
Let $\phi=(\phi_{j})_{j\geq0}$ be a sequence of continuous functions
on $\mathbb{R}^{d}$, viewed as a function from $\mathbb{R}^{d}$
to $\ell_2$. For $\sigma>\frac{d}{2}$, we assume that
\begin{equation}\label{eq:phi}
\|  \phi \| _{2,\sigma}\stackrel{\mathrm{def}}{=}\max \big\{ \sup_{k\geq 1}  \|  \phi(2^{k}\cdot)\varphi   \|  _{H_{2}^{\sigma}(\mathbb{R}^{d};\ell_2)},\;\|  \phi \varphi^{(0)} \|  _{H_{2}^{\sigma}(\mathbb{R}^{d};\ell_2)} \big\}<\infty.
\end{equation}
Let $\fk=(\fk_{j})_{j\geq0}$ with $\fk_{j}=\mathcal{F}^{-1}(\phi_{j})$.
Then $\fk$ is a Calder\'on-Zygmund kernel with values in $\ell_2$, more
precisely,
\begin{enumerate}[\rm(1)]

\item $\|  \widehat{\fk}\|  _{L_{\infty}(\mathbb{R}^{d};\ell_2)}\lesssim \|  \phi \| _{2,\sigma}$;

\item $\int_{  |s |\geq\frac{1}{2}}  \|  \fk(s) \|  _{\ell_2}ds\lesssim\|  \phi \| _{2,\sigma}$;

\item $\sup_{t\in\mathbb{R}^{d}}\int_{  |s |>2  |t |}  \|  \fk(s-t)-\fk(s) \|  _{\ell_2}ds\lesssim \|  \phi \| _{2,\sigma}$.

\end{enumerate}
The relevant constants depend only on $\varphi $, $\sigma$ and $d$.
\end{lem}

\begin{proof}
For any $\xi\in\mathbb{R}^{d}$ and $k\geq1$, by the Cauchy-Schwarz inequality, we have
\begin{equation*}
\begin{split}
  \|  \phi(2^{k}\xi)\varphi  (\xi) \|_{\ell_2}   & = \big \|  \int\mathcal{F}^{-1}(\phi(2^{k}\cdot)\varphi  )(s)e^{-2\pi \rm i s\cdot\xi}ds \big \| _{\ell_2}  \\
&   \leq   \|  \phi(2^{k}\cdot)\varphi \| _{H_{2}^{\sigma}(\mathbb{R}^{d};\ell_{2})}  (\int  (1+  |s |^{2} )^{-\sigma}ds )^{\frac{1}{2}}\lesssim  \|  \phi \|  _{2,\sigma}.
\end{split}
\end{equation*}
In other words,  we have $  \|  \phi\varphi  (2^{-k}\cdot) \| _{L_\infty(\R;\ell_2)}\lesssim \|  \phi \|  _{2,\sigma}$. Likewise, $  \|  \phi\varphi ^{(0)}  \| _{L_\infty(\R;\ell_2)}\lesssim \|  \phi \|  _{2,\sigma}$ also holds. Thus, by \eqref{eq: supp varphi_j} and \eqref{eq:resolution of unity},  we easily deduce
that $\|  \widehat{\fk}\|  _{L_{\infty}(\mathbb{R}^{d};\ell_{2})}\lesssim  \|  \phi \|  _{2,\sigma}$.

To show the third property of $\fk$, we decompose $\phi$ into
\[
\phi=\sum_{k\geq0}\phi\varphi^{(k)}.
\]
 The convergence of the above series can be proved by a limit procedure of its partial sums, which is quite formal.
By \eqref{eq: supp varphi_j} and \eqref{eq:resolution of unity}, we write
\[
\phi\varphi^{(k)}=\phi(\varphi^{(k-1)}+\varphi^{(k)}+\varphi^{(k+1)})\varphi^{(k)}  \stackrel{{\rm def}}{=}\phi_{(k)}\varphi^{(k)},\quad k\geq0.
\]
Here we make the convention that $\varphi^{(k)}=0$ if $k<0$.
Then for $s\in\mathbb{R}^{d}$,
\begin{gather*}
\mathcal{F}^{-1}(\phi  \varphi^{(k)})(s)  =  \mathcal{F}^{-1}(\phi_{(k)})*\mathcal{F}^{-1}(\varphi^{(k)})(s)=2^{kd}\mathcal{F}^{-1}(\phi_{(k)}(2^{k}\cdot))*\mathcal{F}^{-1}(\varphi  )(2^{k}s),\quad k\geq 0.
\end{gather*}
By \eqref{ineq-potential-Sobolev}, we have
\begin{equation*}
  (\int_{\mathbb{R}^{d}}  (1+  |2^{k}s |^{2} )^{\sigma}  \|  \mathcal{F}^{-1}(\phi\varphi^{(k)} )(s) \|_{\ell_{2} }  ^{2}ds )^{\frac{1}{2}} \lesssim  2^{\frac{kd}{2}}   \|  \phi_{(k)}(2^k\cdot) \|  _{H_{2}^{\sigma}(\mathbb{R}^{d};\ell_{2})}.
\end{equation*}
Notice that if $k\geq 1$, we have $\varphi^{(k)}(2^k\cdot)=\varphi$. Thus, if $k\geq 2$,
\begin{equation*}
\begin{split}
\|  \phi_{(k)}(2^k\cdot) \|  _{H_{2}^{\sigma}(\mathbb{R}^{d};\ell_{2})}& \leq    \sum_{j=-1}^{1}  \|  \phi(2^k\cdot)\varphi^{(k-j)}(2^k\cdot) \|  _{H_{2}^{\sigma}(\mathbb{R}^{d};\ell_{2})}\\
& \lesssim   \sum_{j=-1}^{1} \|  \phi(2^{k-j}\cdot)\varphi^{(k-j)}(2^{k-j}\cdot) \|  _{H_{2}^{\sigma}(\mathbb{R}^{d};\ell_{2})}\\
& =  \sum_{j=-1}^{1} \|  \phi(2^{k-j}\cdot)\varphi\|  _{H_{2}^{\sigma}(\mathbb{R}^{d};\ell_{2})} \leq  3 \|\phi\|_{2,\sigma}.
\end{split}
\end{equation*}
For $k=0,1$, we treat $\phi_{(k)}(2^k\cdot) $ in the same way:
\begin{equation*}
\begin{split}
  \|  \phi_{(1)}(2\cdot) \|  _{H_{2}^{\sigma}(\mathbb{R}^{d};\ell_{2})} &\lesssim   \|\phi\varphi^{(0)}\|_{H_{2}^{\sigma}(\mathbb{R}^{d};\ell_{2})}+  \|  \phi(2\cdot)\varphi \|  _{H_{2}^{\sigma}(\mathbb{R}^{d};\ell_{2})}+  \|  \phi(4\cdot)\varphi \|  _{H_{2}^{\sigma}(\mathbb{R}^{d};\ell_{2})}\leq  3 \|\phi\|_{2,\sigma};\\
  \|  \phi_{(0)}\|  _{H_{2}^{\sigma}(\mathbb{R}^{d};\ell_{2})} & \lesssim  \|\phi\varphi^{(0)}\|_{H_{2}^{\sigma}(\mathbb{R}^{d};\ell_{2})}+ \|\phi(2\cdot)\varphi \| _{H_{2}^{\sigma}(\mathbb{R}^{d};\ell_{2})}\leq  3 \|\phi\|_{2,\sigma}.
\end{split}
\end{equation*}
In summary, we obtain
\begin{equation*}
  (\int_{\mathbb{R}^{d}}  (1+  |2^{k}s |^{2} )^{\sigma}  \|  \mathcal{F}^{-1}(\phi\varphi^{(k)} )(s) \|_{\ell_{2} } ^{2}ds )^{\frac{1}{2}}  \lesssim 2^{\frac{kd}{2}}  \|\phi\|_{2,\sigma}.
\end{equation*}
Thus, by the Cauchy-Schwarz inequality, for any $t\in\mathbb{R}^{d}\setminus\{0\}$ and $k\geq0$, we
have
\begin{equation}\label{eq:estimates 1}
\begin{split}
\int_{  |s |>  |t |}  \|  \mathcal{F}^{-1}(\phi\varphi^{(k)})(s) \|_{\ell_{2} }   ds &\lesssim  2^{\frac{kd}{2}}  \|  \phi \|  _{2,\sigma}  (\int_{  |s |>  |t |}(1+  |2^{k}s |^{2})^{-\sigma}ds )^{\frac{1}{2}}\\
 &\lesssim  (2^{k}  |t |)^{\frac{d}{2}-\sigma}  \|  \phi \|  _{2,\sigma}.
\end{split}
\end{equation}
Consequently,
\[
\int_{  |s |>2  |t |}  \|  \mathcal{F}^{-1}(\phi\varphi^{(k)})(s)-\mathcal{F}^{-1}(\phi\varphi^{(k)})(s-t) \|_{\ell_{2} }   ds\lesssim(2^{k}  |t |)^{\frac{d}{2}-\sigma}  \|  \phi \|  _{2,\sigma}.
\]
We notice that $\frac{d}{2}-\sigma<0$, so the estimate above is good
only when $2^{k}  |t |\geq1$. Otherwise, we need another
estimate 
\begin{equation*}
\begin{split}
& \mathcal{F}^{-1}(\phi\varphi^{(k)})(s)-\mathcal{F}^{-1}(\phi\varphi^{(k)})(s-t) \\
 &= \mathcal{F}^{-1}(\phi_{(k)}\varphi^{(k)}(1-e_{t}))(s)\\
  & = 2^{kd}\mathcal{F}^{-1}(\phi_{(k)}(2^{k}\cdot))*  [\mathcal{F}^{-1}(\varphi  )-\mathcal{F}^{-1}(\varphi  )(\cdot-2^{k}t) ](2^{k}s),
\end{split}
\end{equation*}
where $e_{t}(\xi)=e^{2\pi {\rm i}\xi\cdot t}$.
Thus,
\begin{equation*}
\begin{split}
 &     (\int_{\mathbb{R}^{d}}  (1+  |2^{k}s |^{2} )^{\sigma}  \|  \mathcal{F}^{-1}(\phi\varphi^{(k)})(s)-\mathcal{F}^{-1}(\phi\varphi^{(k)})(s-t) \|_{\ell_{2} } ^{2}ds )^{\frac{1}{2}}\\
 & \lesssim  2^{\frac{kd}{2}}  \|  \phi \|  _{2,\sigma}2^{k}  |t |\int  (1+  |s |^{2} )^{\sigma}  |\mathcal{F}^{-1}(\varphi  )(s-\theta 2^{k}t) |ds\\
 & \lesssim  2^{\frac{kd}{2}}  \|  \phi \|  _{2,\sigma}2^{k}  |t |  (\int  |J^{\sigma}  [\varphi  (s)e^{2\pi {\rm i} s\cdot \theta 2^{k}t} ] |^{2}ds )^{\frac{1}{2}}\\
 & \lesssim  2^{\frac{kd}{2}}  \|  \phi \|  _{2,\sigma}2^{k}  |t |,
\end{split}
\end{equation*}
where $\theta \in [0,1]$. Then as before, for $2^{k}  |t |<1$,
we have
\[
\int_{  |s |>2  |t |}  \|  \mathcal{F}^{-1}(\phi\varphi^{(k)})(s)-\mathcal{F}^{-1}(\phi\varphi^{(k)})(s-t) \|_{\ell_2}  ds\lesssim 2^{k}  |t |  \|  \phi \|  _{2,\sigma}.
\]
Combining the previous estimates, we obtain
\begin{equation*}
\begin{split}
& \sup_{t\in\mathbb{R}^{d}}\int_{  |s |>2  |t |}  \|  \fk(s-t)-\fk(s) \|_{\ell _2}   ds \\
& \leq \sup_{t\in\mathbb{R}^{d}}\sum_{k\geq0}\int_{  |s |>2  |t |}  \|  \mathcal{F}^{-1}(\phi\varphi^{(k)})(s)-\mathcal{F}^{-1}(\phi\varphi^{(k)})(s-t) \| _{\ell_2} ds\\
 & \lesssim   \|  \phi \|  _{2,\sigma} \sup_{t\in\mathbb{R}^{d}}\sum_{k\geq0}\min  (2^{k}  |t |,(2^{k}  |t |)^{\frac{d}{2}-\sigma} )\lesssim  \|  \phi \|  _{2,\sigma}.
\end{split}
\end{equation*}

Finally, the second estimate of $\fk$ can  be deduced from \eqref{eq:estimates 1} by letting $|  t | =\frac{1}{2}$:
\begin{equation*}
\begin{split}
\int_{  |s |\geq\frac{1}{2}}  \|  \fk(s) \|_{\ell_2}  ds  & \leq  \sum_{k\geq0}\int_{  |s |\geq\frac{1}{2}}  \|  \mathcal{F}^{-1}(\phi\varphi^{(k)})(s) \|_{\ell_2}  ds\\
&  \leq   \sum_{k\geq0}(2^{k-1})^{\frac{d}{2}-\sigma}  \|  \phi \|  _{2,\sigma}\lesssim  \|  \phi \|  _{2,\sigma}.
\end{split}
\end{equation*}
The proof is complete.
\end{proof}

We keep the notation $H = \oplus_{j=0}^\infty H_j$. By the above lemma, we can apply the (local) Calder\'on-Zygmund theory introduced in section \ref{section-oHardy}, to deduce the following lemma:

\begin{lem}\label{lem:p to p}
Let $1<p<\infty$ and  $\phi=(\phi_j)_{j\geq 0}$ be a sequence of continuous functions on $\mathbb{R}^d$ satisfying \eqref{eq:phi}.  For  any $f=(f_j)_{j\geq 0}\in L_p(\N; H^c)$, we have
\[
\| (\check{\phi}_j * f_j)_{j\geq 0}\| _{L_p(\N; H^c)}\lesssim \|  \phi \|  _{2,\sigma} \| ( f_j)_{j\geq 0}\| _{L_p(\N; H^c)},
\]
where the relevant constant depends only on $\varphi$, $\sigma$, $p$ and $d$.
\end{lem}
\begin{proof}
Consider $\fk$ as a diagonal matrix with diagonal entries $(\fk_j)_{j\geq 0}$ determined by $\widehat{\fk}_j=\phi _j$ and $f=(f_j)_{j\geq 0}$ as a column matrix. The associated Calder\'on-Zygmund operator is defined on $L_p( B(H)\overline{\otimes}\N)$ by 
 \[
\fk(f)(s)=\int_{\R} \fk(s-t)f(t)dt.
\]
Now it suffices to show that $\fk$ is a bounded operator on $L_p(\N;H^c)$.

We claim that $\fk$ is bounded from $L_\infty(\N; H^c)$ into $\bmo(\R,B(H)\overline{\otimes}\M)$. Put $K(s)=\fk (s)\otimes 1_{\M}\in B(H)\overline{\otimes } \M$, for any $s\in \R$. Then we have $\|\fk(s)\|_{\ell_2}\geq \|\fk(s)\|_{\ell_\infty}=\| K(s)\|_{B(H)\overline{\otimes} \M}$ and $\| f\|_{L_\infty(\N; H^c)}=\| f\|_{B(H)\overline{\otimes}\N}$. Thus, the claim is equivalent to saying that $K$ is bounded from $L_\infty(\N; H^c)$ into $\bmo(\R,B(H)\overline{\otimes}\M)$, if we regard $L_\infty(\N;  H^c)$ as a subspace of  $ B(H)\overline \otimes\N$.

First, we show that $K$ is bounded  from $L_\infty(\N;H^c)$ into $\bmo^c(\R,B(H)\overline{\otimes}\M)$.
 Let $Q$ be a cube in $\R$ centered at $c$. We decompose $f$ as $f=g+h$ with $g=f\un_{\widetilde{Q}}$, where $\widetilde{Q}=2Q$ is  the cube which has the same center as $Q$ and twice the side length of $Q$. Set
\[
a=\int_{\R \backslash \widetilde{Q}} K(c-t)f(t)dt.
\]
Then
\[
K(f)(s)-a=K(g)(s)+\int [K(s-t)-K(c-t)]h(t)dt.
\]
Thus, for $Q$ such that $|  Q|  <1$, we have
\[
\frac{1}{|  Q| }\int_Q|  K(f)-a| ^2ds\leq 2(A+B),
\]
where
\begin{equation*}
\begin{split}
A &= \frac{1}{|  Q| }\int_Q|  K(g)| ^2ds,\\
B &= \frac{1}{|  Q| }\int_Q|  \int [K(s-t)-K(c-t)]h(t)dt | ^2ds.
\end{split}
\end{equation*}
The term $A$ is easy to estimate. By Lemma \ref{lem:C-Z l_2} and the Plancherel formula \eqref{eq: Planchel},
\begin{equation*}
\begin{split}
|  Q|  A &  \leq  \int |  \widehat{K}(\xi)\widehat{g}(\xi)| ^2d\xi =\int \widehat{g}(\xi)^*\widehat{K}(\xi)^* \widehat{K}(\xi) \widehat{g}(\xi)d\xi \leq \int \|  \widehat{K}(\xi) \|_{B(H)\overline\otimes \M}  ^2 |  \widehat{g}(\xi)| ^2d\xi\\
&   \lesssim   \int \|  \widehat{k}(\xi) \|_{\ell_2}^2 |  \widehat{g}(\xi)| ^2d\xi\lesssim \|  \phi \| _{2,\sigma}^2\int_{\widetilde{Q}}\,|  f(s)|  ^2ds\\
&      \leq     |  \widetilde{Q} |  \,\|  \phi \| _{2,\sigma}^2 \|  f\|^2
 _{B(H)\overline \otimes \N}=|  \widetilde{Q} |  \,\|  \phi \| _{2,\sigma}^2 \|  f\| _{L_\infty(\N; H^c)}^2,
\end{split}
\end{equation*}
whence
\[
\|  A\| _{B(H)\overline{\otimes} \M}\lesssim  \|  \phi \| _{2,\sigma}^2 \|  f\| _{L_\infty(\N;  H^c)}^2.
\]
To estimate $B$, writing $h=(h_j)_{j\geq 0}$, by Lemma \ref{lem:C-Z l_2}, we get
\begin{equation*}
\begin{split}
& \big |  \int [K(s-t)- K(c-t)]h(t)dt \big | ^2\\
&\lesssim  \int _{\R\backslash\widetilde{Q}}\|  K  (s-t)-K  (c-t)\| _{B(H)\overline\otimes \M}dt \int _{\R\backslash\widetilde{Q}}\|  K  (s-t)-K  (c-t)\| _{B(H)\overline\otimes \M}|  h(t)| ^2dt\\
& \lesssim  \int _{\R\backslash\widetilde{Q}}\|  \fk  (s-t)-\fk  (c-t)\| _{\ell_2}dt \int _{\R\backslash\widetilde{Q}}\|  \fk  (s-t)-\fk  (c-t)\| _{\ell_2} |h(t)| ^2dt\\
& \lesssim  \|  \phi \| _{2,\sigma}^2 \|  f\| _{B(H)\overline \otimes \N}^2 \lesssim \|  \phi \| _{2,\sigma}^2 \|  f\| _{L_\infty(\N; H^c)}^2.
\end{split}
\end{equation*}
Hence,
 \[
\|  B\| _{B(H)\overline{\otimes} \M}\leq \frac{1}{|  Q| }\int _Q \big\|  \int [K  (s-t)-K  (c-t)]h(t)dt \big\| _{B(H)\overline{\otimes} \M}^2ds \lesssim  \|  \phi \| _{2,\sigma}^2 \|  f\| _{L_\infty(\N;H^c)}^2.
\]
Combining the previous inequalities, we deduce that, for any $|  Q| <1$ \[
\Big\| (\frac{1}{|  Q| }\int_Q|  K  (f)-a| ^2ds)^\frac{1}{2} \Big\|_{B(H)\overline{\otimes} \M} \lesssim \|  \phi \| _{2,\sigma} \|  f\| _{L_\infty(\N;H^c)}.
\]
Now we consider the case when $|  Q| =1$. We have
\[
\frac{1}{|  Q| }\int_Q|  K  (f)| ^2ds\leq 2\frac{1}{|  Q| }\int_Q|  K  (g)| ^2ds+2\frac{1}{|  Q| }\int_Q|  K  (h)| ^2ds.
\]
The first term on the right hand side of the above inequality is equal to the term $A$, so it remains  to estimate the second term. When $t\in \R\backslash\widetilde{Q}$, $s\in Q$ and $|  Q| =1$, we have $|  s-t| \geq \frac{1}{2}$. Then by (2) in Lemma \ref{lem:C-Z l_2} and the Cauchy-Schwarz inequality \eqref{eq: 2-1}, we easily deduce that 
\begin{equation*}
\begin{split}
& |  K (h)(s)| ^2= \big|  \int | K  (s-t)h(t)dt \big| ^2\\
& \leq \int _{\R\backslash\widetilde{Q}}\|  K(s-t)\| _{B(H)\overline{\otimes} \M}dt \int _{\R\backslash\widetilde{Q}}\|  K  (s-t)\| _{B(H)\overline{\otimes} \M} | h(t)| ^2dt\\
&\lesssim   \|  f\| _{L_\infty(\N; H^c)}^2(\int _{\R\backslash\widetilde{Q}}\|  \fk  (s-t)\| _{\ell_2}dt)^2\\
&\lesssim   \|  \phi \| _{2,\sigma}^2 \|  f\| _{L_\infty(\N;H^c)}^2.
\end{split}
\end{equation*}
Thus, we have, for any  $|  Q| =1$,
\[
\Big\|  \big(\frac{1}{|  Q| }\int_Q|  K  (f)| ^2 ds\big)^\frac{1}{2}\Big\|  _{B(H)\overline{\otimes} \M}\lesssim \|  \phi \| _{2,\sigma} \|  f\| _{L_\infty(\N; H^c)}.
\]
Therefore,  $K $ is bounded from $L_\infty(\N; H^c)$ into $\bmo^c(\R,B(H)\overline{\otimes}\M)$.

Next we show that  $K $ is bounded from $L_\infty(\N;H^c)$ into $\bmo^r(\R,B(H)\overline{\otimes}\M)$.
We still use the same decomposition $f=g+h$, then we obtain
\[
\frac{1}{|  Q| }\int_Q | [K  (f)-a]^*| ^2ds\leq 2(A'+B'),
\]
where
\begin{equation*}
\begin{split}
A' &= \frac{1}{|  Q| }\int_Q|  K  (g)^*| ^2ds,\\
B' & =  \frac{1}{|  Q| }\int_Q   \big|  \int [(K  (s-t)-K  (c-t))h(t)]^*dt \big| ^2ds.
\end{split}
\end{equation*}
The estimate of $B'$ can be reduced to that of $B$. Indeed,
\begin{equation*}
\begin{split}
\|  B'\| _{B(H)\overline{\otimes} \M} &  \leq \frac{1}{|  Q| }\int _Q \big\|  \int [(K  (s-t)-K  (c-t))h(t)]^*dt \big\| _{B(H)\overline{\otimes} \M} ^2ds \\
& =  \frac{1}{|  Q| }\int _Q \big\|  \int [K  (s-t)-K  (c-t)]h(t)dt \big\| _{B(H)\overline{\otimes} \M} ^2ds \\
 & \lesssim   \|  \phi \| _{2,\sigma}^2 \|  f\| _{L_\infty(\N;H^c)}^2.
\end{split}
\end{equation*}
However, for $A'$, we need a different argument. $A'$ can be viewed as a bounded operator on $H\otimes L_2(\M)$. So
\begin{equation*}
\|  A'\| _{B(\ell_2)\overline{\otimes} \M}
= \sup_b \lbrace \frac{1}{|  Q| }\int_Q \|  \fk(g)(s) \, b\| _{H\otimes L_2(\M)}^2ds \rbrace,
\end{equation*}
where the supremum runs over all $b$ in the unit ball of $H\otimes L_2(\M)$. By the Plancherel formula  \eqref{eq: Planchel}, we have
\begin{equation*}
\begin{split}
 \int_Q \| \fk(g)(s) \, b\| _{H\otimes L_2(\M)}^2ds  & =  \int_Q \langle  \fk(g)(s) \, b, \fk(g)(s) \, b\rangle_{H\otimes L_2(\M)} ds  \\
&\leq  \int \langle \widehat{\fk}(\xi)\widehat{g}(\xi)\, b, \widehat{\fk}(\xi)\widehat{g}(\xi)\, b \rangle _{H\otimes L_2(\M)} d\xi.
\end{split}
\end{equation*}
Let $\text{diag}({f}_j)_j$ be the diagonal matrix in $B(H)\overline\otimes \N$ with entries in $B(H_j) \overline{\otimes}\N$. By the Cauchy-Schwarz inequality, the Plancherel formula  \eqref{eq: Planchel} and Lemma \ref{lem:C-Z l_2}, we continue the estimate above as
\begin{equation*}
\begin{split}
\int \langle \widehat{\fk}(\xi)\widehat{g}(\xi)\, b, \widehat{\fk}(\xi)\widehat{g}(\xi)\, b \rangle _{H\otimes L_2(\M)}  d\xi &  \leq  \sup_\xi \|  \widehat{\fk  }(\xi)\| _{\ell_2}^2 \int \langle  \widehat{g}(\xi)\, b,  \widehat{g}(\xi)\, b \rangle _{H\otimes L_2(\M)} d\xi\\
& \lesssim  \|  \phi \| _{2,\sigma}^2 \int_{\widetilde{Q}} \|  \text{ diag}({f}_j)_j(s) \, b\|^2_{H\otimes L_2(\M)}ds\\
& \lesssim   |Q | \|  \phi \| _{2,\sigma}^2  \| \text{ diag}({f}_j)_j \|^2_{B(H)\overline\otimes  \N}\| b\|^2_{H\otimes L_2(\M)}\\
& \leq   |  Q| \|  \phi \| _{2,\sigma}^2 \|  f\| _{L_\infty(\N;H^c)}^2 ,
\end{split}
\end{equation*}
whence,
 \[
 \|  A'\| _{B(\ell_2)\overline{\otimes} \M} \lesssim \|  \phi \| _{2,\sigma}^2 \|  f\| _{L_\infty(\N;H^c)}^2.
 \]
Following the estimate of $\frac{1}{|  Q| }\int_Q|  K  (f)(s)| ^2 ds$, we get,  when $|  Q| =1$,
  \begin{equation*}
\begin{split}
  \frac{1}{|  Q| }\int_Q|  K  (f)^*| ^2 ds  & \leq 2A'+ 2\frac{1}{|  Q| }\int_Q|  K  (h)^*| ^2 ds\\
&      \leq  2A'+ 2\frac{1}{|  Q| }\int_Q\|   K  (h)^*\| _{B(H)\overline{\otimes}\M}^2 ds\\
&      =  2A'+ 2\frac{1}{|  Q| }\int_Q\|   K  (h)\| _{B(H)\overline{\otimes}\M}^2 ds\\
&     \lesssim   \|  \phi \| _{2,\sigma}^2 \|  f\| _{L_\infty(\N;H^c)}^2.
  \end{split}
\end{equation*}
 Therefore, $K  $ is bounded from $L_\infty(\N;H^c)$ into $\bmo^r(\R,B(H)\overline{\otimes}\M)$. 
 
 In summary, we have proved that $\fk$ is bounded from $L_\infty(\N;H^c)$ into $\bmo(\R,B(H)\overline{\otimes}\M)$.
 It is also clear that $\fk  $ is bounded from $L_2(\N; H^c)$ into $L_2(B(H)\overline{\otimes}\N)$, then by the interpolation in Theorem \ref{interpolation-hardy}, $\fk  $ is bounded from $L_p(\N; H^c)$ into $L_p(B(H)\overline{\otimes}\N)$ for any $2\leq p<\infty$. The case $1<p<2$ is obtained by duality.
\end{proof}

Note that when all $H_j$ degenerate to one dimensional Hilbert space, then $H=\ell_2$, the above lemma gives a sufficient condition for $(\phi_j)_{j\geq 0}$ being a bounded Fourier multiplier on $L_p(\N ; \ell_2^c)$.  So we can also use Lemmas \ref{lem:C-Z l_2} and \ref{lem:p to p} to prove Theorem \ref{multiplier-global} by an argument similar to the proof of  \cite[Theorem~4.1]{XXY17}; details are left to the reader. But here our target is to extend Theorem \ref{multiplier-global} to a more general setting.

\begin{thm}\label{cor: conic multiplier}
Let $p, \alpha, \sigma$, $(\phi_j)_{j\geq 0}$ and $(\rho_j)_{j\geq 0}$ be the same as  in Theorem \ref{multiplier-global}.
Then, for any $f\in \mathcal{S}'(\R;L_1(\M)+\M)$,
\begin{equation*}
\begin{split}
& \Big\| (\sum_{j\geq 0}2^{j(2\alpha +d)}\int_{B(0,2^{-j})} |  \check \phi_j* \check \rho_j*f(\cdot+t) |  ^2dt)^\frac{1}{2} \Big\| _{p}\\
 & \lesssim \max \big\{\underset{\substack{j\geq 1\\ -2\leq k \leq 2}}{\sup} \|    {\phi} _j (2^{j+k}\cdot)\varphi\| _{H_2^\sigma},\| \phi_0 ({\varphi}^{(0)}+{\varphi}^{(1)})\| _{H_2^\sigma}  \big\}\\
&  \;\;\;\; \cdot \Big\| (\sum_{j\geq 0}2^{j(2\alpha+d)} \int_{B(0,2^{-j})} |\check \rho_j*f(\cdot+t) |  ^2dt)^\frac{1}{2} \Big\| _{p},
\end{split}
\end{equation*}
where the constant depends only on $p$, $\sigma$, $d$ and $\varphi$.
\end{thm}

\begin{proof}

Set $H_j=L_2\big(B(0,2^{-j}),2^{jd}dt\big)$ and $H= \oplus_{j=0}^\infty H_j$. So we have
$$ \big\| (\sum_{j\geq 0}2^{j(2\alpha+d)} \int_{B(0,2^{-j})} |\check \phi_j*\check \rho_j*f(\cdot+t) |  ^2dt)^\frac{1}{2}\big\| _{p} = \|(2^{j\alpha} \check \phi_j*\check\rho_j*f(\cdot+\cdot))_j \|_{L_p(\N; H^c)}.$$
Let
\begin{equation*}
\begin{split}
\zeta_j & =\phi_j(\varphi^{(j-1)}+\varphi^{(j)}+\varphi^{(j+1)}),\,\, j\geq 2,\\
\zeta_1&  = \phi_1(\varphi + \varphi^{(1)} +\varphi^{(2)}), \\
\zeta_0 & =\phi_0(\varphi^{(0)}+\varphi )\quad \text{and} \;\; \zeta_j=0 \text{ if }j<0.
\end{split}
\end{equation*}
By the support assumption on $\phi_j\rho_j$, we have that $\phi_j\rho_j=\zeta_j\rho_j$.  So for any $f\in \mathcal{S}'(\R;L_1(\M)+\M)$,
 \[
 \check{\phi}_j*\check{\rho}_j*f=\check{\zeta}_j*\check{\rho}_j*f, \; j\in \mathbb{N}_0.
\]
 Now we show that $\zeta=(\zeta_j)_{j\geq 0}$ satisfies \eqref{eq:phi} with $\zeta$ instead of $\phi$. Indeed, by the support assumption of $\varphi  $, the sequence $\zeta (2^k\cdot)\varphi  =\big(\zeta_j (2^k\cdot)\varphi \big)_{j\geq0}$ has at most five nonzero terms of indices $j$ with $k-2\leq j\leq k+2$. Thus for any $k\in \mathbb{N}_0$,
\[
\|  \zeta (2^k\cdot)\varphi  \| _{H_2^\sigma(\R;\ell_2)}\leq  \sum_{j=k-2}^{k+2}
\|  \zeta_j(2^k\cdot)\varphi  \| _{H_2^\sigma}.
\]
Moreover, by \eqref{ineq-potential-Sobolev}, we have 
\begin{equation*}
\|  \zeta_j(2^k\cdot)\varphi  \| _{H_2^\sigma}\lesssim \|  \phi_j(2^k\cdot) \varphi  \| _{H_2^\sigma},\quad  k-2\leq j\leq k+2.
\end{equation*}
Therefore, the condition \eqref{eq:j condition thm} yields \[
\sup_{k\geq 1} \|  \zeta(2^k\cdot)\varphi  \| _{H_2^\sigma(\R;\ell_2)}\lesssim \underset{\substack{j\geq 1\\ -2\leq k \leq 2}}{\sup} \|    {\phi} _j (2^{j+k}\cdot)\varphi\| _{H_2^\sigma }   +\|  \phi_0 ({\varphi}^{(0)}+{\varphi}^{(1)})\| _{H_2^\sigma}<\infty,
\]
where the relevant constant depends only on $\sigma$, $\varphi$ and $d$. In a similar way, we have
\[
\|  \zeta\varphi^{(0)}\| _{H_2^\sigma(\R;\ell_2)}\leq \sum_{0\leq j\leq 2}\|  \zeta_j\varphi^{(0)}\| _{H_2^\sigma}\lesssim \underset{\substack{j\geq 1\\ -2\leq k \leq 2}}{\sup} \|    {\phi} _j (2^{j+k}\cdot)\varphi\| _{H_2^\sigma  }+\|  \phi_0 ({\varphi}^{(0)}+{\varphi}^{(1)})\| _{H_2^\sigma}<\infty.
\]
Now applying Lemma \ref{lem:p to p} to $f_j=2^{j\alpha} \check\rho_j*f(\cdot+\cdot)$, and $\zeta_j$ instead of $\phi_j$, we conclude the theorem.
\end{proof}

The above theorem will be useful when we consider the conic square function characterizations of local Hardy spaces and inhomogeneous Triebel-Lizorkin spaces in section \ref{section-charact}.

\subsection{Multipliers on $\h_p^c$}

Note that both Theorem \ref{multiplier-global} and Theorem \ref{cor: conic multiplier} do not deal with the case $p=1$. So we include the corresponding Fourier multiplier results for $\h_p^c$ with $1\leq p \leq 2$ in the following. When the Hilbert space $H$ degenerates to $\ell_2$, we have

\begin{lem}\label{lem:p=1}
Let $1\leq p\leq 2$ and $\phi=(\phi_j)_{j\geq0}$ be a sequence of continuous functions on $\R$ satisfying \eqref{eq:phi}. For $f\in \h_p^c(\R,\M)$,
\[
\big\| (\sum_{j\geq 0}| \check{\phi}_j*f |  ^2)^\frac{1}{2}\big\| _{p}\lesssim \|  \phi \|  _{2,\sigma} \|  f\| _{\h_p^c}.
\]
The relevant constant depends only on $\varphi  $, $\sigma$ and $d$.
\end{lem}

\begin{proof}
Now we view $\fk  =(\fk  _j)_{j\geq 0}=(\check{\phi}_j)_{j\geq0}$ as a column matrix and the associated Calder\'on-Zygmund operator $\fk  $ is defined on $L_p(\N)$:
\[
\fk  (f)(s)=\int_{\R}\fk  (s-t)f(t)dt, \quad\forall s\in R.
\]
Thus $\fk  $ maps function with values in $L_p(\M)$ to sequence of functions. Then we have to show that $\fk  $ is bounded from $\h_p^c(\R,\M)$ to $L_p(\N;\ell_2^c)$ for $1\leq p\leq 2$. The case $p=2$ is trivial, so by interpolation, it suffices to consider the case $p=1$. To prove that $\fk  $ is bounded from $\h_1^c(\R,\M)$ to $L_1(\N;\ell_2^c)$,  passing to the dual spaces, it is equal to proving that the adjoint of $\fk $ is bounded from  $L_\infty(\N;\ell_2^c)$ to $\bmo^c(\R,\M)$.  We keep all the notation in the proof of Lemma \ref{lem:p to p}. For any finite sequence $f=(f_j)_{j\geq 0}$ (viewed as a column matrix), the adjoint of $\fk$ is defined by
\[
\fk ^*(f)(s)=\int_{\R}\sum_j{\widetilde{\fk  }}_j(s-t)f_j(t)dt,
\]
where $\widetilde{\fk  }(s)=\fk  (-s)^*$ (so it is a row matrix).  Put $\widetilde{K}(s)=\widetilde{\fk  }(s)\otimes 1_\M$. In this case, $\|\widetilde{K}(f)\|_{\bmo^c(\R,\M)}=\|\widetilde{K}(f)\|_{\bmo^c(\R,B(\ell_2) \overline{\otimes }\M)}$. Then we  apply the estimates used in Lemma \ref{lem:p to p} by replacing $K$ with $\widetilde{K}$. It follows that $\fk  ^*$ is bounded from $L_\infty(\N;\ell_2^c)$ into $\bmo^c(\R,\M)$, so the desired assertion is proved.
\end{proof}

The next theorem is a complement of Theorem \ref{multiplier-global} for the case $p=1$, which relies heavily on the characterization of $\h_1^c(\R,\M)$ given in Theorem \ref{thm: equivalence hpD}.

\begin{thm}\label{lem: Multiplier p=1}
We keep the assumption in Theorem \ref{multiplier-global}. Assume additionally that for any $j\geq 1$, $\rho_j={{\rho}}(2^{-j}\cdot)$ for some Schwartz function $\rho $ with $\supp {{\rho}}\subset \{\xi : 2^{-1}\leq \vert \xi\vert \leq 2\}$ and ${{\rho}}(\xi)>0$ for any $2^{-1}< \vert \xi\vert < 2$, and that $\supp {{\rho}_0}\subset \lbrace\xi :  \vert \xi\vert \leq 2\rbrace$ and ${{\rho}_0}(\xi)>0$ for any $\vert \xi\vert < 2$. Then for $f\in \mathcal{S}'(\R ; L_1(\M)+\M)$, we have
\begin{equation*}
\begin{split}
\big\| (\sum_{j\geq 0}2^{2j\alpha} |  \check \phi_j* \check \rho_j*f |  ^2)^\frac{1}{2} \big\| _{1}  &\lesssim  \max \big\{ \underset{\substack{j\geq 1\\ -2\leq k \leq 2}}{\sup} \|    {\phi} _j (2^{j+k}\cdot)\varphi\| _{H_2^\sigma }, \|    \phi_0 ( {\varphi}^{(0)} + {\varphi}^{(1)})\| _{H_2^\sigma} \big\}\\
& \;\;\;\; \cdot \big\| (\sum_{j\geq 0}2^{2j\alpha} |  \check \rho_j*f |  ^2)^\frac{1}{2} \big\| _{1}.
\end{split}
\end{equation*}

\end{thm}

\begin{proof}
By the assumptions of $\rho$ and $\rho_0$, we can select a Schwartz function $\widetilde{\rho}$ with the same properties as $\rho$ and a Schwartz function $\widetilde{\rho}_0$ satisfying the same conditions as $\rho_0$, such that 
$$
 \sum_{j=1}^\infty \rho (2^{-j}\xi) \overline{\widetilde{\rho}(2^{-j}\xi)}+\rho_0(\xi) \overline{\widetilde{\rho}_0(\xi)}=1,\quad \forall \xi \in \R.
$$
Let  $\Psi_j=(I_{-\alpha}{\rho})(2^{-j}\cdot)$,  $\widetilde\Psi_j=(I_{\alpha}{\rho})(2^{-j}\cdot)$ for $j\geq 1$ and  $\Psi_0=J_{-\alpha}{\rho}_0$,  $\widetilde{\Psi}_0=J_\alpha \rho_0$. We have
$$
 \sum_{j=1}^\infty \Psi _j(\xi) \overline{\widetilde{\Psi}_j(\xi)}+\Psi_0(\xi) \overline{\widetilde{\Psi}_0(\xi)}=1,\quad \forall \, \xi \in \R.
$$
  Applying Theorem \ref{thm: equivalence hpD} (the equivalence $ \|g_{\Phi}^{c, D}(f)\|_{p}+\|\phi *f\|_{p} \approx \|f\|_{\h^c_p}$) to  $g=J^\alpha f$ with the text functions in the above identity, we get
$$\|  g \| _{\h_1^c}\approx \big\|  (\sum_{j\geq 0}|  \check {\Psi}_j*g| ^2)^\frac{1}{2}\big\| _1.$$
Now let us show the following equivalence:
$$
\big\|  (\sum_{j\geq 0}|  \check {\Psi}_j*g| ^2)^\frac{1}{2}\big\| _1\approx \big\|  (\sum_{j\geq 0}2^{2j\alpha}|   \check{\rho}_j*f| ^2)^\frac{1}{2}\big\| _1.
$$
It is easy to see that $ \check {\Psi}_0*g=\check{\rho}_0*f$ and $2^{j\alpha} \check{\rho}_j*f= \check{\Psi}_j*I^{\alpha}f $, so it suffices to prove
\begin{equation}\label{eq: J and I equi}
 \big\|  (\sum_{j\geq 1}|  \check {\Psi}_j*J^{\alpha}f | ^2)^\frac{1}{2}\big\| _1\approx  \big\|  (\sum_{j\geq 1}| \check {\Psi}_j*I^{\alpha}f | ^2)^\frac{1}{2}\big\| _1.
\end{equation}
First, let us consider the case $\alpha \geq 0$. By  \cite[Lemma 3.2.2]{Stein1970}, there exists a finite measure $\mu_\alpha$ on $\R$  such that 
$$
|\xi |^{\alpha}=\widehat{\mu}_{\alpha}(\xi)(1+|\xi|^2)^\frac{\alpha}{2}.
$$
Thus, we have
$$
\check {\Psi}_j*I^{\alpha}f=\mu_{\alpha}*\check {\Psi}_j*J^{\alpha}f, \quad \forall\, j\geq 1.
$$
This implies that
$$
\big\|  (\sum_{j\geq 1}| \check {\Psi}_j*I^{\alpha}f | ^2)^\frac{1}{2}\big\| _1\lesssim  \big\|  (\sum_{j\geq 1}|  \check {\Psi}_j*J^{\alpha}f | ^2)^\frac{1}{2}\big\| _1.
$$
Then, we move to the case $\alpha<0$. Also by \cite[Lemma 3.2.2]{Stein1970}, there exist two finite measures $\nu_{\alpha}$ and $\lambda_{\alpha}$ on $\R$ such that
$$
(1+|\xi|^2)^{-\frac{\alpha}{2}}=\widehat{\nu}_{\alpha}(\xi)+|\xi |^{-\alpha}\widehat{\lambda}_{\alpha}(\xi).
$$
Let $(\dot\varphi_k)_{k \in \mathbb{Z}}$ be the homogeneous resolution of the unit defined in \eqref{eq:resolution of unity homo}. It follows that
$$
\frac{(1+|\xi|^2)^{-\frac{\alpha}{2}}}{|\xi|^{-\alpha}}\sum_{k\geq 0}\dot\varphi_k(\xi)=\frac{\widehat{\nu}_{\alpha}(\xi)}{|\xi|^{-\alpha}}\sum_{k\geq 0}\dot\varphi_k(\xi)+\widehat{\lambda}_{\alpha}(\xi)\sum_{k\geq 0}\dot\varphi_k(\xi).
$$
Thus, by the support assumption of $\widehat{\rho}$, we have
$$
\check {\Psi}_j*I^{\alpha}f=\omega_\alpha*\check {\Psi}_j*J^{\alpha}f,
$$
with
$$\omega_\alpha=\nu_\alpha *\sum_{k\geq 0}\F^{-1}(I_\alpha\dot{\varphi}_k)+\lambda_\alpha*\F^{-1}(\sum_{k\geq 0}\dot\varphi_k).$$
Both $\F^{-1}(\sum_{k\geq 0}\dot\varphi_k)$ and $\sum_{k\geq 0}\F^{-1}(I_\alpha\dot{\varphi}_k)$ are finite measures.
Since $\sum_{k\geq 0}\dot\varphi_k=1-\sum_{k< 0}\dot\varphi_k $, and $\sum_{k< 0}\dot\varphi_k $ is a Schwartz function, we know that $\F^{-1}(\sum_{k\geq 0}\dot\varphi_k) = \delta_0 - \F^{-1}(\sum_{k< 0}\dot\varphi_k)$ is a finite measure, where $\delta_0 $ denotes the Dirac measure at the origin.
Moreover, it is known in \cite[Lemma 3.4]{XXY17} that $\|\F^{-1}(I_\alpha\dot{\varphi}_k)\|_1\lesssim 2^{k \alpha}$. Then we have
$$
\|\F^{-1}(\sum_{k\geq 0} I_\alpha\dot{\varphi}_k)\|_1\lesssim \sum_{k\geq 0} 2^{k \alpha}<\infty.
$$
Therefore, $\omega_\alpha$ is a finite measure on $\R$. Thus,
$$
\big\|  (\sum_{j\geq 1}| \check {\Psi}_j*I^{\alpha}f | ^2)^\frac{1}{2}\big\| _1\lesssim  \big\|  (\sum_{j\geq 1}|  \check {\Psi}_j*J^{\alpha}f | ^2)^\frac{1}{2}\big\| _1.
$$
Similarly, for $\alpha\in \mathbb{R}$, we can prove that
$$
 \big\|  (\sum_{j\geq 1}|  \check {\Psi}_j*J^{\alpha}f | ^2)^\frac{1}{2}\big\| _1\lesssim \big\|  (\sum_{j\geq 1}| \check {\Psi}_j*I^{\alpha}f | ^2)^\frac{1}{2}\big\| _1.
$$
In summary, we have proved \eqref{eq: J and I equi}, which yields that 
$$\|g\|_{\h_1^c} = \|J^\alpha f \|_{\h_1^c}  \approx  \big\|\big(\sum_{j\geq 0}   2^{2j\alpha}  |\check{\rho}_j  *f   |^2\big)^{\frac 1 2} \big\|_1\,. $$

Now define a new sequence $\zeta=(\zeta_j)_{j\geq 0}$ by setting $\zeta_j=2^{j\alpha}I_{-\alpha}\phi_j\rho_j$ for $j\geq 1$ and $\zeta_0=J_{-\alpha}\phi_0\rho_0$. Then
$$
\check{\zeta}_j*g=2^{j\alpha}\check{\phi}_j*\check{\rho}_j*I^{-\alpha}g \quad \text{and}\quad \check{\zeta}_0*g=\check\phi_0*\check{\rho}_0*f.
$$
Repeating the argument for  \eqref{eq: J and I equi} with  $\zeta=(\zeta_j)_{j\geq 0}$ instead of $\Psi=(\Psi_j)_{j\geq 0}$, we get
\begin{equation*}
\big\|  (\sum_{j\geq 0} 2^{2j\alpha}|\check{\phi}_j*\check{\rho}_j*f | ^2)^\frac{1}{2}\big\| _1 \
=  \big\|  (\sum_{j\geq 0}| \check{\zeta}_j*I^{\alpha} f | ^2)^\frac{1}{2}\big\| _1\approx  \big\|  (\sum_{j\geq 0}| \check{\zeta}_j*g | ^2)^\frac{1}{2}\big\| _1.
\end{equation*}
Then, we apply Lemma \ref{lem:p=1} to $g$ with this new $\zeta$ instead of $\phi$ to get
\begin{equation*}
 \big\|  (\sum_{j\geq 0}| \check{\zeta}_j*g | ^2)^\frac{1}{2}\big\| _1
\lesssim \|\zeta\|_{2,\sigma}\|  g\| _{\h_1^c}\approx \|\zeta\|_{2,\sigma} \big\|  (\sum_{j\geq 0}2^{2j\alpha}| \check{\rho}_j*f | ^2)^\frac{1}{2}\big\| _1.
\end{equation*}
It suffices to estimate the term $\|\zeta\|_{2,\sigma}$. By the definition of $\zeta =(\zeta_j) _{j\geq } $, we have
\begin{equation*}
\begin{split} 
\underset{\substack{j\geq 1\\ -2\leq k \leq 2}}{\sup} \|    {\zeta} _j (2^{j+k}\cdot)\varphi\| _{H_2^\sigma } & \lesssim \underset{\substack{j\geq 1\\ -2\leq k \leq 2}}{\sup} \|    {\phi} _j (2^{j+k}\cdot)\varphi\| _{H_2^\sigma },\\
\|    \zeta_0 ( {\varphi}^{(0)} + {\varphi}^{(1)})\| _{H_2^\sigma} &  \lesssim   \|    \phi_0 ( {\varphi}^{(0)} + {\varphi}^{(1)})\| _{H_2^\sigma}. \end{split}
\end{equation*}
So we can use the same argument at the end of the proof of Theorem \ref{cor: conic multiplier}, to get
$$\|\zeta\|_{2,\sigma}\lesssim  \max \big\{ \underset{\substack{j\geq 1\\ -2\leq k \leq 2}}{\sup} \|    {\phi} _j (2^{j+k}\cdot)\varphi\| _{H_2^\sigma }, \|    \phi_0  ({\varphi}^{(0)} + {\varphi}^{(1)})\| _{H_2^\sigma}  \big\}.$$
Combining the above inequalities, we get the desired assertion.
\end{proof}

In the setting where $\ell_2$ is replaced by $H= \oplus_{j= 0}^\infty H_j$ with $H_j=L_2\big(B(0,2^{-j}),2^{jd}dt\big)$, the counterpart of Lemma \ref{lem:p=1} is the following:

\begin{lem}\label{lem:p=1-conic}
Let $\phi=(\phi_j)_{j\geq0}$ be a sequence of continuous functions on $\R$ satisfying \eqref{eq:phi}. Then for $1\leq p\leq 2$ and  $f\in \h_p^c(\R,\M)$,
\[
\big\| (\sum_{j\geq 0}2^{dj} \int_{B(0,2^{-j})} |\check \phi_j*f(\cdot+t) |  ^2dt)^\frac{1}{2} \big\| _{p}\lesssim \|  \phi \|  _{2,\sigma} \|  f\| _{\h_p^c}.
\]
The relevant constant depends only on $\varphi  $, $\sigma$ and $d$.
\end{lem}

\begin{proof}
The proof of this lemma is similar to Lemma \ref{lem:p=1}; let us point out the necessary change. Consider the $H$-valued Calder\'on-Zygmund operator $\fk   $  defined on $L_p(\N)$ given by
\[
\fk  (f)_j(\cdot +t)=\check{\phi}_j  * f(\cdot+t).
\]
The lemma is then reduced to showing that $\fk  $ is bounded from $\h_p^c(\R,\M)$ to $L_p(\N; H^c)$ for $1\leq p<2$. Since each $H_j $ is a normalized Hilbert space, such that the constant function $1$ has Hilbert norm one, the kernel estimates of our $\fk$ here are the same as the ones in Lemma \ref{lem:p to p}. So we can repeat the proof in Lemma \ref{lem:p to p} and Lemma \ref{lem:p=1}. The desired assertion follows.
\end{proof}

Combining the above lemma with Theorem \ref{thm: equivalence hpD} ($\|s_{\Phi}^{c, D}(f)\|_{L_{p}(\N)}+\|\phi *f\|_{p} \approx \|f\|_{\h^c_p}$), we can deduce the analogue of Theorem \ref{lem: Multiplier p=1} in the setting $H= \oplus_{j= 0}^\infty H_j$ with $H_j=L_2\big(B(0,2^{-j}),2^{jd}dt\big)$. Its proof is similar to that of Theorem \ref{lem: Multiplier p=1}, and is left to the reader.

\begin{thm}\label{lem: Multiplier conic p=1}
Keep the assumption in Theorem \ref{cor: conic multiplier} and assume additionally that for any $j\geq 1$, $\rho_j={{\rho}}(2^{-j}\cdot)$ for some Schwartz function with $\supp {{\rho}}\subset \lbrace\xi : 2^{-1}\leq \vert \xi\vert \leq 2\rbrace$ and ${{\rho}}(\xi)>0$ for any $2^{-1}< \vert \xi\vert < 2$, and that $\supp {{\rho}_0}\subset \lbrace\xi :  \vert \xi\vert \leq 2\rbrace$ and ${{\rho}_0}(\xi)>0$ for any $\vert \xi\vert < 2$. Then for any $f\in \mathcal{S}'(\R; L_1(\M)+\M)$,
\begin{equation*}
\begin{split}
& \Big\| (\sum_{j\geq 0}2^{j(2\alpha +d)}\int_{B(0,2^{-j})} |  \check \phi_j* \check \rho_j*f(\cdot+t) |  ^2dt)^\frac{1}{2}\Big\| _{1}\\
&\lesssim  \max \big\{ \underset{\substack{j\geq 1\\ -2\leq k \leq 2}}{\sup} \|    {\phi} _j (2^{j+k}\cdot)\varphi\| _{H_2^\sigma }, \|    \phi_0 ( {\varphi}^{(0)}+\varphi^{(1)})\| _{H_2^\sigma}  \big\} \\
  & \;\;\;\; \cdot  \Big\| (\sum_{j\geq 0}2^{j(2\alpha+d)} \int_{B(0,2^{-j})} |\check \rho_j*f(\cdot+t) |  ^2dt)^\frac{1}{2}\Big\| _{1}.
\end{split}
\end{equation*}

\end{thm}

This theorem fills the gap of $p=1$ left by Theorem \ref{cor: conic multiplier}. Both of them will be useful when we consider the conic square functions of inhomogeneous Triebel-Lizorkin spaces in section \ref{section-charact}.

 \section{Operator-valued Triebel-Lizorkin spaces}\label{section-TL}

In this section, we give the definition of operator-valued Triebel-Lizorkin spaces, and then prove some basic properties of them. Among the others, we connect operator-valued Triebel-Lizorkin spaces with local Hardy spaces introduced in \cite{XX18} via Bessel potentials. By this connection, we are able to deduce  the duality and  the complex interpolation of Triebel-Lizorkin spaces. We also show that for $\alpha>0$ the $F_1^{\alpha,c} (\R,\M)$-norm is the sum of two homogeneous norms.

\subsection{Definitions and basic properties}\label{section: Definitions of T_L}
Recall that  $\varphi$ is a Schwartz function satisfying \eqref{condition-phi}.
For each $j \in \mathbb{N}$,  $\varphi_j$ is the function whose Fourier transform is equal to $\varphi(2^{-j}\cdot)$, and  $\varphi_0$ is the function whose Fourier transform is equal to $1-\sum_{j\geq 1}\varphi (2^{-j}\cdot)$. Moreover, the Fourier transform of $\varphi_j$ is denoted by $\varphi^{(j)}$ for $j\in \mathbb{N}_0$.

\begin{defn}\label{def:T-L R}
Let $1\leq p<\infty$ and $\alpha \in \mathbb{R}$.
\begin{enumerate}[\rm(1)]
\item The column Triebel-Lizorkin space $F_p^{\alpha,c}(\R,\M)$ is defined by
$$
F_p^{\alpha,c}(\R,\M)=\lbrace f\in \mathcal{S}'(\R; L_1(\M)+\M): \|  f\| _{F_p^{\alpha,c}}<\infty\rbrace,
$$
where
$$
\|  f\| _{F_p^{\alpha,c}}= \big\|  (\sum_{j\geq 0}2^{2j\alpha}|  {\varphi}_j*f | ^2)^\frac{1}{2}\big\| _p.
$$
\item The row space $F_p^{\alpha,r}(\R,\M)$ consists of all $f$ such that $f^*\in F_p^{\alpha,c}(\R,\M)$, equipped with the norm $\|  f\| _{F_p^{\alpha,r}}=\|  f^*\| _{F_p^{\alpha,c}}$.
\item The mixture space $F_p^{\alpha}(\R,\M)$ is defined to be
\begin{equation*}
F_p^{\alpha}(\R,\M)=
\begin{cases}
F_p^{\alpha,c}(\R,\M)+F_p^{\alpha,r}(\R,\M) \quad \mbox{if}\quad  1\leq p\leq 2\\
F_p^{\alpha,c}(\R,\M)\cap F_p^{\alpha,r}(\R,\M) \quad \mbox{if}\quad  2<p<\infty,
\end{cases}
\end{equation*}
equipped with
\begin{equation*}
\|  f\| _{F_p^{\alpha}}=
\begin{cases}
\inf \lbrace \|  g\| _{F_p^{\alpha,c}}+\|  h\| _{F_p^{\alpha,r}}: f=g+h \rbrace  \quad \mbox{if}\quad  1\leq p\leq 2\\
\max \lbrace \|  f\| _{F_p^{\alpha,c}},\|  f\| _{F_p^{\alpha,r}} \rbrace  \quad \mbox{if}\quad  2< p<\infty.\\
\end{cases}
\end{equation*}
\end{enumerate}
\end{defn}

In the sequel, we focus on the study of the column Triebel-Lizorkin spaces. All results obtained in the rest of this paper also admit the row versions.
The following proposition shows that $F_p^{\alpha,c}$-norm is independent of the choice of the function $\varphi$ satisfying \eqref{condition-phi}.

\begin{prop}\label{prop: charac of F}
Let $\psi $ be another Schwartz function satisfying the same condition \eqref{condition-phi} as $\varphi$. For each $j \in \mathbb{N}$,  let $\psi_j$ be the function whose Fourier transform is equal to $\psi(2^{-j}\cdot)$, and let $\psi_0$ be the function whose Fourier transform is equal to $1-\sum_{j\geq 1}\psi (2^{-j}\cdot)$. Then
$$
\|  f\| _{F_p^{\alpha,c}} \approx \big\|  (\sum_{j\geq 0}2^{2j\alpha}|  {\psi}_j*f | ^2)^\frac{1}{2}\big\| _p.
$$
\end{prop}

\begin{proof}
For any $f\in \mathcal{S}'(\R; L_1(\M)+\M)$, by the support assumption of $\psi$ and $\varphi$, we have, for any $j\geq 0$,
$$
{\psi}_j*f=\sum_{k=-1}^1{\psi}_j*{\varphi}_{k+j}*f,
$$
with the convention ${\varphi}_{-1}=0$. Thus by Theorems \ref{multiplier-global} and \ref{lem: Multiplier p=1},
\begin{equation*}
\begin{split}
&\big\|  (\sum_{j\geq 0}2^{2j\alpha}|  {\psi}_j*f | ^2)^\frac{1}{2}\big\| _p \\
 & \leq \sum_{k=-1}^{1}\big\|  (\sum_{j\geq 0}2^{2j\alpha}|  {\psi}_j*{\varphi}_{k+j}*f | ^2)^\frac{1}{2}\big\| _p\\
 & \lesssim  \max \big\{ \sup_{-2\leq k \leq 2} \|    {\psi} (2^{k}\cdot)\varphi\| _{H_2^\sigma }, \|    \psi_0 ({\varphi}^{(0)} + {\varphi}^{(1)})\| _{H_2^\sigma}  \big\}
 \big\|  (\sum_{j\geq 0}2^{2j\alpha}|  {\varphi}_j*f | ^2)^\frac{1}{2}\big\| _p \\
 &  \lesssim  \big\|  (\sum_{j\geq 0}2^{2j\alpha}|  {\varphi}_j*f | ^2)^\frac{1}{2}\big\| _p.
\end{split}
\end{equation*}
Changing the role of $\varphi$ and $\psi$, we get the reverse inequality.
\end{proof}

\begin{prop}\label{prop: F}
Let $1\leq p <\infty$ and $\alpha \in \mathbb{R}$. Then
\begin{enumerate}[\rm(1)]
\item $F_p^{\alpha,c}(\R,\M)$ is a Banach space.
\item   $F_p^{\alpha,c}(\R,\M)\subset F_p^{\beta,c}(\R,\M)$ if $\alpha > \beta$.
\item $F_p^{0,c}(\R,\M)=\h _p^c(\R,\M)$ with equivalent norms.
\end{enumerate}
\end{prop}

\begin{proof}
(1) Let $\{f_i\}$ be a Cauchy sequence in $F_p^{\alpha,c}(\R,\M)$. Then, the sequence $\{a_i\}$ with $a_i=({\varphi}_0*f_i, \ldots ,2^{j\alpha}{\varphi}_j*f_i, \ldots)$ is also a Cauchy sequence in $L_p(\N;\ell_2^c(\mathbb{N}_0))$. Thus, $a_i$ converges to a function $f=(f^0,\ldots, f^j,\ldots)$ in  $L_p(\N;\ell_2^c(\mathbb{N}_0))$. Formally we take
\begin{equation}\label{eq:f=sum fk}
f=\sum_{j\geq 0}f^j.
\end{equation}
 Since for each $j\in \mathbb{N}$, $\supp \widehat{f^j}\subset\{\xi:  2^{j-1}\leq |  \xi | \leq 2^{j+1}\}$ and $\text{supp } \widehat{f^0} \subset \{\xi:  |  \xi | \leq 2\}$, the series \eqref{eq:f=sum fk} converges in $\mathcal{S}'(\R;L_p(\M))$. Let ${\varphi}_j=0$ if $j<0$. By the support assumption of $\varphi$, when $i\rightarrow \infty$, we get
 $$
{\varphi}_j*f_i=\sum_{k=j-1}^{j+1}{\varphi}_k*{\varphi}_j*f_i\rightarrow \sum_{k=j-1}^{j+1}{\varphi}_j*f^k={\varphi}_j*f,
 $$
which implies that $f^j=2^{j\alpha}{\varphi}_j*f$, for any $j\geq 0$. Thus, $f\in F_p^{\alpha,c}(\R,\M)$ and $\{f_i\}$ converges to $f$ in $F_p^{\alpha,c}(\R,\M)$.

(2) is obvious.

(3) It is easy to see that any $\varphi$ satisfying \eqref{condition-phi} also satisfies \eqref{eq: reproduceD}. Then by the discrete characterization of $\h _p^c(\R,\M)$ given in Theorem \ref{thm: equivalence hpD}, we get the desired assertion.
\end{proof}

Given $a\in\mathbb{R}_+$, we define $D_{i, a}(\xi)=(2\pi {\rm i}\xi_i)^a$ for $\xi\in\mathbb{R}^d$, and $D_i^a$ to be the  Fourier multiplier with symbol $D_{i, a}(\xi)$ on Triebel-Lizorkin spaces $F_p^{\alpha ,c}(\R,\M)$. We set  $D_a=D_{1,a_1}\cdots D_{d, a_d}$ and $D^a=D_1^{a_1}\cdots D_d^{a_d}$ for any $a=(a_1,\cdots, a_d)\in\R_+$. Note that if $a$ is a positive integer, $D_i^a=\partial_i^a$, so there does not exist any conflict of notation. The operator $D^a$ can be viewed as a  fractional extension of partial derivatives. The following is the so-called lifting (or reduction) property of Triebel-Lizorkin spaces.

\begin{prop}\label{prop: lifting}
Let $1\leq p<\infty$ and $\alpha\in \mathbb{R}$.
\begin{enumerate}[\rm(1)]
\item For any $\beta\in \mathbb{R}$, $J^\beta$ is an isomorphism between $F_p^{\alpha,c}(\R, \M)$ and $F_p^{\alpha-\beta,c}(\R, \M)$. In particular, $J^\alpha$ is an isomorphism between $F_p^{\alpha,c}(\R, \M)$ and $\h_p^{c}(\R, \M)$.
\item Let $\beta >0$. Then $f \in F_p^{\alpha,c}(\R, \M)$  if and only if $\varphi_0*f\in L_p(\N)$ and $D_i^\beta f\in F_p^{\alpha-\beta,c}(\R, \M)$ for all $i=1,\ldots, d$. Moreover, in this case,
$$
\|f\|_{ F_p^{\alpha,c}}\approx \|\varphi_0*f \|_p+\sum_{i=1}^d \|D_i^\beta f \|_{ F_p^{\alpha-\beta,c}}.
$$
\end{enumerate}

\end{prop}
\begin{proof}
(1) Let $f\in F_p^{\alpha,c}(\R, \M)$. Applying Theorems \ref{multiplier-global} and \ref{lem: Multiplier p=1} with $\rho = \varphi $, we obtain
\begin{equation}\label{eq: lifting J}
\begin{split}
\|  J^\beta f \| _{F_p^{\alpha-\beta,c}}&=  \big\|  (\sum_{j\geq 0}2^{2j(\alpha-\beta)}|  {\varphi}_j*J^\beta f | ^2)^\frac{1}{2}\big\| _p \\
&\lesssim    \max\{  \underset{\substack{j\geq 1\\ -2\leq k \leq 2}}{\sup}2^{-j\beta}\|  J_\beta(2^{j+k}\cdot){\varphi}\| _{H_2^\sigma}, \|J_\beta (\varphi^{(0)}+\varphi^{(1)})\|_{H_2^\sigma}\}\\
&\;\;\;\;\cdot\big\|  (\sum_{j\geq 0}2^{2j\alpha}|  {\varphi}_j*f | ^2)^\frac{1}{2}\big\| _p.
\end{split}
\end{equation}
It is easy to check that all partial derivatives of $2^{-j\beta}J_\beta(2^{j+k}\cdot){\varphi}$ of order less than or equal to $[\sigma]+1$ are bounded uniformly in $j\geq 1$ and $-2\leq k\leq 2$, and that $J_\beta(\varphi^{(0)}+\varphi^{(1)})\in H_2^\sigma(\R)$. Thus $\|  J^\beta f \| _{F_p^{\alpha-\beta,c}}\lesssim \|   f \| _{F_p^{\alpha,c}}$. So $J^\beta$ is continuous from $F_p^{\alpha,c}(\R, \M)$ to $F_p^{\alpha-\beta,c}(\R, \M)$, and its inverse $J^{-\beta}$ is also continuous from  $F_p^{\alpha-\beta,c}(\R, \M)$ to $F_p^{\alpha,c}(\R, \M)$.

(2) If we take $\sigma\in(\frac{d}{2},\beta+\frac{d}{2})$, then we have $\|D_{i,\beta} \0\|_{H_2^\sigma}<\infty$ and $\|D_{i,\beta}{\varphi}\|_{H_2^\sigma}<\infty$. Replacing $J^\beta$ by $D_i^\beta$ in \eqref{eq: lifting J}, we obtain  that, for any $i=1,\ldots, d$,
$$
\|D_i^\beta f \|_{ F_p^{\alpha-\beta,c}}\lesssim \| f \|_{ F_p^{\alpha,c}},
$$
which implies immediately that
$$
 \|\varphi_0*f \|_p+\sum_{i=1}^d \|D_i^\beta f \|_{ F_p^{\alpha-\beta,c}}\lesssim \|f\|_{ F_p^{\alpha,c}}.
$$
To show the reverse inequality, we choose a nonnegative infinitely differentiable function  $\chi$ on $\mathbb{R}$ such that $\chi(s)=0$ if $|s|<\frac{1}{2\sqrt{d}}$ and $\chi(s)=1$ if $|s|\geq \frac{1}{\sqrt{d}}$. For $i=1,\ldots, d$, we define $\chi_i$ on $\R$ as follows:
$$
\chi_i(\xi)=\frac{1}{\chi(\xi_1)|\xi_1|^\beta+\ldots+\chi(\xi_d)|\xi_d|^\beta}\frac{\chi(\xi_i)|\xi_i|^\beta}{(2\pi {\rm i}\xi_i)^\beta},
$$
whenever the first denominator is positive, say, when $|\xi |\geq 1$. Then for any $j\geq 1$, $\chi_i\varphi_j$ is a well-defined infinitely differentiable function on $\R\backslash \{\xi: \xi_i=0 \}$ and
$$
\varphi^{(j)}=\sum_{i=1}^d \chi_i D_{i,\beta}\varphi^{(j)}.
$$
Then by Theorem \ref{multiplier-homogeneous},  we have
\begin{equation*}
\begin{split}
\|f\|_{F_p^{\alpha,c}} &\leq   \|\varphi_0*f \|_p+ \sum_{i=1}^d \big\|\big(\sum_{j\geq 1}2^{2j\alpha} |\check {\chi}_i* \varphi_j* D_i^\beta f |^2\big)^\frac{1}{2}\big\|_p\\
&  \lesssim   \|\varphi_0*f \|_{p}+ \sum_{i=1}^d \underset{\substack{j\geq 1\\ -2\leq k \leq 2}}{\sup} 2^{j\beta}\|\chi_i(2^{j+k}\cdot)\varphi \|_{H_2^\sigma}  \big\|\big(\sum_{j\geq 1}2^{2j(\alpha-\beta)} |{\varphi}_j* D_i^\beta f |^2\big)^\frac{1}{2}\big\|_p\,.
\end{split}
\end{equation*}
However,
$$
2^{j\beta}\|\chi_i(2^{j+k}\cdot)\varphi \|_{H_2^\sigma(\R)}=\|\phi_i (2^k\cdot)\varphi\|_{H_2^\sigma(\R)},
$$
where
$$
\phi_i(\xi)=\frac{1}{\chi(2^j\xi_1)|\xi_1|^\beta+\ldots+\chi(2^j\xi_d)|\xi_d|^\beta}\frac{\chi(2^j\xi_i)|\xi_i|^\beta}{(2\pi {\rm i}\xi_i)^\beta}.
$$
Since all partial derivatives of $\phi_i\varphi(2^k \cdot )$, of order less than a fixed integer, are bounded uniformly in $j, k$ and $i$, and the norm of $\phi_i\varphi(2^k \cdot)$ in $H_2^\sigma(\R)$ are bounded from above by a constant independent of $j, k$ and $i$. Then we deduce
\begin{equation*}
\begin{split}
\|f\|_{F_p^{\alpha,c}} &\lesssim  \|\varphi_0*f \|_p+\sum_{i=1}^d\big\|\big(\sum_{j\geq 1}2^{2j(\alpha-\beta)} | {\varphi}_j* D_i^\beta f |^2\big)^\frac{1}{2}\big\|_p\\
&\leq  \|\varphi_0*f \|_{p }+\sum_{i=1}^d \|D_i^\beta f \|_{ F_p^{\alpha-\beta,c}}.
\end{split}
\end{equation*}
The assertion is proved.
\end{proof}

\begin{defn}
For $\alpha\in \mathbb{R}$, we define $F^{\alpha,c}_\infty(\R,\M)$ as the space of all $f\in \mathcal{S}'(\R; \M)$  such that 
$$
\| \varphi_0*f\|_\N+\sup_{|Q|<1} \Big\|  \frac{1}{|  Q| }\int_Q\sum_{j\geq -\log_2(l(Q))}2^{2j\alpha}|  \varphi_j*f(s) | ^2ds \Big\| _\M^\frac{1}{2}<\infty.
$$
We endow the space $F^{\alpha,c}_\infty(\R,\M)$ with the norm:
$$
\|  f\|  _{F^{\alpha,c}_\infty}=\| \varphi_0*f\|_\N+\sup_{|Q|<1} \Big\|  \frac{1}{|  Q| }\int_Q\sum_{j\geq -\log_2(l(Q))}2^{2j\alpha}|  \varphi_j*f (s)| ^2ds \Big\| _\M^\frac{1}{2}.
$$
\end{defn}

\begin{prop}
Let $1\leq p <\infty$, $\alpha\in \mathbb{R}$ and $q$ be the conjugate index of $p$. Then the dual space of $F_p^{\alpha,c}(\R,\M)$ coincides isomorphically with $F_q^{-\alpha,c}(\R,\M)$.
\end{prop}

\begin{proof}
First, we show that $J^\alpha$ is an isomorphism between $F_\infty^{\alpha,c}(\R,\M)$ and $\bmo ^c(\R,\M)$. To this end, we use the discrete Carleson characterization of $\bmo ^c(\R,\M)$ in \cite[Corollary~5.13]{XX18}:
\begin{equation}\label{eq: discrete Carleson}
\|  f\|  _{\bmo ^c}\approx\| \phi*f\|_\N+\sup_{|Q|<1}  \Big\|  \frac{1}{|  Q| }\int_Q\sum_{j\geq -\log_2(l(Q))} |\Phi_j*f (s)| ^2ds\Big\| _\M^\frac{1}{2},
\end{equation}
where $\Phi\in \mathcal{S}(\R)$ and $\phi \in H_2^\sigma(\R)$ satisfying \eqref{eq: reproduceD 2}. By taking $\phi=\varphi_0$ and $\Phi=J^{-\alpha} \varphi$, we apply  \eqref{eq: discrete Carleson} to $J^\alpha f$:
\begin{equation*}
\begin{split}
\| J^\alpha  f\|  _{\bmo ^c}   &\approx \| \varphi_0*f\|_\N+\sup_{|Q|<1}   \Big\|  \frac{1}{|  Q| }\int_Q\sum_{j\geq -\log_2(l(Q))} |(J^{-\alpha} \varphi)_j*(J^\alpha f) (s)| ^2ds \Big\| _\M^\frac{1}{2}\\
&= \| \varphi_0*f\|_\N+\sup_{|Q|<1}  \Big\| \frac{1}{|  Q| }\int_Q\sum_{j\geq -\log_2(l(Q))}2^{2j\alpha}|  \varphi_j*f (s)| ^2ds \Big\|  _\M^\frac{1}{2}\\
&= \|  f\|  _{F^{\alpha,c}_\infty}.
\end{split}
\end{equation*}
Since $J^\alpha$ is also an isomorphism between $F_p^{\alpha,c}(\R, \M)$ and $\h_p^{c}(\R, \M)$ for any $1<p<\infty$, by the $\h_1^c$-$\bmo^c$ duality and the $\h_p^c$-$\h_q^c$ duality in Theorem  \ref{dual-hardy}, we see that  $F_p^{\alpha,c}(\R,\M)^*=F_q^{-\alpha,c}(\R,\M)$ with equivalent norms.
\end{proof}

\subsection{Interpolation}

Now we indicate a complex interpolation result of Triebel-Lizorkin spaces. It is deduced from the interpolation of local Hardy and bmo spaces in Theorem \ref{interpolation-hardy}, and the boundedness of complex order Bessel potentials on them.

\begin{prop}\label{complex-inter-1&infty}
Let $\alpha_0,\alpha_1\in \mathbb{R}$ and $1<p<\infty$. Then
$$
\big( F_\infty^{\alpha_0,c}(\R,\M), F_1^{\alpha_1,c}(\R,\M)\big)_{\frac{1}{p}}=F_p^{\alpha,c}(\R,\M), \quad \alpha=(1-\frac{1}{p})\alpha_0+\frac{\alpha_1}{p}.
$$
\end{prop}

\begin{proof}
Let $f\in F_p^{\alpha,c}(\R,\M)$. By Proposition \ref{prop: lifting}, we have $J^\alpha f\in \h _p^c(\R,\M)$. Therefore, according to Theorem \ref{interpolation-hardy} (1), there exists a continuous function on the strip $\{z\in \mathbb{C}: 0\leq \text{Re}{z}\leq 1\}$, analytic in the interior, such that $J^\alpha f=F(\frac{1}{p})\in \h _p^c(\R,\M)$ and 
$$
\sup_{t\in \mathbb{R}}\| F({\rm i}t) \|_{{\bmo}^c}<\infty \quad \text{and}\quad \sup_{t\in \mathbb{R}}\| F(1+{\rm i}t) \|_{{\h}_1^c}<\infty.
$$
We consider Bessel potentials of complex order. For $z\in\mathbb{C}$, define $J_z(\xi)=(1+|\xi |^2)^\frac{z}{2}$, and $J^z$ to be the associated Fourier multiplier. We set
$$
\widetilde{F}(z)=e^{(z-\frac{1}{p})^2}J^{-(1-z)\alpha_0-z\alpha_1}F(z).
$$
For any $t\in \mathbb{R}$,
$$
\| \widetilde{F}({\rm i}t) \|_{F_\infty^{\alpha_0,c}}\approx e^{-t^2+\frac{1}{p^2}}\| J^{{\rm i}t(\alpha_0-\alpha_1)}F({\rm i}t)\|_{{\bmo}^c}
$$
and
$$
\| \widetilde{F}(1+{\rm i}t) \|_{F_1^{\alpha_1,c}}\approx e^{-t^2+(1-\frac{1}{p})^2} \| J^{{\rm i}t(\alpha_0-\alpha_1)}F(1+{\rm i}t)\|_{{\h}_1^c}.
$$
We claim that $J^{{\rm i}t}$ is a bounded Fourier multiplier on ${\h}_1^c(\R,\M)$, so by duality, it is bounded  on ${\bmo}^c(\R,\M)$ too.  
Therefore, we will have
$$
\sup_{t\in \mathbb{R}}\| \widetilde{F}({\rm i}t) \|_{F_\infty^{\alpha_0,c}}<\infty \quad \text{and}\quad \sup_{t\in \mathbb{R}}\| \widetilde{F}(1+{\rm i}t) \|_{F_1^{\alpha_1,c}}<\infty.
$$
This will imply that $f=\widetilde{F}(\frac{1}{p})\in \big( F_\infty^{\alpha_0,c}(\R,\M), F_1^{\alpha_1,c}(\R,\M)\big)_{\frac{1}{p}}$.
Hence,
$$
F_p^{\alpha,c}(\R,\M)\subset  \big( F_\infty^{\alpha_0,c}(\R,\M), F_1^{\alpha_1,c}(\R,\M)\big)_{\frac{1}{p}}.
$$
By duality, we will get the reverse inclusion for the Calder\'on's second interpolation $(\cdot, \cdot)^{\frac 1 p }$. Then by the inclusion $(\cdot, \cdot)_{\frac 1 p }\subset (\cdot, \cdot)^{\frac 1 p }$ between two kinds of complex interpolations (see \cite[Theorem~4.3.1]{BL1976}), we will obtain the desired assertion.

Now, we prove the claim. First, we easily check that $J_{{\rm i}t}$ is $d$-times differentiable on $\R\setminus \{0\}$, and for any $m\in \mathbb{N}_0^d$ and $| m|_1\leq d$, we have
$$
\sup \big\{ |\xi |^ {|m|_1} |D^m J_{{\rm i}t}(\xi)|: \xi\neq 0 \big\}\lesssim (1+|t|)^d.
$$
Next, we check that (with $J_{{\rm i}t}(2^{k}\xi)=(1+|2^{k}\xi|^2)^{\frac{{\rm i}t}{2}}$),
$$
\underset{-2\leq k\leq 2}{\max} \|J_{{\rm i}t}(2^{k}\cdot)\varphi \|_{H_2^d}\lesssim (1+|t|)^d \quad \text{and}\quad
\| J_{{\rm i}t} ({\varphi}^{(0)}+\varphi^{(1)})\| _{H_2^d} \lesssim (1+|t|)^d.
$$
By (3) in Proposition \ref{prop: F}, if we take $(\varphi_j)_{j\geq 0}$ to be the Littlewood-Paley decomposition on $\R$ satisfying \eqref{eq: supp varphi_j} and \eqref{eq:resolution of unity}, we have
$$
\|J^{{\rm i}t} f \|_{\h_1^c} \approx \big\| (\sum_{j\geq 0} |\check{J_{{\rm i}t}}* {\varphi}_j*f | ^2)^\frac{1}{2}\big\| _1
$$
and 
$$
\| f \|_{\h_1^c} \approx \big\|  (\sum_{j\geq 0} | {\varphi}_j*f | ^2)^\frac{1}{2}\big\| _1.
$$
Then, we apply Theorem  \ref{lem: Multiplier p=1} with $\rho_j=\varphi_j$, $\phi_j(2^j\cdot)=\check{J_{{\rm i}t}}$, and $\alpha=0$, $\sigma=d$, 
\[
\| J^{{\rm {i}}t}f \|_{\h_1^c} \lesssim \max \big\{ \max_{-2\leq k\leq 2} \| J_{{\rm i}t} (2^k\cdot)\varphi\| _{H_2^d}, \| J_{{\rm i}t}( {\varphi}^{(0)}+ {\varphi}^{(1)})\| _{H_2^d}  \big\}  \| f \|_{\h_1^c}\lesssim (1+|t|)^d  \| f \|_{\h_1^c}.
\]
The claim is proved.
\end{proof}

\begin{rmk}
The real interpolation of the couple $\big( F_\infty^{\alpha,c}(\R,\M), F_1^{\alpha,c}(\R,\M)\big)$ follows easily from that of Hardy spaces (see Theorem \ref{interpolation-hardy}) and Proposition \ref{prop: lifting}. But if $\alpha_1 \neq \alpha_2$, the real interpolation of $\big( F_\infty^{\alpha_1,c}(\R,\M), F_1^{\alpha_2,c}(\R,\M)\big)$ will give Besov type spaces. We will not consider this problem in this paper, and refer the reader to \cite{XXY17} for similar results on homogeneous Triebel-Lizorkin (and Besov) spaces.
\end{rmk}

\subsection{Triebel-Lizorkin spaces with $\alpha>0$}

The following result shows that when $\alpha>0$, the $F_{1}^{\alpha,c}(\R,\M)$-norm can be rewritten as the sum of two homogeneous norms. Recall that for a fixed Schwartz function $\varphi$ in \eqref{condition-phi}, the functions $\dot{\varphi}_j$'s determined by $\widehat{\dot{\varphi}}_j(\xi)=\varphi (2^{-j}\xi)$, $j\in\mathbb{Z}$ give a homogeneous Littlewood-Paley decomposition on $\R$ satisfying \eqref{eq:resolution of unity homo}.

\begin{prop}\label{thm:homo}
Let $\alpha>0$. If $1\leq p<\infty$, then
\begin{equation*}
\| f\|_{F_{p}^{\alpha,c}}\approx \|  {\varphi_0}*f\| _{p}+\big\|  (\sum_{j=-\infty}^{+\infty}2^{2j\alpha}|  \dot{\varphi}_j*f| ^2)^\frac{1}{2}\big\| _{p}\, , \quad \forall\, f\in F_{p}^{\alpha,c}(\R,\M).
\end{equation*}
If $1\leq p\leq 2$,
\begin{equation*}
\| f\|_{F_{p}^{\alpha,c}}\approx\|  f\| _{p}+ \big\| (\sum_{j=-\infty}^{+\infty}2^{2j\alpha}|  \dot{\varphi}_j*f| ^2)^\frac{1}{2}\big\| _{p}\, , \quad \forall\, f\in F_{p}^{\alpha,c}(\R,\M).
\end{equation*}
\end{prop}

\begin{proof}
Firstly, we prove the first equivalence. By the definition of the $F_{p}^{\alpha,c}$-norm, it is obvious that
$$
\|f\|_{F_{p}^{\alpha,c}}\lesssim \|  {\varphi_0}*f\| _{p}+\big\|  (\sum_{j=-\infty}^{+\infty}2^{2j\alpha}|  \dot{\varphi}_j*f| ^2)^\frac{1}{2}\big\| _{p}.
$$
 To prove the reverse inequality,  it suffices to show:
$$
\big\|  (\sum_{j=-\infty}^{0}2^{2j\alpha}|  \dot{\varphi}_j*f| ^2)^\frac{1}{2}\big\| _p\lesssim \|  {\varphi_0}*f\| _p+\big\|  (\sum_{j=1}^{+\infty}2^{2j\alpha}|  \dot{\varphi}_j*f| ^2)^\frac{1}{2}\big\| _{p}.
$$
By the support assumption of $\varphi $, we have $\varphi^{(0)}=1$ for any $| \xi |\leq 1$. Thus, when $j< 0$,
$$
{\varphi}(2^{j}\cdot)={\varphi}(2^{j}\cdot)\varphi^{(0)}.
$$
Then
\begin{equation}\label{eq: <0}
{\dot\varphi}_j*f={\dot\varphi}_j*{\varphi_0}*f.
\end{equation}
By the triangle inequality, \eqref{eq: <0} and  \cite[Lemma 1.7]{XXY2015}, we obtain
\begin{equation*}
\begin{split}
\big\|  (\sum_{j=-\infty}^{0}2^{j\alpha}|  {\dot\varphi}_j*f | ^2)^\frac{1}{2}\big\| _p &\lesssim \sum_{j=-\infty}^{-1}2^{j\alpha}\|  {\dot\varphi}_j*\varphi_0*f \| _p +\|{\dot\varphi}_0*f\|_p \\
&\lesssim \sum_{j=-\infty}^{-1}2^{j\alpha}\|  {\dot\varphi}_j\| _1\|  {\varphi_0}*f\| _p+ \big\| \varphi (\varphi_0+\varphi_1+\varphi_2)*f \big\| _{p}\\
& \lesssim \sum_{j=-\infty}^{0}2^{j\alpha}\|  {\varphi_0}*f\| _p +\big\|  (\sum_{j=1}^{+\infty}2^{2j\alpha}|  {\varphi}_j*f| ^2)^\frac{1}{2}\big\| _{p}\\
& \lesssim  \|  {\varphi_0}*f\| _p +\big\|  (\sum_{j=1}^{+\infty}2^{2j\alpha}| {\varphi}_j*f| ^2)^\frac{1}{2}\big\| _{p}.
\end{split}
\end{equation*}
Therefore, we have proved that $ \|  {\varphi_0}*f\| _p +\big\|  (\sum_{j=1}^{+\infty}2^{2j\alpha}| {\varphi}_j*f| ^2)^\frac{1}{2}\big\| _{p}$ gives rise to an equivalent norm on $F_{p}^{\alpha,c}(\R,\M)$ when $\alpha>0$.

Now let us deal with the second equivalence. For any $1\leq p\leq 2$ and $\alpha>0$, we have $F_{p}^{\alpha,c}(\R,\M)\subset \h_p^c(\R,\M)\subset L_p(\N)$. Therefore $\|  f\| _p \lesssim \|  f\| _{F_{p}^{\alpha,c}}$. Combined with the equivalence obtained above, we see that 
$$\|  f\| _{p}+ \big\| (\sum_{j=-\infty}^{+\infty}2^{2j\alpha}|  \dot{\varphi}_j*f| ^2)^\frac{1}{2}\big\| _{p}\lesssim \|  f\| _{F_{p}^{\alpha,c}}.$$
The reverse inequality can be easily deduced by the fact that $\| {\varphi_0}* f\| _p\leq \|  {\varphi_0}\| _1\|  f\| _p$.
\end{proof}

We also have a continuous counterpart of Proposition \ref{thm:homo}. For any $\e\geq 0$, we define  $\dot{\varphi}_\e=\F^{-1}(\varphi(\e\cdot))$.
\begin{cor}\label{cor: homo}
Let $1\leq p\leq 2$ and $\alpha>0$. Then, for any $f\in F_{p}^{\alpha,c}(\R,\M)$, 
\begin{equation*}
 \| f\|_{F_{p}^{\alpha,c}} \approx \|  f\| _{p}+\Big\|  (\int _0^\infty\e^{-2\alpha}|  \dot{\varphi}_\e*f| ^2\frac{d\e}{\e})^\frac{1}{2}\Big\| _{p}.
\end{equation*}
\end{cor}


\section{Characterizations}\label{section-charact}

In this section we give two kinds of characterizations of the Triebel-Lizorkin spaces defined previously: one is done by directly replacing the function $\varphi$ in Definition \ref{def:T-L R} by more general convolution kernels; the other is described by Lusin square functions. Since the local Hardy spaces are included in the family of inhomogeneous Triebel-Lizorkin spaces, these two characterizations can be seen as extensions as well as improvements of those in \cite{XX18} for local Hardy spaces, listed in Theorems \ref{thm main1} and \ref{thm: equivalence hpD}. The multiplier theorems in section \ref{section: Multiplier theorems} will play a crucial role in this section.

\subsection{General characterizations}
We have seen in section \ref{section: Definitions of T_L} that the definition of Triebel-Lizorkin spaces is independent of the choice of $\varphi$ satisfying \eqref{condition-phi}. In this 
section, we will show that this kernel is not even necessarily a Schwartz function.

 Let $\sigma > \frac d 2$ and  $\Phi^{(0)}$, $\Phi$ be two complex-valued  infinitely differentiable  functions defined respectively on $\R$ and $\R\backslash\lbrace 0\rbrace$, which satisfy
\begin{equation}\label{eq: condition Phi 0}
 \left \{ \begin{split}
  \displaystyle &
| \Phi^{(0)}(\xi)| >0 \quad  \mbox{if }\, |  \xi| \leq 2,\\
  \displaystyle & \sup_{k\in \mathbb{N}_0}2^{-k\alpha_0} \|  \Phi^{(0)}(2^k\cdot)\varphi  \| _{H_2^\sigma}<\infty,
 \end{split} \right.
\end{equation}
and
\begin{equation}\label{eq: condition Phi}
 \left \{ \begin{split}
  \displaystyle &
 |  \Phi(\xi)| >0 \quad  \mbox{if } \,\frac{1}{2}\leq|  \xi| \leq 2,\\
  \displaystyle & \sup_{k\in \mathbb{N}_0}2^{-k\alpha_0}\|  \Phi(2^k\cdot)\varphi  \| _{H_2^\sigma}<\infty,\\
  \displaystyle & \int_{\R}(1+|  s| ^2)^\sigma |  \F ^{-1}(\Phi \varphi^{(0)}I_{-\alpha_1})(s)|  ds<\infty.
 \end{split} \right.
\end{equation}
Recall that here $I_{-\alpha_1}(\xi)$ is the symbol of the Riesz potential $I^{-\alpha_1}=(-(2\pi )^{-2}\Delta)^{\frac{-\alpha_1}{2}}$.

Let  ${\Phi}^{(j)}=\Phi(2^{-j}\cdot)$ for $j\geq 1$, and $\Phi_j$ be the function whose Fourier transform is equal to ${\Phi}^{(j)}$ for any $j\in \mathbb{N}_0$.

\begin{thm}\label{thm:general}
Let $1\leq p < \infty$ and $\alpha \in \mathbb{R}$. Assume that $\alpha_0<\alpha<\alpha_1$, $\alpha_1\geq  0$ and $\Phi^{(0)}$, $\Phi$ satisfy conditions \eqref{eq: condition Phi 0} and \eqref{eq: condition Phi} respectively. Then for any $f\in \mathcal{S}'(\R; L_1(\M)+\M)$, we have
\begin{equation}\label{eq: general chara}
\|  f\| _{F_p^{\alpha,c}} \approx \big\|  (\sum_{j\geq 0}2^{2j\alpha}|  {\Phi}_j*f| ^2)^\frac{1}{2}\big\| _{p},
\end{equation}
where the relevant constants are independent of $f$.
\end{thm}

\begin{proof}
We follow the pattern of the proof of \cite[Theorem 4.17]{XXY17}. Denote the norm on the right hand side of \eqref{eq: general chara} by $\|  f \| _{F_{p,\Phi}^{\alpha,c}}$.

\emph{Step 1.} Let $\varphi_k=0$ (and so is $\varphi^{(k)}$) if $k<0$. Given a positive integer $K$, for  any $j\in \mathbb{N}_0$, we write
$$
\Phi^{(j)}=\sum_{k=-\infty}^{K-1}\Phi^{(j)}\varphi^{(j+k)}+\sum_{k=K}^\infty\Phi^{(j)}\varphi^{(j+k)}.
$$
 Then
 \begin{equation}\label{eq:convergence}
{\Phi}_j*f=\sum_{k\leq K-1}{\Phi}_j*{\varphi}_{j+k}*f+\sum_{k\geq K}{\Phi}_j*{\varphi}_{j+k}*f.
 \end{equation}
Temporarily we take for granted that the second series is convergent not only in $\mathcal {S}'(\R; L_1(\M)+\M)$ but also in $F_p^{\alpha, c}(\R,\M)$, which is to be settled up in the last step. Then we obtain
$$
\|  f \| _{F_{p,\Phi}^{\alpha,c}}\leq \rm{I}+\rm{II}+\rm{III},
$$
where
\begin{equation*}
\begin{split}
\rm{I} &= \sum_{k\leq K-1}\big\|  (\sum_{j\geq 1}2^{2j\alpha} |  {\Phi}_j*{\varphi}_{j+k}*f| ^2)^\frac{1}{2}\big\| _p,\\
\rm{II} &= \sum_{k\leq K-1}\|  {\Phi}_0*{\varphi}_{k}*f\| _p,\\
\rm{III} &=  \sum_{k\geq K}\big\|  (\sum_{j\geq 0}2^{2j\alpha} |  {\Phi}_j*{\varphi}_{j+k}*f | ^2)^\frac{1}{2}\big\| _p.
\end{split}
\end{equation*}
The term $\rm{II}$ is easy to deal with. By \eqref{Young-Banach} and \eqref{eq: condition Phi 0}, we obtain
\begin{equation*}
\begin{split}
\sum_{k=0}^{ K-1}\| {\Phi}_0*{\varphi}_{k}*f \| _p & = \sum_{k=0}^{K-1} \|{\Phi}_0*({\varphi}_{k-1}+{\varphi}_{k}+{\varphi}_{k+1})*{\varphi}_{k}*f \| _p\\
 & \lesssim  \sum_{k=0}^{K-1}\| {\varphi}_k*f \| _p  \|{\Phi}_0*({\varphi}_{k-1}+{\varphi}_{k}+{\varphi}_{k+1}) \|_1  \\
& \lesssim  \sup_{k\in \mathbb{N}_0} 2^{-k\alpha_0}   \| \Phi^{(0)}(2^k\cdot)\varphi  \| _{H_2^\sigma} \sum_{k=0}^{K-1} 2^{k(\alpha_0-\alpha)}  \| 2^{k\alpha} {\varphi}_k*f \| _p\\
& \lesssim C_K\| f\|_{F_p^{\alpha,c}}.
 \end{split}
\end{equation*}

Let us treat the terms $\rm I$ and $\rm{ III}$ separately. By the support assumption of $\varphi^{(k)}$ and the property that it is equal to $1$ when $| \xi |\leq 1$, for $k\leq K-1$, we have
\begin{equation}\label{split 1}
\begin{split}
\Phi(\xi)\varphi^{(k)}(\xi)& = \frac{\Phi(\xi)\varphi^{(0)}(2^{-K}\xi)}{|  \xi | ^{\alpha_1}}|  \xi | ^{\alpha_1}\varphi^{(k)}(\xi)\\
 &= 2^{k\alpha_1}\eta(\xi)\rho^{(k)}(\xi),
 \end{split}
\end{equation}
where $\eta$, $\rho$ are defined by
$$
\eta(\xi)=\frac{\Phi(\xi)\varphi^{(0)}(2^{-K}\xi)}{|  \xi | ^{\alpha_1}}\quad \mbox{and}\quad \rho(\xi)=|  \xi | ^{\alpha_1}\varphi  (\xi).
$$
 Let $\eta^{(j)}=\eta(2^{-j}\cdot)$, $j\in \mathbb{Z}$. For $j\geq 1$, define ${\eta}_j=\F^{-1}(\eta^{(j)})$. Then for any $j\geq 1$, we have $$
{\Phi}_j*{\varphi}_{j+k}*f=2^{k\alpha_1}{\eta}_j*{\rho}_{j+k}*f.
$$
Now we are ready to estimate  $\rm I$. Applying Theorems \ref{multiplier-global} and \ref{lem: Multiplier p=1} twice, we get
\begin{equation}\label{sum-in-I}
\begin{split}
\rm I  &=  \sum_{k\leq K-1}2^{k(\alpha_1-\alpha)}\big\|  (\sum_{j\geq 1}2^{2(j+k)\alpha}|  {\eta}_j*{\rho}_{j+k}*f | ^2)^\frac{1}{2}\big\| _p\\
 &=  \sum_{k\leq K-1}2^{k(\alpha_1-\alpha)}\big\|  (\sum_{j\geq k+1}2^{2j\alpha}| {\eta}_{j-k}*{\rho}_{j}*f | ^2)^\frac{1}{2}\big\| _p \\
  &\lesssim \sum_{k\leq K-1}2^{k(\alpha_1-\alpha)}\max\big\{  \|  \eta^{(-k)}(\varphi^{(0)}+\varphi^{(1)})\| _{H_2^\sigma},\max_{-2\leq \ell \leq 2}\|  \eta^{(-k-\ell)}\varphi  \| _{H_2^\sigma} \big\}\\
 & \;\;\;\; \cdot  \big\| (\sum_{j\geq 0}2^{2j\alpha}|  {\rho}_{j}*f | ^2)^\frac{1}{2}\big\| _p \\
 &  \lesssim  \sum_{k\leq K-1}2^{k(\alpha_1-\alpha)}\max\big\{   \|  \eta^{(-k)}(\varphi^{(0)}+\varphi^{(1)})\| _{H_2^\sigma}, \max_{-2\leq \ell \leq 2}\|  \eta^{(-k-\ell)}\varphi  \| _{H_2^\sigma} \big\}\\
 & \;\;\;\; \cdot \max \big\{\|  I_{\alpha_1}(\varphi^{(0)}+\varphi^{(1)})\| _{H_2^\sigma},\|  I_{\alpha_1}\varphi\| _{H_2^\sigma} \big\}\big\|  (\sum_{j\geq 0}2^{2j\alpha}|  {\varphi}_{j}*f | ^2)^\frac{1}{2}\big\| _p \\
&=  \sum_{k\leq K-1}2^{k(\alpha_1-\alpha)}\max\big\{ \|  \eta^{(-k)}(\varphi^{(0)}+\varphi^{(1)})\| _{H_2^\sigma}, \max_{-2\leq \ell \leq 2}\|  \eta^{(-k-\ell)}\varphi  \| _{H_2^\sigma} \big\} \\
  &\;\;\;\; \cdot \max \big\{\|  I_{\alpha_1}(\varphi^{(0)}+\varphi^{(1)})\| _{H_2^\sigma},\|  I_{\alpha_1}\varphi\| _{H_2^\sigma} \big\}\|  f \| _{F_p^{\alpha,c}}.
  \end{split}
\end{equation}

Let us deal with all the factors in the last term of the above inequality. Firstly, when $\alpha_1=0$, it is obvious that  $I_{\alpha_1}\varphi \in H_2^\sigma(\R)$ and $ I_{\alpha_1}(\varphi^{(0)}+\varphi^{(1)}) \in H_2^\sigma(\R)$. Secondly, we treat the case $\alpha_1>0$. First,  it is easy to see that $\|  I_{\alpha_1}\varphi\| _{H_2^\sigma}<\infty$. Next, we deal with the term $I_{\alpha_1}(\varphi^{(0)}+\varphi^{(1)})$, which can be reduced to $I_{\alpha_1}\varphi^{(0)}$ by dilation. Let $N $ be a positive integer such that $\alpha_1> \frac 1 N$. If the dimension $d$ is odd, we consider the function $F(z)  =e^{(z-\frac{N+2}{2N+2})^2} |\xi| ^{\alpha_1 -\frac 1 2 -\frac 1 N + (1+\frac 1 N)z} \varphi^{(0)} $, which is continuous on the strip $\{z\in \mathbb{C}\,:\, 0\leq {\rm Re} (z)\leq 1\}$, and analytic in the interior. A direct computation shows that 
$\sup_{t\in \mathbb{R}} \| F({\rm i} t) \|_{H_2^{\frac d 2 -\frac 1 2}} <\infty$ and $\sup_{t\in \mathbb{R}} \| F({1+\rm i} t) \|_{H_2^{\frac d 2 +\frac 1 2}} <\infty$. Then for $\theta = \frac{\frac 1 N +\frac 1 2}{\frac 1 N +1}>\frac 1 2$, we have 
$$F(\theta) =I_{\alpha_1} \varphi^{(0)} \in H_2^{\sigma}(\R) =\big(H_2^{\frac d 2 -\frac 1 2}(\R), H_2^{\frac d 2 +\frac 1 2}(\R)\big)_\theta,$$
 for some $\sigma >\frac d 2$. On the other hand, if $d$ is even, set $F(z)  =e^{(z-\frac{1}{2N})^2} |\xi| ^{N \alpha_1 z +\frac{\alpha_ 1}{ 2 }} \varphi^{(0)} $. We can also check that
$\sup_{t\in \mathbb{R}} \| F({\rm i} t) \|_{H_2^{\frac d 2}} <\infty$, and that $\sup_{t\in \mathbb{R}} \| F({1+\rm i} t) \|_{H_2^{\frac d 2 +1}} <\infty$. Then for $\theta = \frac{1}{2N}$, we have 
$$
F(\theta) =I_{\alpha_1} \varphi^{(0)} \in H_2^{\frac d 2 + \frac{1}{2N}}(\R) =\big(H_2^{\frac d 2 }(\R), H_2^{\frac d 2 +1}(\R)\big)_\theta.
$$
Thus, for any $\alpha_1>0 $, we can always choose a positive $\sigma>\frac d 2 $ such that $I_{\alpha_1} \varphi^{(0)} \in H_2^{\sigma}(\R) $.
Finally, we have to estimate $\|  \eta^{(-k)}\varphi  \| _{H_2^\sigma}$ and $\|  \eta^{(-k)}\varphi^{(0)}\| _{H_2^\sigma}$ uniformly in $k$, which will yield the convergence of the last sum in \eqref{sum-in-I} by dilation again. To this end, note that by \eqref{eq: condition Phi}, $\check{\eta}$ is integrable on $\R$, then we use the  Cauchy-Schwarz inequality in the following way:
\begin{equation*}
\begin{split}
|  \F^{-1}(\eta^{(-k)}\varphi  )(s)| ^2  &=  |  \int_{\R} \check{\eta}(t)\F^{-1}(\varphi  )(s-2^{k}t)dt| ^2 \\
& \leq   \|  \check{\eta}\| _1\int_{\R} |  \check{\eta}(t)|  \cdot |  \F^{-1}(\varphi  )(s-2^{k}t)| ^2 dt.
\end{split}
\end{equation*}
For $k\leq K-1$, we have
\begin{equation}\label{eq:term I'}
\begin{split}
& \|  \eta^{(-k)}\varphi  \| _{H_2^\sigma}^2
= \int _{\R}(1+|  s| ^2)^\sigma |  \F^{-1}(\eta^{(-k)}\varphi  )(s)| ^2ds \\
& \leq  \|  \check{\eta}\| _1 \int_{\R}(1+|  s| ^2)^\sigma \int_{\R} |  \check{\eta}(t)|  \cdot |  \F^{-1}(\varphi  )(s-2^k t)| ^2 dtds \\
& \lesssim  \|  \check{\eta}\| _1 \int_{\R}(1+|  2^k t| ^2)^\sigma |  \check{\eta}(t) |
\int_{\R} (1+|  s-2^k t| ^2)^\sigma |  \F^{-1}(\varphi  )(s-2^kt)| ^2 dsdt\\
&\leq  2^{K\sigma}\|  \check{\eta}\| _1\int_{\R}(1+|  t| ^2)^\sigma |  \check{\eta}(t) |  dt \int_{\R}(1+|  s| ^2)^\sigma |  \F^{-1}(\varphi  )(s) |  ^2ds \\
& \leq  C_{\0,\sigma,K}\big(\int_{\R}(1+|  t| ^2)^\sigma |  \check{\eta}(t) |  dt \big)^2.
\end{split}
\end{equation}
The other term $\|  \eta^{(-k)}\varphi^{(0)}\| _{H_2^\sigma}$ is dealt with in the same way.

Going back to the estimate of $\rm{I}$, by the previous inequalities, we obtain
\begin{equation}\label{eq:I}
{\rm{I}} \lesssim C_{\Phi,\varphi^{(0)},\alpha_1,\alpha,K}\int_{\R}(1+|  t| ^2)^\sigma |  \check{\eta}(t) |  dt\,\|  f\| _{F_p^{\alpha,c}}.
\end{equation}
In order to return from $\eta$ back to $\0$, we write
$$
\eta = I_{-\alpha_1}\Phi [ \varphi^{(0)}(2^{-K}\cdot)-\varphi^{(0)}]+I_{-\alpha_1}\Phi\varphi^{(0)}.
$$
Since $I_{-\alpha_1}\Phi ( \varphi^{(0)} (2^{-K}\cdot )-\varphi^{(0)})$ is an infinitely differentiable function with compact support, we have $$
\int_{\R}(1+|  t| ^2)^\sigma |  \F^{-1}(I_{-\alpha_1}\Phi ( \varphi^{(0)}(2^{-K}\cdot)-\varphi^{(0)}))(t) |  dt =C'_{\Phi,\varphi^{(0)},\alpha_1,\alpha,K}<\infty.
$$
Then \eqref{eq: condition Phi} implies that 
$$
\int_{\R}(1+|  t| ^2)^\sigma |  \check{\eta}(t) |  dt \lesssim C'_{\Phi,\varphi^{(0)},\alpha_1,\alpha,K}+\int_{\R}(1+|  s| ^2)^\sigma |  \F ^{-1}(I_{-\alpha_1}\Phi \varphi^{(0)})(s)|  ds<\infty.
$$
Therefore,
$$
{ \rm{I}} \lesssim  \|  f\| _{F_p^{\alpha,c}}.
$$

\emph{Step 2.} Now it remains to estimate the third term $\rm{III}$. Let $H$ be a Schwartz function such that
\begin{equation}
\supp H\subset \lbrace \xi\in \R: \frac{1}{4}\leq |  \xi |  \leq 4\rbrace \quad \text{ and }\quad
H(\xi)=1\,\,\,\mbox{if}\,\, \,\frac{1}{2}\leq |  \xi| \leq 2.
\end{equation}
Let $H^{(k)}=H(2^{-k}\cdot)$. For $k\geq K$, we have
\begin{equation}
\Phi(\xi)\varphi^{(k)}(\xi)=\frac{\Phi(\xi)}{|  \xi| ^{\alpha_0}}H^{(k)}(\xi)\varphi  ^{(k)}(\xi)|  \xi| ^{\alpha_0},\label{eq: Step 3}\\
\end{equation}
and
\begin{equation}
\Phi^{(0)}(\xi)\varphi^{(k)}(\xi)=\frac{\Phi^{(0)}(\xi)}{|  \xi| ^{\alpha_0}}H^{(k)}(\xi)\varphi  ^{(k)}(\xi)|  \xi| ^{\alpha_0}.\label{eq: Step 3'}
\end{equation}
For any $j\in \mathbb{N}_0$, we keep using the notation ${\Phi}_j=\F^{-1}(\Phi^{(j)})$ and ${H}_j=\F^{-1}(H^{(j)})$. Thus, we have
$$
{\Phi}_j*{\varphi}_{j+k}*f=2^{k\alpha_0}{(I_{-\alpha_0}\Phi)}_j*{H}_{j+k}*{(I_{\alpha_0}\varphi  )}_{j+k}*f.
$$
Therefore,
\begin{equation*}
\begin{split}
\rm{III} &=  \sum_{k\geq K}2^{k(\alpha_0-\alpha)}\big\|  \big(\sum_{j\geq 0}2^{2(j+k)\alpha} |  {(I_{-\alpha_0}\Phi})_j*{H}_{j+k}*{(I_{\alpha_0}\varphi  )}_{j+k}*f| ^2  \big)^\frac{1}{2}\big\| _p\\
 &=  \sum_{k\geq K} 2^{k(\alpha_0-\alpha)}\big\|  \big(\sum_{j\geq k}2^{2j\alpha} |  {(I_{-\alpha_0}\Phi)}_{j-k}*{H}_{j}*{(I_{\alpha_0}\varphi  )}_{j}*f| ^2\big)^\frac{1}{2}\big\| _p.
\end{split}
\end{equation*}
Since both $H$ and $\varphi  $ vanish near the origin, by Theorems \ref{multiplier-global} and  \ref{lem: Multiplier p=1}, we obtain
\begin{equation*}
\begin{split}
&  \sum_{k\geq K} 2^{k(\alpha_0-\alpha)}\big\|  \big(\sum_{j\geq k}2^{2j\alpha} |  {(I_{-\alpha_0}\Phi)}_{j-k}*{H}_{j}*{(I_{\alpha_0}\varphi  )}_{j}*f| ^2\big)^\frac{1}{2}\big\| _p\\
& \lesssim  \sup_{k\in\mathbb{N}_0}\max\big\{  2^{-k\alpha_0} \max _{-2\leq \ell\leq 2}\|  I_{-\alpha_0} \Phi (2^{k+\ell} \cdot)H(2^\ell \cdot)\varphi  \| _{H_2^\sigma}, 2^{-k\alpha_0}\|  I_{-\alpha_0} \Phi^{(0)} (2^k \cdot)H(\varphi^{(0)}+ \varphi^{(1)})\| _{H_2^\sigma}\big\} \\
& \;\;\;\;  \cdot  \sum_{k\geq K} 2^{k(\alpha_0-\alpha)} \|  f\| _{F_p^{\alpha,c}}.
\end{split}
\end{equation*}
Then by \eqref{ineq-potential-Sobolev},  \eqref{eq: condition Phi 0} and \eqref{eq: condition Phi}, we have, for any $-2\leq \ell\leq 2$,
\begin{equation}\label{eq:III 1}
\begin{split}
&  2^{-k\alpha_0}  \| I_{-\alpha_0} \Phi (2^{k+\ell} \cdot)H(2^\ell \cdot)\varphi  \| _{H_2^\sigma}\\
& \leq  2^{-k\alpha_0}  \|  \Phi(2^{k+\ell} \cdot)\varphi  \| _{H_2^\sigma}\int _{\R} (1+|  t| ^2)^{\sigma}|  \F^{-1}(I_{-\alpha_0}H(2^\ell \cdot))(t)|  dt \\
& \lesssim   2^{-k\alpha_0}  \|  \Phi(2^{k+\ell} \cdot)\varphi  \| _{H_2^\sigma}\leq  \sup_{k\in\mathbb{N}_0}  2^{-k\alpha_0} \|  \Phi(2^{k} \cdot)\varphi  \| _{H_2^\sigma}<\infty ,
\end{split}
\end{equation}
and
\begin{equation}\label{eq:III 2}
\begin{split}
& 2^{-k\alpha_0} \|  I_{-\alpha_0} \Phi^{(0)} (2^k \cdot)H (\varphi^{(0)}+ \varphi^{(1)}) \| _{H_2^\sigma} \\
& =  2^{-k\alpha_0} \|  I_{-\alpha_0} \Phi^{(0)} (2^k \cdot)H \sum_{\ell' =-2}^1\varphi(2^{-\ell'}\cdot) \| _{H_2^\sigma} \\
& \lesssim  2^{-k\alpha_0} \sum_{\ell '=-2}^1\|  I_{-\alpha_0} \Phi^{(0)} (2^{k+\ell'} \cdot)H(2^{\ell '} \cdot) \varphi  \| _{H_2^\sigma} \\
& \leq  2^{-k\alpha_0} \sum_{\ell '=-2}^1 \|  \Phi^{(0)}(2^{k+\ell'} \cdot)\varphi  \| _{H_2^\sigma}\int_{\R}  (1+|  t| ^2)^{\sigma}|  \F^{-1}(I_{-\alpha_0}H(2^\ell \cdot))(t)|  dt\\
& \lesssim  \sup_{k\in\mathbb{N}_0}  2^{-k\alpha_0} \|  \Phi^{(0)}(2^{k} \cdot)\varphi  \| _{H_2^\sigma}<\infty .
\end{split}
\end{equation}
Then  we get
$$
{\rm{III}} \leq C_{\Phi,\alpha_0,\alpha,K} \|  f\| _{F_p^{\alpha,c}}.
$$
Combining this estimate with those of $\rm{I}$ and $\rm{II}$, we finally get
$$
\|  f\| _{F_{p,\Phi}^{\alpha,c}}\lesssim \|  f\| _{F_p^{\alpha,c}}.
$$

\emph{Step 3.} We turn to the reverse inequality. Note that $\varphi^{(0)}(\xi)=1$ when $|\xi |\leq 1$, then by  \eqref{eq: condition Phi 0} and \eqref{eq: condition Phi}, for any $j\in \mathbb{N}_0$, we write
\begin{equation}\label{eq:step 3}
\varphi^{(j)}(\xi)=\varphi^{(j)}(\xi)\,\varphi^{(0)}(2^{-j-M}\xi)=\frac{\varphi  ^{(j)}(\xi)}{\Phi^{(j)}(\xi)}\varphi^{(0)}(2^{-j-M}\xi)\Phi^{(j)}(\xi),
\end{equation}
where $M$ is a positive integer to be chosen later. By Theorems \ref{multiplier-global} and \ref{lem: Multiplier p=1},
\begin{equation*}
\begin{split}
\|  f\| _{F_p^{\alpha,c}}&= \big\|  \big(\sum_{j\geq 0}2^{2j\alpha}|  {\varphi}_j*f | ^2\big)^\frac{1}{2}\big\| _p\\
& \lesssim  \max\big\{ \max _{-2\leq \ell\leq 2} \|  \Phi^{-1}(2^{\ell}\cdot) \varphi (2^{\ell}\cdot) \varphi \| _{H_2^\sigma}, \| (\Phi^{(0)})^{-1} \varphi^{(0)}(\varphi^{(0)}+\varphi^{(1)})\| _{H_2^\sigma}\big\}\\
& \;\;\;\;\cdot \big\|  \big(\sum_{j\geq 0}2^{2j\alpha}|  {(\varphi_0)}_{j+M}*{\Phi}_j*f | ^2\big)^\frac{1}{2}\big\| _p\\
& \lesssim \big\|  \big(\sum_{j\geq 0}2^{2j\alpha}|  {(\varphi_0)}_{j+M}*{\Phi}_j*f | ^2\big)^\frac{1}{2}\big\| _p,
\end{split}
\end{equation*}
where $(\varphi_0)_{j+M}$ is the Fourier inverse transform of $\varphi^{(0)}(2^{-j-M}\cdot)$.
Let $h=1-\varphi^{(0)}$. Write $\varphi^{(0)}(2^{-j-M}\xi)\Phi^{(j)}(\xi)=\Phi^{(j)}(\xi)-h^{(j+M)}(\xi)\Phi^{(j)}(\xi)$. Then, we have
$$
\|  f\| _{F_p^{\alpha,c}} \lesssim  \|  f\| _{F_{p,\Phi}^{\alpha,c}}+\big\|  \big(\sum_{j\geq 0}2^{2j\alpha}|  {h}_{j+M}*{\Phi}_j*f | ^2  \big)^\frac{1}{2}\big\| _p,
$$
where the relevant constant depends only on $p$, $\sigma$, $d$ and $\varphi^{(0)}$. Applying  the arguments in the estimate of $\rm III$, \eqref{eq: Step 3} with $h^{(M)}\Phi$ in place of $\Phi$ and $\eqref{eq: Step 3'}$ with $h^{(M)}\Phi^{(0)}$ in place of $\Phi^{(0)}$, we deduce
\begin{equation*}
\begin{split}
 & \big\|  \big(\sum_{j\geq 0}2^{2j\alpha}|  {h}_{j+M}*{\Phi}_j*f | ^2 \big)^\frac{1}{2}\big\| _p \\
 &\leq  C_1 \sup_{k\geq M}2^{-k\alpha_0}\max \big\{\max _{-2\leq \ell\leq 2}  \|  h(2^{k-M+\ell}\cdot)\Phi(2^{k+\ell}\cdot)\varphi  \| _{H_2^\sigma}, \|  h(2^{k-M}\cdot)\Phi^{(0)}(2^k\cdot)\varphi  \| _{H_2^\sigma} \big\} \\
 & \;\;\;\; \cdot \sum_{k\geq M} 2^{k(\alpha_0-\alpha)}\|  f\| _{F_p^{\alpha,c}}\\
&= \sup_{k\geq M}2^{-k\alpha_0}\max \big\{\max _{-2\leq \ell\leq 2}  \|  h(2^{k-M+\ell}\cdot)\Phi(2^{k+\ell}\cdot)\varphi  \| _{H_2^\sigma}, \|  h(2^{k-M}\cdot)\Phi^{(0)}(2^k\cdot)\varphi  \| _{H_2^\sigma} \big\} \\
& \;\;\;\; \cdot C_1\frac{2^{M(\alpha_0-\alpha)}}{1-2^{\alpha_0-\alpha}}\|  f\| _{F_p^{\alpha,c}},
\end{split}
\end{equation*}
where $C_1$ is a constant which depends only on $p$, $\sigma$, $d$, $H$ and $\alpha_0$. Now we replace $h$ in the above Sobolev norm by $1-\varphi^{(0)}$:
$$
\|  h(2^{k-M+\ell}\cdot)\Phi(2^{k+\ell}\cdot)\varphi  \| _{H_2^\sigma}\leq \|  \Phi(2^{k+\ell}\cdot)\varphi  \| _{H_2^\sigma}+\|  \varphi^{(0)}(2^{k-M+\ell}\cdot)\Phi(2^{k+\ell}\cdot)\varphi  \| _{H_2^\sigma}.
$$
The support assumptions of $\varphi^{(0)}$ and $\varphi  $ imply that when $k\geq M$, $\varphi^{(0)}(2^{k-M+\ell}\cdot)\varphi  \neq 0$ if and only if $k+\ell=M$ or $k+\ell=M+1$. Then by \eqref{ineq-potential-Sobolev}, we have $$
\|  \varphi^{(0)}(2^{k-M+\ell}\cdot)\Phi(2^{k+\ell}\cdot)\varphi  \| _{H_2^\sigma}\leq C_2 \|  \Phi(2^{k+\ell}\cdot)\varphi  \| _{H_2^\sigma},
$$
where $C_2$ depends only on $\varphi^{(0)}$, $\sigma$ and $d$. Thus,
$$
\|  h(2^{k-M+\ell}\cdot)\Phi(2^{k+\ell}\cdot)\varphi  \| _{H_2^\sigma}\leq(1+C_2)\|  \Phi(2^{k+\ell}\cdot)\varphi  \| _{H_2^\sigma}.
$$
Similarly, we have 
$$
\|  h(2^{k-M}\cdot)\Phi^{(0)}(2^k\cdot)\varphi  \| _{H_2^\sigma}\leq(1+C_2)\| \Phi^{(0)}(2^k\cdot)\varphi  \| _{H_2^\sigma}.
$$
Putting all the estimates that we have obtained so far together, we get
\begin{equation*}
\begin{split}
\|  f\| _{F_p^{\alpha,c}}  & \leq C_3\Big(C_1(1+C_2)\frac{2^{M(\alpha_0-\alpha)}}{1-2^{\alpha_0-\alpha}}\sup_{k\geq M}2^{-k\alpha_0}\max \{\|  \Phi(2^{k}\cdot)\varphi  \| _{H_2^\sigma},\|  \Phi^{(0)}(2^k\cdot)\varphi  \| _{H_2^\sigma}\}\|  f\| _{F_p^{\alpha,c}} \\
& \;\;\;\;\;\;\;\; +\|  f\| _{F_{p,\Phi}^{\alpha,c}}\Big),
\end{split}
\end{equation*}
where the three constants $C_1,C_2,C_3$ are independent of $M$, so we could take $M$ large enough to make sure the multiple of $\|  f\| _{F_p^{\alpha,c}}$ above is less than $\frac 1 2$, so that we have
 $$
\|  f\| _{F_p^{\alpha,c}}\lesssim \|  f\| _{F_{p,\Phi}^{\alpha,c}}.
$$

\emph{Step 4.} We now settle the convergence issue of the second series in \eqref{eq:convergence}. For every $j\geq 0$, ${\Phi}_j*{\varphi}_{j+k}*f$ is an $L_1(\M)+\M$-valued tempered distribution on $\R$. We now show that the series converges in $\mathcal{S}'(\R;L_1(\M)+\M)$. By \eqref{eq:III 1} and \eqref{eq:III 2}, for any $L>K$, we have
\begin{equation*}
\begin{split}
& 2^{j\alpha}\sum_{k=K}^L\|  {\Phi}_j*{\varphi}_{j+k}*f \| _p\\
&\lesssim  \|  I_{\alpha_0} \varphi   \| _{H_2^\sigma}  \sum_{k\geq K} 2^{k(\alpha_0-\alpha)} \sup_{k\in\mathbb{N}_0} \max\big\{  2^{-k\alpha_0} \|  \Phi(2^k\cdot)\varphi  \| _{H_2^\sigma}, 2^{-k\alpha_0} \|  \Phi^{(0)}(2^k\cdot)\varphi  \| _{H_2^\sigma}\big\} \| f\| _{F_p^{\alpha,c}}\\
&\lesssim \|  f\| _{F_p^{\alpha,c}} .
\end{split}
\end{equation*}
Therefore, for any $j\geq 0$, $\sum_{k\geq K+1} {\Phi}_j*{\varphi}_{j+k}*f$ converges in $L_p(\N)$, so in $\mathcal{S}'(\R;L_1(\M)+\M)$ too. In the same way, we can show that the series also converges in $F_p^{\alpha,c}(\R,\M)$, which completes the proof.
\end{proof}

The following is the continuous analogue of Theorem \ref{thm:general}. We use similar notation for continuous parameters: given $\e>0$,  ${\Phi}_\e$ denotes the function whose Fourier transform is $\Phi^{(\e)}=\Phi(\e\cdot)$.
\begin{thm}\label{thm: charac T-L cont }
Keep the assumption of the previous theorem. Then for $f\in \mathcal{S}'(\R;L_1(\M)+\M)$, we have 
\begin{equation}\label{eq:cont}
\|  f\| _{F_p^{\alpha,c}}\approx \|  {\Phi}_0*f \| _p+ \Big\|\big(\int _0^1\e^{-2\alpha}|  {\Phi}_\e*f| ^2\frac{d\e}{\e}\big)^\frac{1}{2}\Big\| _p.
\end{equation}
\end{thm}
\begin{proof}
This proof is very similar to the previous one. We keep the notation there and  only point out the necessary modifications. First, we need to discretize the integral on the right hand side of \eqref{eq:cont}. There exist two constants $C_1, C_2$ such that 
$$
C_1\sum_{j=0}^\infty2^{2j\alpha}\int_{2^{-j-1}}^{2^{-j}}|  {\Phi}_\e*f| ^2\frac{d\e}{\e} \leq   \int _0^1\e^{-2\alpha}|  {\Phi}_\e*f| ^2\frac{d\e}{\e}\leq C_2 \sum_{j=0}^\infty2^{2j\alpha}\int_{2^{-j-1}}^{2^{-j}}|  {\Phi}_\e*f| ^2\frac{d\e}{\e}.
$$
By approximation, we can  assume that $f$ is good enough so that  each integral over the interval $(2^{-j-1},2^{-j})$ can be approximated uniformly by discrete sums.
Instead of $\Phi^{(j)}(\xi)=\Phi(2^{-j}\xi)$, we have now $\Phi^{(\e)}(\xi)=\Phi(\e\xi)$ with $2^{-j-1}<\e\leq 2^{-j}$. We transfer the split \eqref{split 1}  into:
\begin{equation*}
\Phi^{(\e)}(\xi)\varphi_{j+k}(\xi) = \frac{\Phi(2^{-j}\cdot 2^j\e\xi)\varphi^{(0)}(2^{-K}\xi)}{|  2^{-j}\xi | ^{\alpha_1}}|  2^{-j}\xi | ^{\alpha_1
}\varphi_{j+k}(\xi).
\end{equation*}
Thus,$$
{\Phi}_{\e}*{\varphi}_{j+k}*f= 2^{k\alpha_1}{\eta}_j*{\rho}_{j+k}*f
$$
with
$$
\eta(\xi)=\frac{\Phi(2^j\e\xi)\varphi^{(0)}(2^{-K}\xi)}{|  \xi | ^{\alpha_1}}\quad \mbox{and} \quad \rho(\xi)=|  \xi | ^{\alpha_1}\varphi  (\xi).
$$
We proceed as  in step 1 of the previous theorem, where we transfer \eqref{eq:term I'} to the present setting:
\begin{equation*}
\begin{split}
\|  \eta^{(-k)}\varphi  \| _{H_2^\sigma} & \lesssim C_{\varphi^{(0)},\sigma,k}\int_{\R}(1+|  t| ^2)^\sigma |  \check{\eta}(t) |  dt\\
&=  C_{\varphi^{(0)},\sigma,k}\int_{\R}(1+|  t| ^2)^\sigma \big| \F^{-1}\big(I_{-\alpha_1}\Phi(\delta_j\cdot)\varphi^{(0)}(2^{-K}\cdot)\big)(t) \big| dt\\
& \leq   C_{\varphi^{(0)},\sigma,k} \delta_j^{\alpha_1}\int_{\R}(1+|  t| ^2)^\sigma \big| \F^{-1}\big(I_{-\alpha_1}\Phi \varphi^{(0)}(\delta_j^{-1}2^{-K}\cdot)\big)(t)\big| dt,
\end{split}
\end{equation*}
where $\delta_{j}=2^{j}\e$ and $\frac{1}{2}<\delta_{j}\leq 1$. The last integral is estimated as follows:
\begin{equation*}
\begin{split}
& \int_{\R}(1+|  t| ^2)^\sigma \big| \F^{-1}\big(I_{-\alpha_1}\Phi \varphi^{(0)}(\delta_{j}^{-1}2^{-K}\cdot)\big)(t) \big| dt\\ 
&\leq  \int_{\R}(1+|  t| ^2)^\sigma |  \F^{-1}(I_{-\alpha_1}\Phi \varphi^{(0)})(t) | dt\\
&\;\;\;\; +\int_{\R}(1+|  t| ^2)^\sigma \big|  \F^{-1}\big(I_{-\alpha_1}\Phi[\varphi^{(0)}- \varphi^{(0)}(\delta_{j}^{-1}2^{-K}\cdot)]\big)(t)\big|  dt\\
& \leq \int_{\R}(1+|  t| ^2)^\sigma |  \F^{-1}(I_{-\alpha_1}\Phi \varphi^{(0)})(t)|  dt\\
& \;\;\;\;  + \sup _{\frac{1}{2}<\delta\leq 1}\int_{\R}(1+|  t| ^2)^\sigma \big|  \F^{-1}\big(I_{-\alpha_1}\Phi[\varphi^{(0)}- \varphi^{(0)}(\delta^{-1}2^{-K}\cdot)]\big)(t)\big| dt.
\end{split}
\end{equation*}
Note that the above supremum is finite since $I_{-\alpha_1}\Phi[\varphi^{(0)}- \varphi^{(0)}(\delta^{-1}2^{-K}\cdot)]$ is a compactly supported and infinitely differentiable function whose inverse Fourier transform depends continuously on $\delta$. Then it follows that for $2^{-j-1}\leq \e\leq 2^{-j}$,
$$
\sum_{k\leq K-1}\Big\|  \big(\int _0^1\e^{-2\alpha}|  {\Phi}_\e*f| ^2\frac{d\e}{\e}\big)^\frac{1}{2}\Big\| _p\lesssim \sum_{k\leq K-1}2^{k(\alpha_1-\alpha)}\|  f \|  _{F_p^{\alpha,c}}\lesssim \|  f\|  _{F_p^{\alpha,c}}.
$$
We make similar modifications in step 2 of the previous theorem and then establish the third part. Moreover, by the previous theorem, $\|  {\Phi}_0*f \| _p\lesssim  \|  f\|  _{F_p^{\alpha,c}}$. Thus, we have proved
$$
\|  {\Phi}_0*f \| _p+\Big\| \big(\int _0^1\e^{-2\alpha}|  {\Phi}_\e*f| ^2\frac{d\e}{\e}\big)^\frac{1}{2}\Big\| _p\lesssim  \|  f\|  _{F_p^{\alpha,c}}.
$$

For the  reverse inequality, we follow the argument in step 3 in the previous proof. By \eqref{eq: condition Phi}, there exists $2<a\leq 2\sqrt{2}$ such that $\Phi (\xi)>0$ on $\{\xi: a^{-1}\leq |  \xi | \leq a\}$. Then for $j\geq 1$, $R_j=\{\e: a^{-1}2^{-j+1}<\e\leq a2^{-j-1}\}$ are disjoint sub intervals on $(0,1]$, and $\frac{\varphi^{(j)}}{\Phi^{(\e)}}$ is well-defined for any $\e\in R_j$. We slightly modify \eqref{eq:step 3} as follows: for any $\e\in R_j$, we have
\begin{equation*}
\varphi^{(j)}(\xi)=\varphi^{(j)}\varphi^{(0)}(2^{-j-K}\xi)=\frac{\varphi  ^{(j)}(\xi)}{\Phi^{(\e)}(\xi)}\varphi^{(0)}(2^{-j-K}\xi)\Phi^{(\e)}(\xi), \quad j\in \mathbb{N}_0.
\end{equation*}
Since for any $-2\leq \ell \leq 2$, 
$$
\|  \Phi^{-1}(2^{-j}\e^{-1}2^\ell \cdot) \varphi(2^\ell \cdot )\varphi \| _{H_2^\sigma} \leq \sup_{2a^{-1}\leq \delta \leq \frac{a}{2}}\|  \Phi^{-1}(\delta 2^\ell \cdot) \varphi (2^\ell \cdot) \varphi \| _{H_2^\sigma}<\infty
$$
and
$$
\| (\Phi^{(0)})^{-1}(2^{-j}\e^{-1}\cdot) \varphi^{(0)} (\varphi^{(0)}+\varphi^{(1)}) \| _{H_2^\sigma} \leq \sup_{2a^{-1}\leq \delta \leq \frac{a}{2}}\| (\Phi^{(0)})^{-1}(\delta \cdot) \varphi^{(0)} (\varphi^{(0)}+\varphi^{(1)})  \| _{H_2^\sigma}<\infty,
$$
 we follow the argument in step 3 in the previous theorem to get
\begin{equation*}
\begin{split}
\|  f\|  _{F_p^{\alpha,c}} &\lesssim   \big\|  \big(\sum_{j\geq 0}2^{2j\alpha}\int_{R_j}|  (\varphi_0)_{j+k}*{\Phi}_\e*f | ^2\big)^\frac{1}{2}\big\| _p\\
&\lesssim   \|  \Phi_0 *f \| _p+\Big\|  \big(\int _0^1\e^{-2\alpha}|  \Phi_\e*f| ^2\frac{d\e}{\e}\big)^\frac{1}{2}\Big\| _p +\big\|  \Big(\sum_{j\geq 0}2^{2j\alpha}\int_{R_j}|  {h}_{j+k}*\Phi_{\e} * f | ^2\frac{d\e}{\e}\big)^\frac{1}{2}\Big\| _p.
\end{split}
\end{equation*}
The remaining of the proof follows step 3 with necessary modifications.
\end{proof}

We now concretize the general characterization in the previous theorem to the case of Poisson kernel. Recall that $\mathrm{P}$ denotes the Poisson kernel of $\R$ and
 $$\mathrm{P}_\e(f)(s)=\int_{\mathbb{R}^d}\mathrm{P}_\e(s-t)f(t)dt, \quad (s,\e)\in \mathbb{R}^{d+1}_+\,.$$

The following theorem improves \cite[Section 2.6.4]{Tri1992} even in the classical case: \cite[Section 2.6.4]{Tri1992} requires $k >d+\max\{\alpha, 0\}$ for the Poisson characterization while we only need $k > \max\{\alpha, 0\}$. The proof of this theorem is similar to but easier than that of \cite[Theorem 4.20]{XXY17}, since we assume $k>0$ here; we omit the details. The key ingredient is the improvement of the characterization of Hardy spaces in terms of Poisson kernel given in \cite[Theorem 1.5]{XXX17}

\begin{thm}
Let $1\leq p <\infty$, $\alpha\in \mathbb{R}$, and $k\in \mathbb{N}$ such that $k>\max\{\alpha, 0\}$. Assume that $\Phi^{(0)}$ satisfies \eqref{eq: condition Phi 0}. Then for $f\in \mathcal{S}'(\R;L_1(\M)+\M)$, we have 
\begin{equation*}
\|  f\| _{F_p^{\alpha,c}}\approx \|  {\Phi}_0*f \| _p+ \Big\|\big(\int _0^1\e^{2(k-\alpha)}  \big|  \frac{\partial^k}{\partial \e^k}\mathrm{P}_\e(f) \big| ^2\frac{d\e}{\e}\big)^\frac{1}{2}\Big\| _p.
\end{equation*}

\end{thm}

\medskip

\subsection{Characterizations via Lusin functions}
We are going to give some characterizations for Triebel-Lizorkin spaces via Lusin square functions.  As what we did in the previous part of this section, we still use Fourier multiplier theorems as our main tool. But now we have to rely on the Hilbertian (instead of $\ell_2$) versions of the Fourier multiplier theorems.

The following characterization, via Lusin square functions associated to $\varphi $ given by the condition \eqref{condition-phi}, is a special case of the characterization in Theorem \ref{thm: equivalence hpD}. We keep using the notation $\varphi_j$ being the function whose Fourier transform is equal to $\varphi(2^{-j}\cdot)$ for $j\in \mathbb{N}$, and $\varphi_0$ being the function whose Fourier transform is equal to $1-\sum_{j\geq 1}\varphi (2^{-j}\cdot)$.

 \begin{prop}\label{prop: charac of F via Lusin}
For $1\leq p <\infty$ and $f\in F_p^{\alpha,c}(\R,\M)$, we have
\begin{equation}\label{eq: charac of F via Lusin}
\|  f\| _{F_p^{\alpha,c}} \approx  \| \varphi_0* f\|_p+ \Big\|  (\sum_{j\geq 1}2^{j(2\alpha+d)} \int_{B(0,2^{-j})}|{\varphi }_j*f(\cdot+t) | ^2dt)^\frac{1}{2}\Big\| _p.
\end{equation}

\end{prop}

\begin{proof}
For any $f\in F_p^{\alpha,c}(\R,\M)$, by the lifting property in Proposition \ref{prop: lifting}, we have $J^\alpha f\in \h_p^c(\R,\M)$. Then, we apply the discrete characterization in Theorem \ref{thm: equivalence hpD} with $\phi=J^{-\alpha}\varphi_0$ and $\Phi=I^{-\alpha}\varphi $ to $J^\alpha f$, 
\[
\|  f\| _{F_p^{\alpha,c}} \approx\| J^\alpha f \|_{\h_p^c}\approx\| {\varphi}_0* f\|_{p}+\|s_{I^{-\alpha}\varphi }^{c,D}(J^\alpha f)\|_{p}.
\]
Following the argument in the proof of \eqref{eq: J and I equi}, we can prove
\[
\big\|s_{I^{-\alpha}\varphi }^{c,D}(J^\alpha f)\big\|_{p}\approx \big\|s_{I^{-\alpha}\varphi }^{c,D}(I^\alpha f)\big\|_{p}.
\]
Moreover, we can easily check that
\[
\big\|s_{I^{-\alpha}\varphi }^{c,D}(I^\alpha f)\big\|_{p}=\Big\|  (\sum_{j\geq 1}2^{j(2\alpha+d)} \int_{B(0,2^{-j})}|{\varphi }_j*f(\cdot+t) | ^2dt)^\frac{1}{2}\Big\| _p.
\]
Therefore, we conclude
\[
\|  f\| _{F_p^{\alpha,c}} \approx \| \varphi_0* f\|_p+ \Big\| (\sum_{j\geq 1}2^{j(2\alpha+d)} \int_{B(0,2^{-j})}|{\varphi }_j*f(\cdot+t) | ^2dt)^\frac{1}{2}\Big\| _p.
\]
The assertion is proved.

\end{proof}

From the above Lusin square function by $\varphi$, we can deduce Lusin type characterizations with general convolution kernels by the aide of Theorems \ref{cor: conic multiplier} and \ref{lem: Multiplier conic p=1}.

\begin{thm}\label{thm:general Lusin}
Let $1\leq p < \infty$ and $\alpha \in \mathbb{R}$. Assume that $\alpha_0<\alpha<\alpha_1$, $\alpha_1\geq 0$ and $\Phi^{(0)}$, $\Phi$ satisfy conditions \eqref{eq: condition Phi 0}, \eqref{eq: condition Phi}. Then for any $f\in \mathcal{S}'(\R;L_1(\M)+\M)$, we have
\begin{equation*}
\|  f\| _{F_p^{\alpha,c}(\R,\M)} \approx \|\Phi_0*f\|_{p}+  \Big\|(\sum_{j\geq 1}2^{j(2\alpha+d)}\int_{B(0,2^{-j})} |{\Phi}_j*f(\cdot+t)| ^2 dt)^\frac{1}{2}\Big\| _{p},
\end{equation*}
where the equivalent constant is independent of $f$.
\end{thm}

\begin{proof}
This proof is very similar to that of Theorem \ref{thm:general}. 
The main target is to replace the standard test functions $\varphi_0$ and $\varphi$ in Proposition \ref{prop: charac of F via Lusin} with $\Phi_0$ and $\Phi$ satisfying \eqref{eq: condition Phi 0} and \eqref{eq: condition Phi}. This time we need to use the Lusin type multiplier theorem i.e. Theorem \ref{cor: conic multiplier}, instead of Theorem \ref{multiplier-global}. For the special case $p=1$, we apply Theorem \ref{lem: Multiplier conic p=1} instead of Theorem \ref{lem: Multiplier p=1}.
\end{proof}

Using a similar argument as in Theorem \ref{thm: charac T-L cont }, we also have the following continuous analogue of the above theorem. This is the general characterization of Triebel-Lizorkin spaces by Lusin square functions. Recall that $\widetilde \Gamma =  \{ (t,\varepsilon )\in\mathbb{R}_{+}^{d+1}:  |t |<\varepsilon<1 \} $.
\begin{thm}\label{thm:general Lusin cont}
Keep the assumption in the previous theorem. Then for any $L_1(\M)+\M$-valued tempered distribution $f$ on $\R$, we have
\begin{equation*}
\|  f\| _{F_p^{\alpha,c}(\R,\M)} \approx \|\Phi_0*f\|_{p}+\Big\|  \big(\int_{\widetilde{\Gamma}} \e^{-2\alpha} |{\Phi}_\e*f(\cdot+ t)| ^2 \frac{dtd\e}{\e^{d+1}}\big)^\frac{1}{2}\Big\| _{p}.
\end{equation*}
\end{thm}

 \section{Smooth atomic decomposition}\label{section-atom}

This section is devoted to the study of atomic decomposition of $F_1^{\alpha,c }(\R, \M)$. We aim to decompose $F_1^{\alpha,c }(\R, \M)$ into atoms which have good enough size, smooth and moment conditions. To proceed in an orderly way step by step, we begin with the special case $\alpha=0$, i.e., the space $\h_1^c(\R,\M)$. Even though the result for $\h_1^c(\R,\M)$ below does not lead to the one for general $F_1^{\alpha,c }(\R, \M)$ directly, the main ingredients to obtain smooth decomposition for $F_1^{\alpha,c }(\R, \M)$ are already contained in those for $\h_1^c(\R,\M)$. The main results in this section will be very useful in our forthcoming paper \cite{XX18PDO} on mapping properties of pseudo-differential operators.

\subsection{Smooth atomic decomposition of $\h_{1}^{c}(\R,\M)$}

In the classical theory, the smooth atoms have been widely studied and have played a crucial role when studying the mapping  properties of pseudo-differential operators acting on local Hardy spaces, or more generally, on Triebel-Lizorkin spaces. Details can be found in \cite{CMS1985}, \cite{FJ1988} and \cite{Tri1992}. In this subsection, we will show that in our operator-valued case, the atoms in Theorem \ref{thm:atomic h1} can also be refined to be infinitely differentiable.

As in the classical case, the theory of tent spaces will be of great service in our proof of smooth atomic decomposition theorem. Tent spaces in the operator-valued setting have been introduced in \cite{Mei2007} and \cite{Mei-tent} first; see also \cite{XXX17} for further complement. For our use, we study the local version of tent spaces defined in \cite{XX18}. For any function defined on the strip $S= \mathbb{R}^{d}\times (0,1)$
with values in $L_{1}  (\mathcal{M} )+\M$, whenever it exists, we
define
$$
A^{c}  (f )(s)=  \Big(\int_{\widetilde{\Gamma}}  |f(t+s,\varepsilon) |^{2}\frac{dtd\varepsilon}{\varepsilon^{d+1}} \Big)^{\frac{1}{2}},s\in\mathbb{R}^{d}.
$$
For $1\leq p<\infty$, we define
$$
T_{p}^{c}(\R,\M)=  \{f:A^{c}  (f )\in L_{p}  (\mathcal{N} ) \}
$$
equipped with the norm $  \|  f \|  _{T_{p}^{c}(\R,\M)}=  \|  A^{c}  (f ) \|  _p$.

 First, we introduce a lemma concerning the atomic
decomposition of the tent space $T_{1}^{c}(\R,\M)$. A function $a \in L_1\big(\M; L_2(S,\frac{dsd\e}{\e})\big)$
is called a $T_{1}^{c}$-atom if
\begin{itemize}
\item supp $a\subset T(Q)$ for some cube $Q$ in $\mathbb{R}^{d}$ with $|Q|\leq 1$;
\item $\tau  \Big(\int_{T(Q)}  |a(s,\varepsilon) |^{2}\frac{dsd\varepsilon}{\varepsilon} \Big)^{\frac{1}{2}}\leq  |Q |^{-\frac{1}{2}}$.
\end{itemize}

Let $T_{1,at}^{c}  (\mathbb{R}^{d},\mathcal{M} )$ be the
space of all $f: S\rightarrow L_1(\M)$ admitting a representation of the form
\begin{equation}   \label{eq: atom tent}
f =\sum_{j=1}^{\infty}\lambda_{j}a_{j},
\end{equation}
where the $a_{j}$'s are $T_{1}^{c}$-atoms and $\lambda_{j}\in\mathbb{C}$
such that $\sum_{j=1}^{\infty}  |\lambda_{j} |<\infty$.
We equip $T_{1,at}^{c}  (\mathbb{R}^{d}, \mathcal{M} )$ with
the following norm
\[
  \|  f \|  _{T_{1,at}^{c}}=\inf  \{ \sum_{j=1}^{\infty}  |\lambda_{j} |:f=\sum_{j=1}^{\infty}\lambda_{j}a_{j}; \,a_{j}\text{'s are \ensuremath{T_{1}^{c}}-atoms}, \lambda_{j}\in\mathbb{C} \} .
\]

\begin{lem}
\label{lem: atomic decomp Tent}
We have $T_{1,at}^{c}  (\mathbb{R}^{d},\mathcal{M} )=T_{1}^{c}  (\mathbb{R}^{d},\mathcal{M} )$
with equivalent norms.
\end{lem}
\begin{proof}
In order to prove $T_{1,at}^{c}  (\mathbb{R}^{d},\mathcal{M} )\subset T_{1}^{c}  (\mathbb{R}^{d},\mathcal{M} )$, it is enough to show that any $T_1^c$-atom $a$ satisfies  $\| a\|_{T_1^c}\lesssim 1$. By the support assumption of $a$, we have
\begin{equation*}
\begin{split}
\| a\|_{T_1^c}  = \big\|A^c(a)\big\|_1
&  =  \tau \int_{\R}\Big(\int_0^1\int_{B(t,\e)}|a(s,\e)|^2\frac{dsd\e}{\e^{d+1}}\Big)^\frac{1}{2}dt\\
 & \lesssim  |Q|^{\frac{1}{2}}\tau \Big(\int_{\R}\int_0^1\int_{B(t,\e)} |a(s,\e)|^2 \frac{dsd\e}{\e^{d+1}}dt\Big)^\frac{1}{2}\\
&  =   c_d ^\frac{1}{2}|Q|^{\frac{1}{2}}\tau \Big(\int_{T(Q)} |a(t,\e)|^2 \frac{dtd\e}{\e}\Big)^\frac{1}{2}\lesssim 1.
\end{split}
\end{equation*}
Then by the duality $T_1^c(\R,\M)^*=T_\infty^c(\R,\M)$ (see \cite{XX18}), we have $T_{\infty}^{c}  (\mathbb{R}^{d},\mathcal{M} )\subset  T_{1,at}^{c}  (\mathbb{R}^{d},\mathcal{M} )^{*} $.

Now let $Q$  be a cube in $\R$ with $|Q|\leq 1$. If $f\in L_{1}\big(\mathcal{M};L_{2}^{c}  (T  (Q ),\frac{dsd\varepsilon}{\varepsilon} )\big)$, then
\[
a=|Q|^{-\frac{1}{2}}\| f\|^{-1}_{L_{1}\big(\mathcal{M};L_{2}^{c}  (T  (Q ),\frac{dsd\varepsilon}{\varepsilon} )\big)}f
\]
is a $T_1^c$-atom supported in $T(Q)$. Hence, 
\[
\| f\|_{T_{1,at}^c}\leq | Q|^{\frac{1}{2}} \| f\|_{L_{1}\big(\mathcal{M};L_{2}^{c}  (T  (Q ),\frac{dsd\varepsilon}{\varepsilon} )\big)}.
\]
Thus, $L_{1}\big(\mathcal{M};L_{2}^{c}  (T  (Q ),\frac{dsd\varepsilon}{\varepsilon} )\big)\subset T_{1,at}^{c}  (\mathbb{R}^{d},\mathcal{M} )$
for every cube $Q$. Therefore, every continuous functional $\ell$ on $T_{1,at}^{c}$
induces a continuous functional on $L_{1}\big(\mathcal{M};L_{2}^{c}  (T  (Q ),\frac{dsd\varepsilon}{\varepsilon} )\big)$
with norm smaller than or equal to $  | Q  |^{\frac{1}{2}}  \|  \ell \|  _{  (T_{1,at}^{c} )^{*}}$. Let $Q_0$ be the cube centered at the origin with side length 1 and  $Q_m= Q_0+ m$ for each $m\in \mathbb{Z}^d$. Then $\R=\cup_{m\in \mathbb{Z}^d}Q_m$. 
Consequently, we can choose a sequence of functions $(g_m)_{m\in \mathbb{Z}^d}$ 
such that
\[
\ell  (a )=\tau\int_{T(Q_m)} a(s,\e) g_m^{*}(s,\e)\frac{dsd\varepsilon}{\varepsilon},\quad \forall \,T_{1}^{c}\text{-atom } a \text{ with }\supp a\subset T(Q_m),
\]
and
\[
  \|  g_m\|  _{L_{\infty}\big(\mathcal{M};L_{2}^{c}  (T  (Q_m ),\frac{dsd\varepsilon}{\varepsilon} )\big)}\leq \|  \ell \|  _{  (T_{1,at}^{c} )^{*}}.
\]
Let $g(s,\e)=g_m(s,\e)$ for  $(s,\e)\in T(Q_m)$. Then, we have 
\[
\ell  (a )=\tau\int_S a(s,\e) g^{*}(s,\e)\frac{dsd\varepsilon}{\varepsilon},\quad \forall \,T_{1}^{c}\text{-atom }a.
\]
It follows that, for any cube $Q$ with $|Q|\leq 1$,
$
 g |_{T(Q)}\in L_{\infty}\big(\mathcal{M};L_{2}^{c}  (T(Q),\frac{dsd\varepsilon}{\varepsilon} )\big)
$ and 
\[
  \| g |_{T(Q)} \|  _{L_{\infty}\big(\mathcal{M};L_{2}^{c}  (T  (Q_m ),\frac{dsd\varepsilon}{\varepsilon} )\big)}\leq |Q|^\frac{1}{2} \|  \ell \|  _{  (T_{1,at}^{c} )^{*}}, 
\]
which implies $g\in T_\infty^c(\R,\M)$. Hence, $  T_{1,at}^{c}  (\mathbb{R}^{d},\mathcal{M} )^{*}\subset T_{\infty}^{c}  (\mathbb{R}^{d},\mathcal{M} )$.  Therefore, $T_{\infty}^{c}  (\mathbb{R}^{d},\mathcal{M} )= T_{1,at}^{c}  (\mathbb{R}^{d},\mathcal{M} )^{*}$ with equivalent norms. Finally, by the density of $ T_{1,at}^{c}  (\mathbb{R}^{d},\mathcal{M} )$ in $ T_{1}^{c}  (\mathbb{R}^{d},\mathcal{M} )$, we get the desired equivalence.
\end{proof}

The following Lemma shows the connection between $T_p^c(\R,\M)$ and $\h_p^c(\R,\M)$. The proof is modelled on the classical argument of \cite[Theorem 6]{CMS1985}.

\begin{lem}
\label{lem: bdd projection}
Fix a Schwartz function $\Phi$ on $\R$ satisfying:
\begin{equation}\label{eq: Phi condition}
\begin{cases}
\Phi\text{ is supported in the cube with side length 1 and centered at the origin};\\
 \int_{\R} \Phi(s) ds=0;\\
\Phi \text{ is nondegenerate in the sense of }\eqref{nondegenerate-Phi} .
\end{cases}
\end{equation}
Let $\pi_\Phi$ be the map given by
\[
\pi_{\Phi}(f)(s)=\int_{0}^1\int_{\R}\Phi_{\varepsilon}(s-t)f(t,\varepsilon)\frac{dt d\varepsilon}{\varepsilon},\quad s\in \mathbb{R}^d.
\]
Then $\pi_\Phi$ is bounded from $T_p^c(\R,\M)$ to $\h_p^c(\R,\M)$ for any $1\leq p<\infty$.
\end{lem}

\begin{proof}
For any $1<p<\infty$, let $q$ be its conjugate index. By Theorem \ref{dual-hardy}, it suffices to estimate $\tau \int \pi_{\Phi}(f)(s)g^*(s)ds$, for any $g\in \h_q^c(\R,\M)$.  Note that
\begin{equation*}
\begin{split}
\tau \int_{\R} \pi_{\Phi}(f)(s)g^*(s)ds &= \tau \int_{\R}  \int_0^1\Phi_{\varepsilon}(s-t)f(t,\varepsilon)\frac{dt d\varepsilon}{\varepsilon}g^*(s)ds\\
&=  \tau \int_{\R}  \int_0^1 f(t,\varepsilon) (\widetilde\Phi_{\varepsilon}*g)^*(t)\frac{d\varepsilon dt}{\varepsilon},
\end{split}
\end{equation*}
where $\widetilde \Phi(s)=\overline{\Phi(-s)}$. Then by the H\"older inequality,
 \begin{equation*}
\begin{split}
\Big| \tau \int_{\R}  \pi_{\Phi}(f)(s)g^*(s)ds\Big|  &=  \frac{1}{c_d}\Big| \tau \int_{\R}  \int_0^1\int_{B(s,\e)} f(t,\varepsilon) (\widetilde\Phi_{\varepsilon}*g)^*(t)\frac{d\varepsilon dt}{\varepsilon^{d+1}}ds\Big|\\
&   =   \frac{1}{c_d} \Big|\tau \int_{\R} \int_{\widetilde\Gamma} f(s+t,\varepsilon)(\widetilde\Phi_{\varepsilon}*g)^*(s+t)\frac{d\varepsilon dt}{\varepsilon^{d+1}}ds\Big|\\
 & \lesssim  \|A^c(f)\|_p\|  s_{\widetilde \Phi}^c(g)\| _q\\
& \lesssim   \|  f\| _{T_p^c}\|  g\| _{ \h_q^c}.
\end{split}
\end{equation*}

Now we deal with the case $p=1$. The argument below is based on the atomic decompositions of $\h_1^c(\R,\M)$ and $T_1^c(\R,\M)$.
By Lemma \ref{lem: atomic decomp Tent}, it is enough to show that $\pi _\Phi$ maps a $T_1^c$-atom to a bounded multiple of an $\h_1^c$-atom.
Let $a$ be an atom in $T_1^c$ based on some cube $Q$ with $|Q|\leq 1$. Since $\Phi$ is supported in the unit cube, it follows from the definition of $\pi_\Phi$ that $\pi_\Phi(a)$ is supported in $2Q$. Moreover, it satisfies the moment cancellation that $\int \pi_\Phi(a)(s)ds=0$ since $\widehat{\Phi}(0)=0$. So it remains to check that $\pi_\Phi(a)$ satisfies the size estimate.  To this end, we use the Cauchy-Schwarz inequality and the Plancherel formula \eqref{eq: Planchel},  
 \begin{equation} \label{eq: atom project atom}
 \begin{split}
& \| \pi_\Phi(a) \|_{L_1(\M;L_2^c(\R))}=\tau \big( \int_{\R}  |\widehat{\pi_\Phi(a)}(\xi) |^2d\xi \big)^\frac{1}{2} \\
&  =\tau  \Big( \int_{\R} | \int_0^1 \widehat{\Phi}(\e \xi) \widehat{a}(\xi,\varepsilon)\frac{d\varepsilon}{\varepsilon}|^2 d\xi \Big)^\frac{1}{2} \\
&  \leq \tau \Big(\int_{\R} \int_0^1 |\widehat{{\Phi}}(\e \xi) |^2\frac{d\varepsilon}{\varepsilon} \int_0^1 |\widehat{a}(\xi,\varepsilon)|^2\frac{d\varepsilon}{\varepsilon} d\xi \Big)^{\frac{1}{2}}\\
& \leq  \tau  \Big(\int_{T(Q)}  |a(s,\varepsilon) |^{2}\frac{dsd\varepsilon}{\varepsilon} \Big)^{\frac{1}{2}}\leq  |Q |^{-\frac{1}{2}}.
 \end{split}
\end{equation}
Therefore we obtain the boundedness of $\pi_\Phi$ from $T_{1,at}^c(\R,\M)$ to $\h_{1,{\rm at}}^c(\R,\M)$.
\end{proof}

Now we are able to refine the smoothness of the atoms given in Theorem \ref{thm:atomic h1}.

\begin{thm}\label{thm:smooth decomp}
 For any $f\in L_{1}  (\mathcal{M};\mathrm R_{d}^{c} )+L_{\infty}  (\mathcal{M};\mathrm R_{d}^{c} )$,  $ f$ belongs to $ \h_1^c(\mathbb{R}^d,\mathcal{M})$ if and only if it can be represented as
\begin{equation}\label{eq: 16}
f=\sum_{j=1}^{\infty}(\mu_j b_j+\lambda_jg_j),
\end{equation}
where
\begin{itemize}
\item the $b_j$'s are infinitely differentiable atoms supported in $2Q_{0,j}$ with $|Q_{0,j}|=1$. For any multiple index $\gamma \in \mathbb{N}^d_0$, there exists a constant $C_\gamma$ which depends on $\gamma$ satisfying  
\begin{equation}\label{eq: derivative of b}
\tau \big(\int_{2Q_{0,j}} |D^{\gamma}b_j (s)|^2 ds\big)^\frac{1}{2}\leq C_\gamma;
\end{equation}
\item  the $g_j$'s are infinitely differentiable atoms  supported in $2 Q_j$ with $| Q_j | <1$, and such that 
\begin{equation}\label{eq: improve atom}
\tau \big(\int_{2 Q_j} |g_j (s)|^2 ds\big)^\frac{1}{2}  \lesssim | Q_j|^{-\frac{1}{2}} \quad \text{and}\quad \int _{2 Q_j}g_j (s)ds=0;
\end{equation}
 \item  the coefficients $\mu_j$ and $\lambda_j$ are complex numbers such that
\begin{equation}\label{eq: 17}
\sum_{j=1}^{\infty}(  |\mu_j |+  |\lambda_j |)<\infty.
\end{equation}
\end{itemize}
Moreover, the infimum of \eqref{eq: 17} with respect to all admissible representations gives rise to an equivalent norm on $\h_1^c(\mathbb{R}^d,\mathcal{M})$.
\end{thm}

\begin{proof}
Since the $b_j$'s and $g_j$'s are atoms in $\h_{1}^{c}(\R,\M)$, 
it suffices to show that any $f\in \h_{1}^{c}(\R,\M)$ can be represented as in \eqref{eq: 16} and \[
\sum_{j=1}^{\infty}(  |\mu_j |+  |\lambda_j |)\lesssim \|  f\| _{\h_1^c}.
\]
To begin with, we construct a smooth resolution of the unit on $\R$.
Let $\kappa$ be a radial, real and infinitely differentiable function on $\R$ which is supported in the unit cube centered at the origin. Moreover, we assume that $\widehat{\kappa}(0)>0$. We take $\widehat{\Phi} =|\cdot|^2\widehat{\kappa}$, which can be normalized as:
\[
\int_0^\infty   \widehat{\Phi} (\e\xi)^2
\frac{d\e}{\e}=1, \quad \xi\in \R\backslash \{0\}.
\]
And we  define
\begin{equation}\label{eq: basic 3}
\widehat{\phi }(\xi)=1-\int_0^1  \widehat{\Phi}(\e\xi) ^2
\frac{d\e}{\e}, \quad \xi \in \R.
\end{equation}
 By the Paley-Wiener theorem, $\widehat{\Phi}$ can be extended to an analytic function $\widehat{\Phi}(z)$ of $d$ complex variables $z=(z_1,\dots, z_d)$, and for any $\lambda >0$, there exists a constant $C_\lambda$ such that 
\[
|\widehat{\Phi}(z)|\leq C_\lambda e^{(\frac{\lambda}{4}+\frac{\sqrt{d}}{2})|\xi_2|}(|\xi_1|^2+|\xi_2|^2)
\]
holds for any $z=\xi_1+\rm{i} \xi_2$. Therefore, 
\begin{equation*}
\begin{split}
\int_0^1| \widehat{\Phi} (\e z)|^2
\frac{d\e}{\e}  & \leq   C_\lambda^2\int_0^1 e^{\e(\frac{\lambda}{2}+\sqrt{d})|\xi_2|}\e^3 d\e \cdot(|\xi_1|^2+|\xi_2|^2)^2\\
 & \leq    C_\lambda^2\int_0^1 \e^3 d\e \cdot  e^{(\frac{\lambda}{2}+\sqrt{d}) |\xi_2|} (|\xi_1|^2+|\xi_2|^2)^2\\
  & \leq   C_\lambda^2 e^{(\frac{\lambda}{2}+\sqrt{d} )|\xi_2|}(1+|\xi_1|^2)^2(1+|\xi_2|^2)^2\\
& \leq    C_\lambda^2 e^{(\lambda+2\sqrt{d})|\xi_2|}(1+|\xi_1|)^4.
\end{split}
\end{equation*}
Now applying the Paley-Wiener-Schwartz theorem to distributions, we obtain that $\phi$ is a distribution with support in $\{s \in \R: |s |\leq 2\sqrt{d}\}$. On the other hand, if we define its value at the origin as $0$, the function $\int_{0}^{1}\widehat{\Phi}  (\varepsilon\cdot )^2\frac{d\varepsilon}{\varepsilon}$  is an infinitely differentiable function on $\R$, which ensures that $\phi $ is a Schwartz function. Thus, $\supp \phi\subset \{s \in \R: | s |\leq 2\sqrt{d}\}$.  By \eqref{eq: basic 3}, we arrive at the following decomposition of $f$:
 \begin{equation}\label{eq:decomp f}
  f=\phi  *f+\int_0^1\Phi_\varepsilon*{\Phi}_\e*f\frac{d\varepsilon}{\varepsilon}.
 \end{equation}

We first deal with  $\phi *f$. By Theorem \ref{thm:atomic h1}, we obtain an atomic decomposition of $f$:
\begin{equation}\label{eq:atom}
f=\sum_j \widetilde\mu_j a_j,
\end{equation}
where the $a_j$'s are $\h_1^c$-atoms and $\sum_j {  |\widetilde\mu_j|} \lesssim \|  f \| _{\h_1^c}$. Thus,
$$\phi *f=\sum_j \widetilde\mu_j\, \phi *a_j.$$
 We now show that every $\phi *a_j$ can be decomposed into smooth atoms supported in cubes with side length two. Let $\mathcal{X}_0$ be a nonnegative infinitely differentiable function on $\R$ such that $\supp \mathcal{X}_0\subset 2Q_0$ (with $Q_0$ the unit cube centered at the origin), and $\sum_{k\in\mathbb{Z}^d} \mathcal{X}_0(s-k)=1$ for every $s\in \R$. See \cite[Section~VII.2.4]{Stein1993} for the existence of such $\mathcal{X}_0$. Set $\mathcal{X}_k = \mathcal{X}_0(\cdot-k)$. Then $\mathcal{X}_k$ is supported in the cube $2Q_k=k+2Q_0$, and all $\mathcal{X}_k$'s form a smooth resolution of the unit:
\begin{equation}\label{eq: unit resolution}
1=\sum_{k\in\mathbb{Z}^d}\mathcal{X}_k(s), \quad \forall \, s\in \mathbb{R}^d. 
\end{equation}
 Take $a$ to be one of the atoms in \eqref{eq:atom} supported in $Q$.  Since $\phi $ has compact support, i.e. there exists a constant $C$ such that $\supp \phi \subset CQ_0$, then $\phi *a$ is supported in $(C+1)Q_0$. Thus, we get the decomposition
 \[
 \phi *a=\sum_{k=1}^N b_k
\quad \mbox{with }\, b_k=\mathcal{X}_k\cdot (\phi *a),
 \]
where $N$ is a positive integer depending only on the dimension $d$ and $C$. For any $\beta,\gamma\in \mathbb{N}_0^d$, denote $\beta\leq \gamma$ if $\beta_j\leq \gamma_j$ for every $1\leq j \leq d$. Then, by the Cauchy-Schwarz inequality, for any $k$,
\begin{equation*}
\begin{split}
\tau   \big(\int_{\R}  |D^\gamma b_k(s) |^2ds \big)^\frac{1}{2}  &\lesssim \sum_{ \beta \leq \gamma  }   \tau   \big(\int_{2Q_k}  |D^\beta \phi *a(s)\cdot D^{\gamma-\beta}\mathcal{X}_k(s) |^2ds \big)^\frac{1}{2}\\
&  \lesssim  \sum_{ \beta \leq \gamma  }  \tau  \big( \int_{2Q_k}   | \int_{\R} D^\beta \phi (s-t)a(t)dt |^2ds \big)^\frac{1}{2}\\
& \leq   \sum_{ \beta \leq \gamma  }     \big( \int_{Q} \int_{2Q_k}   | D^\beta \phi (s-t) |^2 ds dt \big)^\frac{1}{2}\cdot \tau   \big(\int _Q  | a(t) |^2dt \big)^\frac{1}{2}\\
& \lesssim    | Q |^\frac{1}{2}\tau   \big(\int_Q  | a(t) |^2dt \big)^\frac{1}{2}\leq 1,
\end{split}
\end{equation*}
where the relevant constants depend only on $\gamma$, $\phi$ and $\mathcal{X}_0$.
Thus, we have proved that $\phi *f$ can be decomposed as follows:
\[
\phi *f=\sum_{j}\mu_jb_j,
\]
with $b_j$ as desired. Furthermore, $\sum_{j}|\mu_j| \lesssim \|  f \|  _{\h_1^c}$.

Now it remains to deal with the second term on the right hand side of \eqref{eq:decomp f}. It follows from the definition of the tent space and Theorem \ref{thm main1} that ${\Phi}_\varepsilon *f\in T_{1}^{c}(\mathbb{R}^{d},\mathcal{M})$ and
\[
\|  \phi *f \| _{1}+\| {\Phi}_{\varepsilon}*f\| _{T_{1}^{c}}\lesssim \|  f \|  _{\h_1^c}.
\]
By Lemma \ref{lem: atomic decomp Tent}, we decompose ${\Phi}_\e*f$ as follows:
\begin{equation}
{\Phi}_\e*f(s)=\sum_{j=1}^{\infty}\lambda_{j}\widetilde{a}_{j}  (s,\varepsilon )\quad \text{with}\quad  \sum_{j=1}^{\infty}  |\lambda_{j} |\lesssim  \|  {\Phi}_\varepsilon * f \|  _{T_{1}^{c}},
\label{eq: tent decomp 1}
\end{equation}
where the $\widetilde{a}_j$'s are $T_1^c$-atoms based on cubes with side length less than or equal to 1.
 For each $\widetilde{a}_j(s,\varepsilon)$ based on $Q_j$ in \eqref{eq: tent decomp 1}, we set
\begin{equation}
g_{j}(s)=\int_{0}^{1}\Phi_{\varepsilon}*\widetilde{a}_{j}(s,\varepsilon)\frac{d\varepsilon}{\varepsilon}=\pi_{\Phi}\widetilde{a}_j(s), \quad \forall s\in \R.
\label{eq: proj of atom}
\end{equation}
We observe from the proof of Lemma \ref{lem: bdd projection} that $g_j$ is a bounded multiple of an $\h_1^c$-atom supported in $2Q_j$ with vanishing mean. Moreover,  $g_j$ is infinitely differentiable. Thus, $g_j$ satisfies \eqref{eq: improve atom} with relevant constant depending only on $\Phi$. Combining \eqref{eq: tent decomp 1} and \eqref{eq: proj of atom}, we obtain  the decomposition
\begin{equation*}
\int_{0}^{1}\Phi_{\varepsilon}*{\Phi}_\e*f \frac{d\varepsilon}{\varepsilon}
=\sum_{j=1}^{\infty}\lambda_{j}g_{j},
\end{equation*}
with $\sum_{j=1}^{\infty}  |\lambda_{j} |\lesssim  \|  f \|  _{\h_{1}^{c}}$.
 The proof is complete.
\end{proof}

 \subsection{Atomic decomposition for $F_1^{\alpha ,c }(\R,\M)$}
 
 Now we turn to the general space $F_1^{\alpha,c}(\R,\M)$.
For every $l=  (l_{1},\cdots,l_{d} )\in\mathbb{Z}^{d}$, $\mu \in \mathbb{N}_0$, we define  $Q_{\mu, l}$ in $\R$ to be the cubes centered at $2^{-\mu}l$, and with side length $2^{-\mu}$. For instance, $Q_{0,0} = [-\frac 1 2 , \frac 1 2 )^d$ is the unit cube centered at the origin. Let $\mathbb{D}_d$ be the collection of all the cubes $Q_{\mu, l}$ defined above. We write $(\mu , l)\leq (\mu ', l')$ if
 $$
 \mu \geq \mu ' \quad \mbox{and}\quad Q_{\mu, l}\subset 2Q_{\mu' ,l'}.
 $$
 For $a \in \mathbb{R}$, let $a_+=\max\{a,0\}$ and $[a]$ the largest integer less than or equal to $a$. Recall that $|\gamma|_1 =\gamma_1+\cdots+\gamma_d$ for $\gamma\in \mathbb{N}_0^d$, $s^{\beta}=s_1^{\beta_1}\cdots s_d^{\beta_d}$ for $s\in \R$, $\beta \in \mathbb{N}_0^d$ and $J^\alpha$ is the Bessel potential of order $\alpha$.

 \begin{defn}\label{def:smooth T atom}
Let $\alpha\in\mathbb{R}$, and let $K$ and $L$ be two integers such that
$$
K\geq ([\alpha]+1)_+ \quad \mbox{and} \quad L\geq \max{\{[-\alpha],-1\}}.
$$
\begin{enumerate}[\rm(1)]
\item A function $b\in L_1\big(\mathcal{M};L_2^c(\R)\big)$  is called an $(\alpha, 1)$-atom if
\begin{itemize}
\item $\supp b\subset 2 Q_{0,k}$;
\item  $\tau  \big(\int_{\R}  |D^{\gamma}b(s) |^{2}ds \big)^{\frac{1}{2}}\leq1$, $\,\,\forall\gamma \in  \mathbb{N}_0^d\,, \,\, |\gamma |_1\leq K$.
\end{itemize}

\item Let $Q=Q_{\mu, l}\in \mathbb{D}_d$, a function $a\in L_1\big(\mathcal{M};L_2^c(\R)\big)$  is called an $(\alpha, Q)$-subatom if
\begin{itemize}
\item $\supp a\subset 2Q$;
\item $\tau   \big(\int_{\R}  |D^{\gamma}a(s) |^{2}ds  \big)^{\frac{1}{2}}\leq  |Q |^{\frac{\alpha}{d}-\frac{  |\gamma |_1}{d}}$,
$\,\,\forall\gamma \in  \mathbb{N}_0^d\,, \,\, |\gamma |_1\leq K$;
\item $\int_{\mathbb{R}^{d}}s^{\beta}a  (s )ds=0$, $\,\,\forall\beta \in  \mathbb{N}_0^d\,, \,\, |\beta |_1\leq L$.
\end{itemize}

\item A function $g\in L_1\big(\mathcal{M};L_2^c(\R)\big)$ is called an $(\alpha, Q_{k,m})$-atom if
\begin{equation}\label{eq: Q atom-sc}
\tau  \big(\int_{\R}  | J^\alpha g(s) |^{2}ds  \big)^{\frac{1}{2}}\lesssim  |Q_{k,m} |^{-\frac 1 2 }\quad \text{and} \quad g=\sum_{  (\mu,l )\leq (k,m )}d_{\mu, l}a_{\mu, l},
\end{equation}
for some $k\in\mathbb{N}_0$ and $m\in\mathbb{Z}^{d}$, where the $a_{\mu ,l}$'s
are $(\alpha,Q_{\mu ,l})$-subatoms and the $d_{\mu ,l}$'s are complex numbers
such that $$   \big(\sum_{  (\mu,l )\leq  (k,m )}  |d_{\mu ,l} |^{2}  \big)^{\frac{1}{2}}\leq  |Q_{k,m} |^{-\frac{1}{2}}.$$
\end{enumerate}
\end{defn}

\begin{rmk}
If $L<0$, the third assumption of an $(\alpha, Q)$-subatom means that no moment cancellation is required. In the second assumption of an $(\alpha, 1)$-atom $b$ and that of an $(\alpha, Q)$-subatom $a$, it is tacitly assumed that $b$ and $a$ have  derivatives up to order $K$. For such $a$, we can define a norm by
$$\| a \|_{*}=\sup_{|\gamma|_1 \leq K}\|D^\gamma a\|_{L_1\big(\M;L_2^c(\R)\big)}.$$
Then the convergence in \eqref{eq: Q atom-sc} is understood in this norm, and we will see that the atom $g$ in \eqref{eq: Q atom-sc} belongs to $F_1^{\alpha,c}(\R,\M)$.
\end{rmk}

\begin{rmk}
In the classical case, the first size estimate in \eqref{eq: Q atom-sc} is not necessary. In other words, if $g=\sum_{  (\mu,l )\leq (k,m )}d_{\mu, l}a_{\mu, l}$ with the subatoms $a_{\mu, l}$'s and the complex numbers $d_{\mu, l}$'s such that $  \big(\sum_{  (\mu,l )\leq  (k,m )}  |d_{\mu ,l} |^{2}  \big)^{\frac{1}{2}}\leq  |Q_{k,m} |^{-\frac{1}{2}}$, then $g$ satisfies that estimate in \eqref{eq: Q atom-sc} automatically. We refer the readers to \cite{Tri1992} for more details. Unfortunately, in the current setting, we are not able to prove this estimate, so we just add it in \eqref{eq: Q atom-sc} for safety.
\end{rmk}

The following is our main result on the atomic decomposition of $F_{1}^{\alpha, c}  (\mathbb{R}^{d},\mathcal{M} )$. The idea comes from \cite[Theorem 3.2.3]{Tri1992}, but many techniques used are different from those of \cite[Theorem 3.2.3]{Tri1992} due to noncommutativity. 

\begin{thm}\label{thm: atomic decop T_L}
Let $\alpha\in\mathbb{R}$ and $K$, $L$ be two integers fixed as in Definition \ref{def:smooth T atom}. Then any
$f\in F_{1}^{\alpha, c}  (\mathbb{R}^{d},\mathcal{M} )$ can be represented as
\begin{equation}\label{eq:decomp T-L}
f=\sum_{j=1}^{\infty}\big(\mu_j b_j+\lambda_{j}g_{j}\big),
\end{equation}
where the $b_j$'s are $(\alpha,1)$-atoms, the $g_{j}$'s are $(\alpha,Q)$-atoms, and $\mu_j$, $\lambda_{j}$ are complex numbers
with
\begin{equation}\label{eq:mu+lambda}
\sum_{j=1}^{\infty}  (  |\mu_j |+  |\lambda_{j} | )<\infty.
\end{equation}
 Moreover, the infimum of \eqref{eq:mu+lambda} with respect to all admissible representations
is an equivalent norm in $F_{1}^{\alpha,c}   (\mathbb{R}^{d},\mathcal{M} )$.
\end{thm}

\begin{proof}
\emph{Step 1.} First, we show that any $f\in F_{1}^{\alpha, c}  (\mathbb{R}^{d},\mathcal{M} )$ admits the representation \eqref{eq:decomp T-L} and
$$
\sum_{j=1}^{\infty}  (  |\mu_j |+  |\lambda_{j} | )\lesssim \|f\|_{F_{1}^{\alpha, c}}.
$$
The proof of this part is similar to the proof of Theorem \ref{thm:smooth decomp}. Let $\kappa$ be the Schwartz function defined in the proof of Theorem \ref{thm:smooth decomp}. We take $\widehat{\Phi} =|\cdot|^N\widehat{\kappa}$, where $N$ is  a positive even integer such that $N\geq \max\{L,\alpha\}$, then $\Phi$ can be normalized as follows:
$$
\int_0^\infty \widehat{\Phi} (\e\xi)^2 \frac{d\e}{\e}=1, \quad \forall \, \xi \in \R\setminus\{0\}.
$$
Since $-\alpha+N\geq 0$, both $\sum_{j=-\infty}^{\infty} {(J_{-\alpha}\widehat{\Phi} )}(2^{-j}\xi)^2 $ and $\sum_{j=-\infty}^{\infty} {(J_{-\alpha}\widehat{\Phi} )}(2^{-j}\xi)^2 $ are rapidly decreasing and infinitely differentiable functions on $\R\setminus\{0\}$. So we have
\begin{equation}\label{eq: normalize J Phi}
\sum_{j=-\infty}^{\infty} {(J_{-\alpha}\widehat{\Phi} )}(2^{-j}\xi)^2 < \infty
\end{equation}
and 
\begin{equation}\label{eq: normalize I Phi}
\sum_{j=-\infty}^{\infty} {(I_{-\alpha}\widehat{\Phi} )}(2^{-j}\xi)^2  <\infty .
\end{equation}
Applying the Paley-Wiener-Schwartz theorem, we get a compactly supported function $\Phi_0 \in \mathcal{S}$ such that
$$
\widehat{\Phi}_0(\xi)=1-\int_0^1 \widehat{\Phi} (\e\xi)^2
\frac{d\e}{\e}.
$$
 Denote by $\Phi_\varepsilon$ the  Fourier inverse transform of ${\Phi}(\e\cdot)$. 
For any $f\in F_{1}^{\alpha,c}(\mathbb{R}^{d},\mathcal{M})$, we have
\begin{equation}\label{eq: T-L f decomp}
f=\Phi_0*f+\int_0^1\Phi_\varepsilon*{\Phi}_\e*f\frac{d\varepsilon}{\varepsilon}.
\end{equation}
Let us deal with the two terms on the right hand side of \eqref{eq: T-L f decomp} separately.

The term $\Phi_0*f$ is easy to treat. If $\alpha\geq 0$, Proposition \ref{prop: F} ensures that $F_1^{\alpha,c}(\R,\M)\subset \h_1^c(\R,\M)$. Then we can repeat the first part of the proof of Theorem \ref{thm:smooth decomp}: for any $f\in {F_1^{\alpha,c}(\R,\M)}$, $\Phi_0*f$ admits the decomposition
$$
\Phi_0*f=\sum_j\mu_j b_j ,
$$
with 
$$
\sum_j |  \mu_j|  \lesssim \|  f \| _{\h_1^c}\lesssim \|  f\|  _{F_1^{\alpha,c}},
$$
where the $b_j$'s, together with their derivatives $D^\gamma b_j$'s, satisfy \eqref{eq: derivative of b} with some constants $C_\gamma$ depending on $\gamma$. When $K$ is fixed, we can normalize the $b_j$'s by ${\max_{|\gamma|_1\leq K} |C_\gamma|}$, then the new $b_j$'s are $(\alpha,1)$-atoms.
If $\alpha <0$, by Propositions \ref{prop: F} and \ref{prop: lifting}, we have $J^{[\alpha]} f\in F_1^{\alpha-[\alpha],c}\subset \h_1^c$. Then $J^{[\alpha]}\Phi_0* f$ admits the decomposition
 $$
J^{[\alpha]} \Phi_0*f =\sum_{j}\mu_jb_j,
$$
with $\sum_{j}|\mu_j| \lesssim  \|J^{[\alpha]}f \|  _{\h_1^{c}}\lesssim \|  f\| _{F_1^{\alpha,c}}$.
Then
$$
\Phi_0*f=\sum_{j}\mu_j  J^{-[\alpha]}b_j.
$$
If $-[\alpha]$ is even, it is obvious that $\supp J^{-[\alpha]}b_j \subset \supp b_j$. Moreover, for any $\gamma \in \mathbb{N}_0^d$ such that $|\gamma |_1\leq K$, we have
$$
 \tau  (\int_{\R}  |D^{\gamma}J^{-[\alpha]} b_j(s) |^{2}ds )^{\frac{1}{2}}\lesssim \sum_{|\gamma'|_1\leq K-2[\alpha]} \tau  (\int_{\R}  |D^{\gamma'} b_j(s) |^{2}ds )^{\frac{1}{2}}\leq C_K.
$$
 We normalize $J^{-[\alpha]}b_j$ by this constant $C_K$ depending on $K$, then we can make it an $(\alpha,1)$-atom.
When $-[\alpha]$ is odd, it suffices to replace $[\alpha]$ in the above argument by $[\alpha]-1$, and then we get the desired decomposition. 

\emph{Step 2.} Now we turn to the second term on the right hand side of \eqref{eq: T-L f decomp}. It follows from Theorem \ref{thm:general Lusin cont} and the definition of the tent space that  $\varepsilon^{-\alpha}{\Phi}_\e*f\in T_{1}^{c}(\mathbb{R}^{d},\mathcal{M})$ and
$$
\|  \varepsilon^{-\alpha}{\Phi}_\e*f\| _{T_{1}^{c}}\lesssim \|  f \|  _{F_1^{\alpha,c}}.
$$
By Lemma \ref{lem: atomic decomp Tent}, we have
\begin{equation}
\varepsilon^{-\alpha}{\Phi}_\e*f(s)=\sum_{j=1}^{\infty}\lambda_{j}b_{j}  (s,\varepsilon ),
\label{eq: tent decomp}
\end{equation}
where the $b_j$'s are $T_1^c$-atoms based on the cubes $Q_j$'s with $|Q_j|\leq 1$. Then, if we set $a_j(s,\e)=\e^{\alpha}b_j(s,\e)$, we obtain
$$
{\Phi}_\e*f(s)=\sum_{j=1}^{\infty}\lambda_{j}a_{j}  (s,\varepsilon )
$$
and
\begin{equation}\label{eq: tent decomp coefficient}
\sum _{j=1}^{\infty}  |\lambda_{j} |\lesssim  \|  \e^{-\alpha}\Phi_\varepsilon * f \|  _{T_{1}^{c}}\lesssim \| f\|_{F_1^{\alpha,c}}.
\end{equation}
In particular, 
\begin{equation}\label{eq: a_j e}
\supp a_j\subset T(Q_j)\quad \text{ and}\quad \tau\Big(\int_{T(Q_j)}\e^{-2\alpha} |a_j(s,\e)|^2\frac{dsd\e}{\e}\Big)^\frac{1}{2}\leq |Q_j|^{-\frac{1}{2}}.
\end{equation}
 For every $a_j$, we set
\begin{equation}\label{eq: def of g atom}
g_{j}(s)=\pi _{\Phi}(a_j)(s)=\int_{0}^{1}\Phi_{\varepsilon}*a_{j}(s,\varepsilon)\frac{d\varepsilon}{\varepsilon}.
\end{equation}\label{eq: a_j}
Then $\supp g_j\subset 2Q_j$. We arrive at the decomposition
\begin{equation*}
 \int_{0}^{1}\Phi_\e*{\Phi}_{\varepsilon}*f\frac{d\varepsilon}{\varepsilon}
  =  \sum_{j=1}^{\infty}\lambda_{j}g_{j}.
\end{equation*}

Now we show that every $g_j$ is an $(\alpha, Q_{k_j,m_j})$-atom. Firstly, for any $Q_j$, there exist  $k_j\in \mathbb{N}_0$ and $s\in \R$ such that
$$
2^{-k_j-1}\leq l(Q_j) \leq 2^{-k_j} \quad \text{and}\quad c_{Q_j}=l(Q_j)s.
$$
Take $m_j=[s]\in \mathbb{Z}^d$, where $[s]=([s_1],\cdots,[s_d])$. Then,  we easily check that 
 \begin{equation}\label{eq: cube inclusion}
Q_j\subset 2Q_{k_j,m_j}, \quad Q_{k_j,m_j}\in \mathbb{D}_d.
\end{equation}
Next, by the argument similar to that in \eqref{eq: atom project atom} and by \eqref{eq: a_j e}, we have
 \begin{equation*}
\tau  \big(\int_{\R}  | I^\alpha \pi_\Phi(a_j)(s) |^{2}ds \big)^{\frac{1}{2}}  \lesssim  \tau  \Big(\int_{T(Q_j)} \e^{-2\alpha} | a_j(t,\varepsilon) |^{2}\frac{dtd\varepsilon}{\varepsilon} \Big)^{\frac{1}{2}}\leq |Q_j |^{-\frac{1}{2}}\lesssim |Q_{k_j,m_j} |^{-\frac{1}{2}}.
\end{equation*}
If $\alpha\leq 0$, it is clear that 
$$
\tau  \big(\int_{\R}  | J^\alpha \pi_\Phi(a_j)(s) |^{2}ds \big)^{\frac{1}{2}} \leq \tau  \big(\int_{\R}  |  I^\alpha  \pi_\Phi(a_j)(s) |^{2}ds \big)^{\frac{1}{2}}\lesssim  |Q_j |^{-\frac{1}{2}}\lesssim |Q_{k_j,m_j} |^{-\frac{1}{2}}.
$$
If $\alpha>0$, we have
 \begin{equation*}
 \begin{split}
\tau \big(\int_{\R}  | J^\alpha \pi_\Phi(a_j)(s) |^{2}ds \big)^{\frac{1}{2}}  & \lesssim \tau  \big(\int_{\R}  |  \pi_\Phi(a_j)(s) |^{2}ds \big)^{\frac{1}{2}}+\tau  \big(\int_{\R}  | I^\alpha \pi_\Phi(a_j)(s) |^{2}ds \big)^{\frac{1}{2}} \\
& \lesssim  \tau  \Big(\int_{T(Q_j)}  | a_j(t,\varepsilon) |^{2}\frac{dtd\varepsilon}{\varepsilon} \Big)^{\frac{1}{2}} + |Q_j |^{-\frac{1}{2}}\\ 
&\lesssim  \tau  \Big(\int_{T(Q_j)} \e^{-2\alpha} | a_j(t,\varepsilon) |^{2}\frac{dtd\varepsilon}{\varepsilon} \Big)^{\frac{1}{2}} + |Q_j |^{-\frac{1}{2}}\\
& \leq 2 |Q_j |^{-\frac{1}{2}}\lesssim |Q_{k_j,m_j} |^{-\frac{1}{2}}.
\end{split}
\end{equation*}
Then we get, for any $\alpha\in \mathbb{R}$,
 \begin{equation}\label{eq: size esitimate of g}
\tau  \big(\int_{\R}  | J^\alpha g_j(s) |^{2}ds \big)^{\frac{1}{2}} =\tau  \big(\int_{\R}  | J^\alpha \pi_\Phi(a_j)(s) |^{2}ds \big)^{\frac{1}{2}}  \lesssim |Q_{k_j,m_j} |^{-\frac{1}{2}}.
 \end{equation}
Finally,  we decompose the slice $T(Q_j)\cap  \{ 2^{-\mu-1}\leq \e\leq2^{-\mu} \} $
into  $(d+1)$-dimensional dyadic cubes whose projections on $\R$ belong to $\mathbb{D}_d$, and with side length $2^{-\mu}$, $\mu \in\mathbb{N}_0$. Let $\widehat{Q}$
be one of those dyadic cubes with side length $2^{-\mu }$ and $Q$ be its projection on $\R$.
Let
$$
a(s)=\int_{\widehat{Q}}\Phi _\e (s-t)a_{j}(t,\varepsilon)\frac{dtd\varepsilon}{\varepsilon}.
$$
By the support assumption of $\Phi$, it follows that 
$$
\supp a\subset 2Q,\qquad\supp a\subset 2Q_{j}\subset 4Q_{k_j,m_j}.
$$
Then
$$
\widehat{a}(\xi)=\int_{2^{-\mu}}^{2^{-\mu+1}}\widehat{\Phi}(\varepsilon \xi)  \,\mathcal{F} \big(a_{j}(\cdot,\varepsilon)\un_{Q} \big)(\xi)\frac{d\varepsilon}{\varepsilon}.
$$
Since $ D^\beta \widehat{\Phi}(0)=0$ for any $|  \beta|_1 \leq N$, we obtain
$$
\int_{\mathbb{R}^{d}}(-2\pi {\rm i}s)^{\beta}a(s)ds=D^\beta\widehat{a}(0)=0,\quad\forall \, |\beta |_1\leq L.
$$
 Furthermore, by the Cauchy-Schwarz inequality,
we have
\begin{equation*}
\begin{split}
    \tau  \big(\int  | a(s) |^{2}ds \big)^{\frac{1}{2}}
 & = \tau \Big(\int_{5Q}  \big|\int_{2^{-\mu}}^{2^{-\mu+1}}\int_{Q}\Phi_\e  (s-t)a_{j}(t,\varepsilon)\frac{dtd\varepsilon}{\varepsilon} \Big|^{2}ds \Big)^{\frac{1}{2}}\\
 & \lesssim  |Q |^{\frac{1}{2}}  \Big(\int_{2^{-\mu}}^{2^{-\mu+1}}\int_{Q}\varepsilon^{-2d} \frac{dtd\varepsilon}{\varepsilon} \Big)^{\frac{1}{2}}\cdot\tau  \Big(\int_{2^{-\mu}}^{2^{-\mu+1}}\int_{Q}  |a_{j}(t,\varepsilon) |^{2}\frac{dtd\varepsilon}{\varepsilon} \Big)^{\frac{1}{2}}\\
 & \lesssim \tau  \Big(\int_{2^{-\mu}}^{2^{-\mu+1}}\int_{Q}  |a_{j}(s,\varepsilon) |^{2}\frac{dsd\varepsilon}{\varepsilon} \Big)^{\frac{1}{2}}\\
  & \lesssim   |Q |^{\frac{\alpha}{d}}\tau  \Big(\int_{2^{-\mu}}^{2^{-\mu+1}}\int_{Q}\e^{-2\alpha}  |a_{j}(s,\varepsilon) |^{2}\frac{dsd\varepsilon}{\varepsilon} \Big)^{\frac{1}{2}}.
\end{split}
\end{equation*}
 Similarly, we have
$$
\tau  \big(\int  |D^{^{\gamma}}a(s) |^{2}ds \big)^{\frac{1}{2}}\leq C'_\gamma |Q |^{\frac{\alpha}{d}-\frac{  |\gamma |_1}{d}}\tau  \Big(\int_{2^{-\mu}}^{2^{-\mu+1}}\int_{Q}\e^{-2\alpha}  |a_{j}(s,\varepsilon) |^{2}\frac{dsd\varepsilon}{\varepsilon} \Big)^{\frac{1}{2}}.
$$
The above discussion gives
\begin{equation}\label{eq: sum form of g}
g_{j}=\sum_{(\mu,l)\leq (k_{j},m_{j})}d_{\mu ,l}^{j}a_{\mu, l}^{j},
\end{equation}
 where each $a^j_{\mu, l}$ is an $(\alpha,Q_{\mu, l})$-subatom. The normalizing factor
is given by
$$
d_{\mu ,l}^{j}= \max_{|\gamma|_1\leq K}\{C'_\gamma\}\tau  \Big(\int_{2^{-\mu}}^{2^{-\mu+1}}\int_{Q_{\mu, l}}\e^{-2\alpha}  |a_{j}(s,\varepsilon) |^{2}\frac{dsd\varepsilon}{\varepsilon} \Big)^{\frac{1}{2}}.
$$
 By the elementary fact that $\ell_{2}  (L_{1}(\mathcal{M}) )\supset L_{1}  (\mathcal{M};\ell_{2}^{c} )$,
we get
\begin{equation}\label{eq: coefficiant of g}
 \big(\sum_{(\mu,l)\leq (k_{j},m_{j})}  |d_{\mu, l}^{j} |^{2} \big)^{\frac{1}{2}}\leq C\tau  \Big(\int_{T(Q_{j})}\e^{-2\alpha}  | a_{j}(s,\varepsilon) |^{2}\frac{dsd\varepsilon}{\varepsilon} \Big)^{\frac{1}{2}}\leq C  |Q_{k_{j}, m_{j}} |^{-\frac{1}{2}},
\end{equation}
where $C$ is independent of $f$. We may assume $C=1$, otherwise, we can put $C$ in \eqref{eq: tent decomp} in the numbers $\lambda_j$, which does not change \eqref{eq: tent decomp coefficient}. In summary, \eqref{eq: cube inclusion}, \eqref{eq: size esitimate of g},  \eqref{eq: sum form of g} and  \eqref{eq: coefficiant of g} ensure that $g_j$ is an $(\alpha, Q_{k_j,m_j})$-atom.

\emph{Step 3.} Now we show the reverse assertion: if $f$ is given by \eqref{eq:decomp T-L}, then $f\in F_{1}^{\alpha, c}(\R,\M)$ and
 $$
 \|f\|_{F_{1}^{\alpha, c}}\lesssim \sum_{j=1}^{\infty}  (  |\mu_j |+  |\lambda_{j} | ).
 $$
  To this end, we have to show that every $(\alpha,1)$-atom $b$ and every $(\alpha,Q)$-atom $g$ belong to $F_{1}^{\alpha,c}  (\mathbb{R}^{d},\mathcal{M} )$ and 
$$  \|  b \|  _{F_{1}^{\alpha,c}}\lesssim 1 \quad  \text{ and}\quad   \|  g \|  _{F_{1}^{\alpha,c}}\lesssim 1.
$$
Let $b$  be an $(\alpha,1)$-atom in $F_1^{\alpha,c}(\R,\M)$. We observe that $b$ is also an atom in $\h_1^c(\R,\M)$. For $\alpha\leq 0$, by Proposition \ref{prop: F}, $\h_1^c\subset F_1^{\alpha,c}$. Then, we have $\|  b \| _{F_1^{\alpha,c}}\lesssim\|  b\| _{\h_1^{c}}\lesssim 1$. If $\alpha >0$,  by Proposition \ref{prop: lifting}, we have
$$\|  b \| _{F_1^{\alpha,c}}\approx   \|\varphi_0* b \|_1+ \sum_{i=1}^d\| D_i^K  b \| _{F_1^{\alpha-K,c}}.
$$
Note that for any $1\leq i\leq d$, $D_i^K b$ is an atom in $\h_1^c(\R,\M)$. Since $\alpha-K<0$, by Proposition \ref{prop: F}, we have
$$
\|  b \| _{F_1^{\alpha,c}}\lesssim \|\varphi_0* b \|_1+ \sum_{i=1}^d\| D_i^K  b\| _{\h_1^{c}}\lesssim 1.
$$

On the other hand, let $g$ be an $(\alpha,Q_{k,m})$-atom in the sense of Definition \ref{def:smooth T atom}. 
We may use the discrete general characterization of $F_1^{\alpha,c}(\R,\M)$ given in Theorem \ref{thm:general}, i.e.
$$
\|g\|_{F_1^{\alpha,c}}\approx \big\|  (\sum_{j= 0}^\infty2^{2j\alpha}|  \Phi_j*g | ^2)^\frac{1}{2}\big\| _{1}.
$$
We split $\sum_{j= 0}^\infty$ into two parts $\sum_{j= 0}^{k-1}$ and $\sum_{j=k}^\infty$. When $j\geq k$, by the support assumption of $\Phi$, we have $\supp \Phi_j*g\subset 5Q_{k,m}$. If $\alpha\geq 0$,
by \eqref{eq: normalize I Phi}, \eqref{eq: Q atom-sc} and the Plancherel formula \eqref{eq: Planchel}, we obtain
\begin{equation*}
\begin{split}
 \tau\big( \int_{5Q_{k,m}}\sum_{j= k}^\infty 2^{2j\alpha} | \Phi_j*g(s)|  ^2ds\big)^\frac{1}{2}
&= \tau\big( \int_{5Q_{k,m}}\sum_{j= k}^\infty | (I^{-\alpha}\Phi)_j*I^\alpha g(s)|  ^2ds\big)^\frac{1}{2}\\
& \leq  \tau\big( \int_{\R}\sum_{j= k}^\infty | (I_{-\alpha} \widehat \Phi )(2^{-j}\xi)| ^2 |I_{\alpha} \widehat g(\xi)| ^2d\xi\big)^\frac{1}{2} \\
&\lesssim  \tau\big( \int_{\R}| I_{\alpha} \widehat g(\xi)| ^2d\xi\big)^\frac{1}{2}= \tau\big( \int_{\R}| I^{\alpha} g(s)| ^2ds \big)^\frac{1}{2}\\
& \leq \tau\big( \int_{\R}| J^{\alpha} g(s)| ^2ds \big)^\frac{1}{2}\leq |Q_{m,k}|^{-\frac{1}{2}}.
\end{split}
\end{equation*}
If $\alpha<0$, by \eqref{eq: normalize J Phi},  \eqref{eq: Q atom-sc} and the Plancherel formula \eqref{eq: Planchel} again, we have
\begin{equation*}
\begin{split}
 \tau\big( \int_{5Q_{k,m}}\sum_{j= k}^\infty 2^{2j\alpha} | \Phi_j*g(s)|  ^2ds\big)^\frac{1}{2}
& \leq                                                                                                     \tau\big( \int_{5Q_{k,m}}\sum_{j= k}^\infty 2^{2j\alpha} | J^{-\alpha} \Phi_j*J^\alpha g(s)|  ^2ds\big)^\frac{1}{2}\\
& \leq  \tau\big( \int_{\R}\sum_{j= k}^\infty | (J_{-\alpha} \widehat\Phi )(2^{-j}\xi)| ^2 | J_{\alpha} \widehat g(\xi)| ^2d\xi\big)^\frac{1}{2} \\
& \lesssim  \tau\big( \int_{\R}| J_{\alpha} \widehat g(\xi)| ^2d\xi\big)^\frac{1}{2}=\tau\big( \int_{\R}| J^{\alpha} g(s)| ^2ds \big)^\frac{1}{2}\\
& \leq |Q_{m,k}|^{-\frac{1}{2}}.
\end{split}
\end{equation*}
It follows that
 $$\big\|  (\sum_{j= k}^\infty2^{2j\alpha}|  \Phi_j*g| ^2)^\frac{1}{2}\big\| _{1}\lesssim 1.$$
  In order to estimate the sum $\sum_{j= 0}^{k-1}$, we begin with a technical modification of $g$. Let
$$
\widetilde{g}=2^{k(\alpha-d)}g(2^{-k}\cdot).
$$
Then it is easy to see that $\widetilde{g}$ is an $(\alpha, Q_{0,m})$-atom.  Moreover, we have
$$
\Phi_j*g=2^{k(d-\alpha)}\Phi_{j-k}*\widetilde g(2^k\cdot),
$$
which implies that
\begin{equation}\label{eq: modification a}
\big\|  (\sum_{j= 0}^{k-1} 2^{2j\alpha}| \Phi_j*g | ^2)^\frac{1}{2}\big\| _{1}\leq \big\|  (\sum_{j= -\infty}^{-1} 2^{2j\alpha}| \Phi_j*\widetilde g | ^2)^\frac{1}{2}\big\| _{1}+2^{-k\alpha} \| (\Phi_0)_{-k}*\widetilde g \|_1,
\end{equation}
where $(\Phi_0)_{-k}$ denotes the inverse Fourier transform of the function $\Phi^{(0)}(2^k \cdot)$.
 In other words, we can assume, by translation, that the atom $g$ is based on a cube $Q$ with side length 1 and centered at the origin. Then, let us estimate the right hand side of \eqref{eq: modification a} with $g$ instead of  $\widetilde{g}$.

By the triangle inequality, we have
\begin{equation*}
\begin{split}
\big\|  (\sum_{j= -\infty}^{-1} 2^{2j\alpha}|  \Phi_j* g | ^2)^\frac{1}{2}\big\| _{1} &\leq   \sum_{j=-\infty}^{-1} 2^{j\alpha}\tau \int_{\R} |\Phi_j*g (s)|ds\\
& \leq   \sum_{j= -\infty}^{-1} \sum_{(\mu ,l)\leq(0,0)} |d_{\mu, l}|\,2^{j\alpha}\tau \int_{\R} |\Phi_j*a_{\mu,l}(s)|ds.
\end{split}
\end{equation*}
Now we estimate $2^{j\alpha}\tau \int_{\R} |\Phi_j*a_{\mu,l}(s)|ds$ for every $(\mu ,l)\leq (0,0)$. Let $M=[-\alpha]+1$. Then $M+\alpha>0$ and $L\geq M-1$. By the moment cancellation of $a_{\mu,l}$, we have
 \begin{equation*}
 \begin{split}
& \Phi_j*a_{\mu,l}(s) \\
 & =2^{jd}\int_{2Q_{\mu,l}}\big[\Phi(2^j s-2^j t)-\Phi(2^j s-2^{j} 2^{-\mu}l)\big]a_{\mu,l}(t)dt\\
 &= 2^{j(d+M)}\sum_{|\beta|_1=M}\frac{M+1}{\beta !}\int_{2Q_{\mu,l}}(2^{-\mu}l-t)^\beta \int_0^1(1-\theta)^{M} D^\beta\Phi\big(2^j s-2^j (\theta t+(1-\theta)2^{-\mu}l)\big)a_{\mu,l}(t) d\theta \,dt.
 \end{split}
 \end{equation*}
 It follows that
 \begin{equation*}
 \begin{split}
  |\Phi_j*a_{\mu,l}(s)|^2 &  \lesssim  \sum_{|\beta|_1=M}2^{2j(d+M)}\int_{2Q_{\mu,l}} \int_0^1(1-\theta)^{2M} |D^\beta\Phi \big(2^j s-2^j (\theta t+(1-\theta)2^{-\mu}l)\big)|^2d\theta dt\\
  &\;\;\;\; \cdot\int_{2Q_{\mu,l}} |t-2^{-\mu}l |^{2M} |a_{\mu,l}(t)|^2dt.
 \end{split}
 \end{equation*}
If $\Phi_j*a_{\mu,l}(s)\neq 0$, then we have $|2^j s-2^j t|\leq 1$ for some $t\in 2Q_{\mu,l}$. Hence, $\Phi_j*a_{\mu,l}(s)=0$ if $ |s-2^{-\mu}l  |>3\cdot2^{-j-1}\sqrt {d}$.
 Consequently,
 \begin{equation*}
\begin{split}
 &  \sum_{j=-\infty}^{-1} 2^{j\alpha} \tau \int_{\R} |\Phi_j*a_{\mu,l}(s)|ds \\
 & \lesssim  \sum_{j=-\infty}^{-1} 2^{j(d+M+\alpha)}\tau \big(\int_{2Q_{\mu,l}} |t-2^{-\mu}l |^{2M} |a_{\mu,l}(t)|^2dt \big)^\frac{1}{2}\\
 &\;\;\;\; \cdot \sum_{|\beta|_1=M}\int_{ |s-2^{-\mu}l  |\leq 3\cdot2^{-j-1}\sqrt {d}}\Big( \int_{2Q_{\mu,l}}  \int_0^1(1-\theta)^{2M} |D^\beta\Phi(2^j s-2^j (\theta t+(1-\theta)2^{-\mu}l))|^2d\theta dt \Big)^\frac{1}{2}ds\\
 & \lesssim \sum_{j=-\infty}^{-1} 2^{j(d+M+\alpha)}\cdot 2^{-\mu M} |Q_{\mu,l}|^{\frac{1}{2}} \tau\big (\int_{2Q_{\mu,l}}|a_{\mu,l}(t)|^2dt\big)^\frac{1}{2}\int_{ |s-2^{-\mu}l  | \leq 3\cdot2^{-j-1}\sqrt {d}} ds\\
  & \lesssim \sum_{j=-\infty}^{-1} 2^{j(d+M+\alpha)}\cdot 2^{-jd}\cdot 2^{-\mu (\alpha+M)} |Q_{\mu,l}|^{\frac{1}{2}}\\
 & = 2^{-\mu (\alpha+M)} \sum_{j=-\infty}^{-1} 2^{j(M+\alpha)} |Q_{\mu,l}|^{\frac{1}{2}}\lesssim 2^{-\mu (\alpha+M)} |Q_{\mu,l}|^{\frac{1}{2}}.
 \end{split}
\end{equation*}
Similarly, we also have 
$$
2^{-k\alpha} \tau \int_{\R} |(\Phi_0)_{-k}*a_{\mu,l}(s)|ds\lesssim 2^{-k(M+\alpha)}2^{-\mu (\alpha+M)}|Q_{\mu,l}|^{\frac{1}{2}}\leq 2^{-\mu (\alpha+M)}|Q_{\mu,l}|^{\frac{1}{2}}.
$$
Thus, by the Cauchy-Schwarz inequality, we get 
\begin{equation*}
\begin{split}
\big\|  (\sum_{j= -\infty}^{-1} 2^{2j\alpha}|  \Phi_j* g | ^2)^\frac{1}{2}\big\| _{1}&  \leq  \sum_{j=-\infty}^{-1} 2^{2j\alpha}\tau \int_{\R} |\Phi_j*g(s)|ds \\
 &  \lesssim \sum_{\mu =0}^{\infty}2^{-\mu (\alpha+M)}\big( \sum_{l} |d_{\mu,l}|^2\big)^{\frac{1}{2}} \big( \sum_{l} |Q_{\mu,l}|\big)^{\frac{1}{2}} \\
&  \lesssim \sum_{\mu =0}^{\infty} 2^{-\mu (\alpha+M)} <\infty, 
\end{split}
\end{equation*}
and 
$$
2^{-k\alpha} \| (\Phi_0)_{-k}* g \|_1 \lesssim  \sum_{\mu =0}^{\infty} 2^{-\mu (\alpha+M)} <\infty.
$$
 Therefore, $\|g\|_{F_1^{\alpha,c}}\lesssim 1$. The proof is complete.
\end{proof}

 We close this section by a very useful result of pointwise multipliers, which can be deduced  from the above atomic decomposition. Let $k\in \mathbb{N}$ and $\mathcal{L}^k(\R,\M)$ be the collection of all $\mathcal{M}$-valued functions on $\mathbb{R}^d$ such that  $D^\gamma h \in L_\infty(\mathcal{N})$ for all $\gamma$ with $0\leq |  \gamma | _1\leq k$.
\begin{cor}\label{Cor:bdd mulitiplier}
Let $\alpha \in \mathbb{R}$ and let  $k\in \mathbb{N}$ be sufficiently large and $h\in \mathcal{L}^k(\R,\M)$. Then the map $f\mapsto hf$ is bounded on $F_1^{\alpha,c}(\mathbb{R}^d,\mathcal{M})$
\end{cor}
\begin{proof}
First, consider the case $\alpha>0$. We apply the atomic decomposition in Theorem \ref{thm: atomic decop T_L} with $K=k$ and $L=-1$. In this case, no moment cancellation of subatoms is required. We can easily check that, multiplying every (sub)atom in Definition \ref{def:smooth T atom} by $h$, we get another (sub)atom. Moreover,
\begin{equation}\label{eq: hf to f}
\|  hf \| _{F_1^{\alpha,c}}\leq \sum_{| \gamma |\leq k}\sup _{s\in \mathbb{R}^d}\|  D^\gamma h(s) \| _\mathcal{M}\cdot \|  f \| _{F_1^{\alpha,c}}.
\end{equation}

The case $\alpha\leq 0$ can be deduced by induction. Assume that \eqref{eq: hf to f} is true  for $\alpha > N \in \mathbb{Z}$.
Let  $\alpha>N-1$. Any $f\in F_1^{\alpha,c}$ can be represented as $f=J^{2} g =(1-(2\pi)^{-2}\Delta  ) g$ with $g\in F_1^{\alpha+2,c}$ and $\|f\|_{F_1^{\alpha,c}}\approx \|g\|_{F_1^{\alpha+2,c}}$. Since
$$
hf=(1-(2\pi)^{-2}\Delta  ) (hg)+ ((2\pi)^{-2}\Delta h )g +(2\pi)^{-2}\nabla h \cdot \nabla g, $$
we deduce
\begin{equation}\label{eq: pointwise multiplication}
\begin{split}
\| hf\|_{F_1^{\alpha,c}}&\lesssim \| (1-(2\pi)^{-2}\Delta  )(hg)\|_{F_1^{\alpha,c}}+ \| (\Delta h) g\|_{F_1^{\alpha,c}}+\sum_{i=1}^d\| \partial_i h \cdot \partial_i g\|_{F_1^{\alpha ,c}}\\
&\lesssim \| g\|_{F_1^{\alpha+2,c}}+ \| (\Delta h) g\|_{F_1^{\alpha+2,c}}+ \sum_{i=1}^d\| \partial_i h \cdot \partial_i g\|_{F_1^{\alpha+1 ,c}}.
\end{split}
\end{equation}
If $k\in \mathbb{N}$ is sufficiently large, we have
$$ \| (\Delta h) g\|_{F_1^{\alpha+2,c}}\lesssim \| g\|_{F_1^{\alpha+2,c}},\;\;\;\;\| \partial_i h \cdot \partial_i g\|_{F_1^{\alpha+1 ,c}}\lesssim \|   \partial_i g\|_{F_1^{\alpha+1 ,c}}.$$
Continuing the estimate in \eqref{eq: pointwise multiplication}, we obtain 
 $$\| hf\|_{F_1^{\alpha,c}}\lesssim \| g\|_{F_1^{\alpha+2,c}}+\sum_i \|   \partial_i g\|_{F_1^{\alpha+1 ,c}}\lesssim \| g\|_{F_1^{\alpha+2,c}} \lesssim \| f\|_{F_1^{\alpha,c}},$$
which completes the induction procedure.
\end{proof}

\noindent{\bf Acknowledgements.} The authors are greatly indebted to Professor Quanhua Xu for having suggested to them the subject of this paper, for many helpful discussions and very careful reading of this paper. The authors are partially supported by the the National Natural Science Foundation of China (grant no. 11301401).

\end{document}